\documentclass[preprint,12pt]{elsarticle}

\usepackage{amssymb}

\usepackage{amsmath}

\usepackage{amsthm}

\usepackage{graphicx} 
\usepackage{mathtools}
\usepackage{enumitem}
\usepackage{diagbox}
\usepackage{siunitx}

\usepackage{tabularx}
\usepackage{xcolor}

\usepackage{hyperref}
\usepackage{cleveref}
\usepackage{thmtools}

\usepackage{algorithm}
\usepackage[noend]{algpseudocode}
\algnewcommand\True{\textbf{true}\space}
\algnewcommand\False{\textbf{false}\space}

\theoremstyle{definition}
\newtheorem{theorem}{Theorem}[section]
\newtheorem{proposition}[theorem]{Proposition}

\newtheorem{corollary}[theorem]{Corollary}
\newtheorem{conjecture}[theorem]{Conjecture}
\newtheorem{lemma}[theorem]{Lemma}
\newtheorem{example}[theorem]{Example}
\newtheorem{remark}[theorem]{Remark}

\newtheorem{definition}[theorem]{Definition}

\newtheorem{notation}[theorem]{Notation}

\hypersetup{hypertexnames=false}

\newcommand{\qf}[1]{\langle #1\rangle}
\newcommand{\laurent}[1]{(\!(#1)\!)}
\newcommand{\Pfister}[1]{\langle\!\langle#1\rangle\!\rangle}

\newcommand{\second}{\unit{s}}
\newcommand{\minute}{\unit{min}}
\newcommand{\hour}{\unit{h}}
\newcommand{\R}{\ensuremath{\mathbb{R}}}
\newcommand{\N}{\ensuremath{\mathbb{N}}}
\newcommand{\F}{\ensuremath{\mathbb{F}}}
\newcommand{\C}{\ensuremath{\mathbb{C}}}
\newcommand{\Z}{\ensuremath{\mathbb{Z}}}
\newcommand{\Q}{\ensuremath{\mathbb{Q}}}
\newcommand{\LL}{\ensuremath{\mathbb{L}}}

\DeclareMathOperator{\spn}{span}
\DeclareMathOperator{\rad}{rad}
\DeclareMathOperator{\Mat}{Mat}

\DeclareMathOperator{\disc}{disc}
\DeclareMathOperator{\Ker}{ker}
\DeclareMathOperator{\Sup}{sup}
\DeclareMathOperator{\Inf}{inf}

\DeclareMathOperator{\Log}{log}
\DeclareMathOperator{\Min}{min}
\DeclareMathOperator{\Max}{max}
\DeclareMathOperator{\Qdim}{Qdim}

\DeclareMathOperator*{\Ast}{\raisebox{-0.6ex}{\scalebox{2}{$\ast$}}}
\newcommand{\tcap}[2]{%
	\underset{\text{\tiny#1}}{%
	#2}%
}

\begin{document}

\begin{frontmatter}

\title{Normal Quaternionic Matrices and Finitely Generated Witt Rings} 

\author[affiliation1]{Nico Lorenz}
\ead{nico.lorenz@ruhr-uni-bochum.de}
 
\author[affiliation2]{Alexander Schönert} 

\ead{alexander.schoenert@tu-dortmund.de}

\address[affiliation1]{
        {Fakultät für Mathematik, Ruhr Universität Bochum},
		{Universitätsstrasse 150,}
		{44780 Bochum,}
		{Germany}}

\address[affiliation2]{
        {Fakultät für Mathematik, Technische Universität Dortmund}, 
		{Vogelpothsweg 87,}
		{44227 Dortmund,}
		{Germany}}


\begin{abstract}
	We present a new approach to verify the Elementary Type Conjecture for abstract Witt rings $(W, G)$ with small $n = \log_2(|G|)$ that also satisfy the property $AP_3$, a certain case of the Arason-Pfister theorem.
	To do so, we make use of an abstract analogue of the $2$-torsion part of the Brauer group. 
	This enables us to describe the entire structure of an abstract Witt ring with $|G| = 2^n$ satisfying $AP_3$ in terms of a unique $(n \times n)$-matrix. 
	Via computational search, we find all these matrices for $n$ up to $7$. 
	The results confirm that all Witt rings of fields of characteristic different from $2$ with at most $2^7$ square classes are of elementary type.
\end{abstract}


\end{frontmatter}

\section{Introduction}\label{sec_Intro}

Throughout this article, if not stated otherwise explicitly, let $F$ denote a field of characteristic not 2.
It can be seen with ease that fields of very different nature may look identically from a quadratic form perspective.
For example, fields may be of different cardinality or characteristic but still may have isomorphic Witt rings. 
It thus seems to be interesting to build an axiomatic theory of quadratic forms that does not rely on an underlying field.
To this end there have been abstract approaches for more than 50 years now, starting with the work of Knebusch, Rosenberg, and Ware in the early 1970s (\cite{KnebuschRosenbergWare_WittRings}).
Over the years, several other approaches focusing on different aspects of the theory of quadratic form were introduced and studied by numerous authors, e.g. \emph{quadratic form schemes} as introduced by Cordes \cite{Cordes_FiniteNoQuaternions}, \emph{abstract Witt rings}, see \cite{Marshall_AbstractWittRings}, or \emph{special groups}, see \cite{DickmannMiragliaSpecialGroups}.

In this article, we mainly try to grasp the ring structure of Witt rings by mimicking the behavior of 2-fold Pfister forms and their associated quaternion algebras, leading to \emph{quaternionic structures} as in Marshall's book from 1980 \cite{Marshall_AbstractWittRings}, which will also serve as our main reference for basic notation and definitions.
The idea of looking at Witt rings from this perspective has its origin in the following result.

\begin{proposition}[{Cordes, \cite[2.2/2.3]{Cordes_WittGroupEquivalenve}}] \label{quadratic equivalence}
    Let $E, F$ be fields with square class groups $G_E := E^\ast / E^{\ast2}$ and $G_F := F^\ast / F^{\ast 2}$. 
    The Witt rings of $E$ and $F$ are isomorphic as rings if and only if there is a group isomorphism $\Phi: G_E \to G_F$ with $\Phi(-E^{\ast 2}) = -F^{\ast 2}$ preserving the isomorphism type of quaternion algebras, that is,
    \[\forall a,b,c,d \in G_E: \quad (a,b)_E \cong (c,d)_E \iff (\Phi(a), \Phi(b))_F \cong (\Phi(c), \Phi(d))_F.\]
\end{proposition} 

It is not very difficult to see and well known that the theories of abstract Witt rings, quaternionic structures and special groups are equivalent, meaning that models from one theory can be constructed from models of the other theories canonically, see \cite{DickmannMiragliaSpecialGroups, KSS_Equivalences, Marshall_AbstractWittRings}.


As a first step to understanding these theories, it is natural to try to understand the finitely generated objects.
All so far known examples of finitely generated Witt rings are of so called \emph{elementary type}, meaning that they can be derived from certain Witt rings only using two simple constructions. 
This gives rise to the \emph{Elementary Type Conjecture} (cf. \Cref{elementary type conjecture}) stating that in fact all finitely generated Witt rings are of such type.

In fact, the Elementary Type Conjecture is just the most basic version of a series of conjectures of the same spirit.
The ETC directly translates to a conjecture about the Galois groups $\mathsf{Gal}(F_2 / F)$, where $F$ is a field of characteristic not 2 with finite square class group $F^\ast / F^{\ast2}$ and $F_2$ its quadratic closure, which has obvious generalizations to other prime numbers than 2, see the discussion following \Cref{rem:comparing_conjectures} for more details.

It is known that the Elementary Type Conjecture holds for Witt rings and thus quaternionic structures over fields $F$ of small \emph{order} $n := \Log_2(|F^\ast / F^{\ast 2}|) \in \N_0$. 
The case $n \le 2$ can be taken care of trivially (cf. \Cref{small quaternionic structures}), while the case $n = 3$ was first done by Kula, Szczepanik and Szymiczek \cite{KulaSzczepanikSzymiczek_8SquareClasses} and the case $n \in \{4, 5\}$ by Carson and Marshall \cite{CarsonMarshall_Decomposition}. 
The case $n = 4$ was also dealt with by Szczepanik in 1985 independently (\cite{Szczepanik_RadicalNotExceeding16}). 

In this paper, we introduce a new perspective involving an abstract version of the $2$-torsion part of the Brauer group and developing \emph{normal quaternionic matrices} which uniquely display quaternionic structures that satisfy an additional property as in \cite{MarshallYucasAP3}. 
We use the abstract $2$-torsion part of the Brauer group to prove that the Elementary Type Conjecture also holds for all these quaternionic structures of orders $6$ and $7$.
In particular, this shows that all Witt rings of fields with up to $2^7$ square classes are of elementary type.
 
While the theoretical parts from \Cref{sec_QuatStruc} to \Cref{sec_NQM} were developed by the second author in terms of his PhD thesis, the computational approach described in \Cref{sec_Computation} was obtained by a joint effort of the authors.

We now give a brief outline of the structure of the article at hand.
In \Cref{sec_QuatStruc}, we introduce our notation and give the definition of \emph{quaternionic structures}. 
We use \Cref{sec_ETC} to collect the necessary ingredients to state the \emph{Elementary Type Conjecture} and related conjectures along with some counting arguments that are needed for the computations later on. 
In \Cref{sec_2Brauer}, we construct an abstract analogue to the $2$-torsion part of the Brauer group and discuss the Arason-Pfister property. 
In \Cref{sec_QM}, we define so called \emph{quaternionic matrices} as a means to efficiently compute and describe quaternionic structures of small order. 
We also introduce the already mentioned \emph{normal quaternionic matrices} which are then further investigated in \Cref{sec_NQM}. 
Finally, in the last \Cref{sec_Computation}, we present the algorithm and the computational means by which we prove our main result.

\section{Quaternionic Structures}\label{sec_QuatStruc}

The abstract point of view on quadratic form theory via certain axiomatic settings may not always be intuitive for readers with a classical background in the subject.
We recommend Lam's standard textbook in \cite{Lam_IntroQF} both for an extensive overview of the classical theory and for the brief introduction into the axiomatic approach in Chapter XII, Section 8 that may help with getting used to the lack of fields.
On top of that, we point out that the notation and terminology applied here aligns with that of \cite{Lam_IntroQF}.
Our main reference for the abstract foundations is Marshall's book \cite{Marshall_AbstractWittRings} which remains the most detailed account on the topic as of time of writing this.
However, as this book has become somewhat difficult to obtain, we try to include a number of other references where some of the results mentioned here can be found.
For a summary of the results of the 1980s and 1990s in abstract quadratic form theory regarding the elementary type conjecture, we recommend to consult Marshall's survey \cite{Marshall_ETC}.

Throughout this paper, we often consider groups $G$ of exponent dividing $2$. The most prominent example for us will be the square class group of a field.
We fix our notation for these:

\begin{definition}
    Let $G$ be a group.
    We call $G$ \emph{boolean} if every element of $G$ is its own inverse, meaning that the exponent of $G$ is $1$ or $2$.
    We then refer to the \emph{dimension of $G$} as its dimension as a vector space over $\F_2$.
    Clearly, this dimension is finite if and only if $G$ is finite and in this case we have $|G|=2^n$ for some $n\in \N_0$.
    In the same way, we define all terms regarding vector spaces like \emph{linear (in)dependence}, \emph{span}, or \emph{basis} of $G$.
\end{definition}

Quaternionic structures as defined below imitate the properties of isomorphism types of quaternion algebras over a given field of characteristic different from $2$.

\begin{definition}
    Let $G$ be a multiplicative boolean group with neutral element $1\in G$ and distinguished element $-1 \in G$ (the case $-1 = 1$ being possible). 
    Write $-a := ( -1 ) a$ for $a\in G$. 
    Let $Q$ be a set with distinguished element $0\in Q$ and $q : G \times G \to Q$ a surjective map.
    We call the elements of $Q$ \emph{quaternions}.
    We say that $S:=(G,-1,q)$ is a \emph{quaternionic structure} if the following hold for all $a,b,c,d\in G$:
    \begin{enumerate}[label=(Q\arabic*)]
        \item\label{QS1} $q(a, -a) = 0$
        \item\label{QS2} $q(a, b) = q(b, a)$
        \item\label{QS3} $q(a, b) = q(a, c) \iff q(a, bc) = 0$
        \item\label{QS4} $q(a, b) = q(c, d) \iff \exists x \in G:\ q(a, b) = q(a, x) = q(c, x) = q(c, d)$
    \end{enumerate}
    Furthermore, we define the \emph{radical} of $S$ as 
    \[\rad(S):=\{a\in G\mid q(a,x)=0\ \forall x\in G\}.\] 
    We say that $S$ is \emph{degenerate} if $\rad(S)\not=\{1\}$, and else \emph{nondegenerate}.
    If $G=\rad(S)$, we say that $S$ is \emph{totally degenerate}.
    We refer to $\dim(G)$ as the \emph{order} of $S$. 
    Note that we sometimes refer to $G$ as the \emph{square class group (of $S$)}.
\end{definition}

We include some well-known properties of quaternionic structures which we frequently make use of.
We include the proof for the reader's convenience.

\begin{lemma}\label{lem:basic_properties_qs}
    Let $S=(G,-1,q)$ be a quaternionic structure.
    The following hold for all $a,b,c,d\in G$:
    \begin{enumerate}[label=(\alph*)]
        \item\label{lem:basic_properties_qs_1} $q(a,b)=q(a,-ab)=q(b,-ab)$
        \item\label{lem:basic_properties_qs_2}
        $q(a,1)=0$ and $q(a,a)=q(a,-1)$
        \item\label{lem:basic_properties_qs_3} $q(a,b)=q(c,d) \iff \ \exists x\in G:\ 0=q(a,bx)=q(ac,x)=q(c,dx)$
    \end{enumerate}
\end{lemma}
\begin{proof}
    The properties in \ref{lem:basic_properties_qs_1} follow from \ref{QS1} and \ref{QS3} as $b\cdot (-ab)=-a$ and $a\cdot (-ab)=-b$.
    Now \ref{lem:basic_properties_qs_2} is just the case $b\in \{1,-1\}$.
    Lastly, \ref{lem:basic_properties_qs_3} is a direct application of \ref{QS3} on the equivalence in \ref{QS4}.
\end{proof}

In light of \Cref{lem:basic_properties_qs} \ref{lem:basic_properties_qs_3} above, it is a simple exercise to show that the condition given in the following definition is equivalent to that in \Cref{quadratic equivalence} and thus aligns with the classical case:

\begin{definition}
    Let $S = (G, -1, q), S'=(G',-1',q')$ be quaternionic structures.
    We say that $S$ and $S'$ are \emph{isomorphic}, in symbols $S\cong S'$, if there exists a group isomorphism $\phi: G \to G'$ satisfying $\phi(-1) = -1'$ and for all $a,b\in G$, we have
    \[q(a,b) = 0 \iff q'( \phi(a), \phi(b) ) = 0.\]
\end{definition}

\begin{remark}\label{rem:field_defines_qs}
    Any field $F$ of characteristic not $2$ defines a quaternionic structure, using $G_F = F^\ast / F^{\ast 2}, -1 = -F^{\ast2}$ and $q_F(a,b) := (a,b)_F$, the quaternion algebra with slots $a$ and $b$ over $F$ (for \ref{QS4}, see \cite[Common Slot Theorem III.4.13]{Lam_IntroQF}). 
    Here, the zero-element in $Q$ can be viewed as the split quaternion algebra or the zero in the Witt ring. 
    We write $S_F:=(G_F, -F^{\ast 2}, q_F)$.
\end{remark}

We include some basic examples (cf. \cite[Example 1.2]{Marshall_ETC}), also to introduce notation:

\begin{example}\label{small quaternionic structures}
\begin{itemize}
\item There is only one quaternionic structure of order zero, denoted as $\LL_0$, and defined by the unique map $\{1\}\times\{1\}\to \{0\}$. 
It is realized as $S_F$ for any quadratically closed field $F$ of characteristic different from $2$, for example $\LL_0\cong S_\C$. 
This structure is both totally degenerate and nondegenerate by definition.
\item There are exactly two isomorphism types of totally degenerate quaternionic structures of order $n\in \N$, depending on the choice whether or not $-1=1$ in the square class group $G$ and defined by the unique map $G\times G\to \{0\}$. 
For $n = 1$, we denote this structure by $\LL_{1,0}$ if $-1 = 1$ and $\LL_{1,1}$ if $-1 \neq 1$.
\item There are three quaternionic structures $S=(\{1,a\},-1,q)$ of order $1$: 
Firstly, if $-1=1$, then 
$q(1,1) = q(1,a) = q(a,1) = q(a,a) = 0$, 
so $S\cong \LL_{1,0}$. 
Secondly, if $-1 = a$ and $S\not\cong \LL_{1,1}$, then it can be easily verified that $S\cong S_\R=:\LL_{1}$. 
Note that $\LL_{1,0}$ and $\LL_{1,1}$ are realized by finite fields of cardinality $\equiv_4 1$ resp. $\equiv_4 3$.
\end{itemize}
\end{example}

Defining the Witt ring of a quaternionic structure works very much in the same way as in the classical theory. 
We briefly collect the necessary definitions for abstract quadratic forms which correspond to \emph{nondegenerate} and \emph{diagonalized} quadratic forms in the classical setting.

\begin{definition}
    Let $S=(G,-1,q)$ be a quaternionic structure.
    \begin{itemize}
        \item 
        A \emph{(quadratic) form over $S$} of \emph{dimension} $n\in \N$ is an $n$-tuple, denoted $\qf{a_1,\ldots,a_n}\in G^n$.
        \item 
        For quadratic forms $\varphi=\qf{a_1,\ldots,a_n}$ and $\psi:=\qf{ b_1,\ldots,b_m}$, the \emph{sum} and (tensor) \emph{product} are defined as
        $\varphi\perp\psi:=\qf{a_1,\ldots,a_n,b_1,\ldots,b_m}$
        and
        $\varphi\otimes \psi:=\perp_{i=1}^m b_i\varphi:=\perp_{i=1}^m \qf{b_ia_1,\ldots,b_ia_n}$.
        We also write $k\times \varphi$ for the $k$-fold sum $\varphi\perp\ldots\perp\varphi$.
        \item 
        \emph{Isometry} of quadratic forms, denoted by $\cong$, is defined inductively as in \cite[§I, Exercise 19]{Lam_IntroQF} or \cite[1.15]{Marshall_AbstractWittRings}:
        
        For any $a,b,c,d\in G$, we have $\qf{a}\cong \qf{b}$ if and only if $a=b$, and $\qf{a,b}\cong \qf{c,d}$ if and only if $ab=cd \text{ and } q(a,b)=q(c,d)$.
        Two forms $\varphi=\qf{a_1,\ldots,a_n},\ \psi=\qf{b_1,\ldots,b_n}$ of the same dimension $n>2$ are isometric if and only if there are $a,b\in G$ and a form $\chi$ of dimension $n-2$ such that 
        \[\qf{a_1,a}\cong \qf{b_1,b},\ \qf{a_2,\ldots,a_n}\cong \qf{a}\perp \chi,\ \qf{b_2,\ldots,b_n}\cong \qf{b}\perp\chi.\]
        This defines an equivalence relation as was shown by Marshall in \cite[2.6]{Marshall_AbstractWittRings} (also in \cite[XII.8.21]{Lam_IntroQF}).
        \item
        The \emph{hyperbolic plane} is defined as $\mathbb{H}:=\qf{1,-1}$. 
        A quadratic form is called \emph{hyperbolic} if it is isometric to the $k$-fold sum of hyperbolic planes $k\times \mathbb{H}$ for some $k\in \N_0$. 
        \item 
        A quadratic form $\varphi$ is \emph{isotropic} if it is isometric to a form $\mathcal{H}\perp \psi$, else \emph{anisotropic}.
        \item 
        Given $n\in \N_0$ and $a_1,\ldots,a_n\in G$, we call $\Pfister{a_1,\ldots,a_n}:=\bigotimes_{i=1}^n \qf{1,-a_i}$ an \emph{$n$-fold Pfister form}.
        \item 
        Write $D_S(\varphi):=\{a\in G\mid \varphi\cong \qf{a,\ldots}\}$ for the \emph{set of represented elements} of a form $\varphi$ over $S$.
        \item 
        The \emph{value group of $a\in G$} is defined as
        $V(a):=D_S(\Pfister{a})$ 
        and can also be written as 
        $\{b\in G\mid q(a,b)=0\}$. 
        Note that in the literature, similar to the ambiguity of the sign convention for Pfister forms, this is sometimes defined as $V(-a)$.
    \end{itemize}
\end{definition}

We do not explicitly need the notion of abstract Witt rings in order to execute our computations in \Cref{sec_Computation}.
However, we give some context in order to embed the theory presented here into its proper context.
This will be of importance for introducing the abstract version of the $2$-torsion part of the Brauer group in \Cref{sec_2Brauer}.
For a more detailed look on axioms for abstract Witt rings, we recommend \cite{Kula_FinitelyGeneratedWittRings}.

\begin{remark}\label{rem:Witt_rings_of_qs}
    Having established that isometry in the above sense indeed defines an equivalence relation, the following results of \cite{Marshall_AbstractWittRings} can mostly be proven in the same way as in the classical theory. 
    Let $S=(G,-1,q)$ be a quaternionic structure and $\varphi,\psi,\chi$ forms over $S$.
    The following hold:
    \begin{itemize}
        \item \emph{Witt cancellation}: We have $\varphi\cong \psi$ if and only if $\varphi\perp\chi\cong \psi\perp \chi$
        \item There is a unique anisotropic form $\varphi_\text{an}$ (the \emph{anisotropic part of $\varphi$}) and a unique $k\in \N_0$ (the \emph{Witt index of $\varphi$}) such that $\varphi\cong \varphi_\text{an}\perp k\times \mathcal{H}$.
        This defines an equivalence relation (\emph{Witt equivalence}) via \[\varphi\sim \psi\iff \varphi_\text{an}\cong \psi_\text{an}.\]
    \end{itemize}
    Now, as in \cite{Lam_IntroQF}, the \emph{Witt ring $W(S)$ of $S$} can be defined as the set of Witt equivalence classes of forms over $S$ along with the binary operations derived from the addition and multiplication of forms.
    
    By abuse of notation, we usually identify quadratic forms over $S$ with their Witt equivalence class.
\end{remark}

As in the classical theory, Pfister forms prove to be crucial in understanding the structure of Witt rings.
We collect the most important facts here.
Once again, proofs can be found in \cite{Marshall_AbstractWittRings} while many properties are also considered in detail in \cite{Szczepanik_QFSnontrivialrad}.

\begin{remark}\label{rem:Pfister_forms_and_fundamental_ideal}
    Let $S=(G,-1,q)$ be a quaternionic structure with Witt ring $W(S)$.
    \begin{enumerate}[label=(\alph*)]
    \item\label{rem:Pfister_forms_and_fundamental_ideal1}
    The following hold:
    \begin{itemize}
        \item The set of represented elements of a Pfister form is a subgroup of $G$.
        In particular, the value group $V_S(a)$ is indeed a group for all $a\in G$.
        \item A Pfister form is either anisotropic or hyperbolic.
        \item For all $a,b,c,d\in G$, we have
        \begin{align*}
            q(a,b)=q(c,d)&\iff\Pfister{a,b}\cong\Pfister{c,d}\\
            &\iff (1-a)(1-b)=(1-c)(1-d)\in W(S)
        \end{align*}
    \end{itemize}
    
    \item 
    It can also be verified as in the classical theory that the set of Witt equivalence classes of even dimensional forms in $W(S)$ defines the unique prime ideal of characteristic $2$ in $W(S)$, called the \emph{fundamental ideal} $I(S)$.
    For any $k\in \N_0$, the $k$-th power $I^k(S)$ of the fundamental ideal is additively generated by the Witt equivalence classes of $k$-fold Pfister forms.
    
    \item\label{rem:Pfister_forms_and_fundamental_ideal2}
    We refer to the following as the \emph{Arason-Pfister property} for $k\in \N_0$:
    \begin{align}
    \text{If  $\varphi\in I(S)^k\setminus\{0\}$ is anisotropic, then } \dim(\varphi)\ge 2^k \tag{$AP_k$}\label{eq_APk}
    \end{align}
    In the case when $S$ is a quaternionic structure of a field, this property holds by a classical result of Arason and Pfister \cite[Hauptsatz]{ArasonPfister_Krull}.

    For arbitrary quaternionic structures $S$, $AP_0$ holds trivially and $AP_1$ holds since all forms in $I(S)$ are of even dimension.
    As in the classical theory, it follows that the forms in $I^2(S)$ are exactly the forms $\varphi$ of even dimension with trivial \emph{discriminant} $\disc{\varphi}:=(-1)^{\frac{n(n-1)}{2}} \det(\varphi)$ where $n:=\dim(\varphi)$ (\cite[3.9]{Marshall_AbstractWittRings} and in a slightly different context \cite{MarshallYucasAP3}).
    This means that the forms of dimension $2$ in $I^2(S)$ can be written as $\qf{a,b}$ with $1=\disc{\qf{a,b}}=-ab$, so $b=-a$ and $\qf{a,b}\cong \mathcal{H}=0\in W(S)$.
    Hence, also $AP_2$ holds in $W(S)$.
    
    In \Cref{sec_2Brauer}, we replicate results of Marshall and Yucas \cite{MarshallYucasAP3} to establish an equivalent condition to $AP_3$.
    
    It is an open question whether or not $AP_k$ generally holds for $W(S)$ for any $k\ge 3$ for quaternionic structures $S$ as defined here.
    \end{enumerate}
\end{remark}

\section{The Elementary Type Conjecture}\label{sec_ETC}

The Elementary Type Conjecture (ETC) as first stated by Marshall \cite[Chapter 5, Problem 4]{Marshall_AbstractWittRings} makes a prediction what the complete list of quaternionic structure of finite order looks like.
The following invariants prove to be important tools for characterizing these structures.
It is possible to compute these invariants for all components used in the ETC which has some interesting implications for open problems in the classical theory, see e.g. \Cref{level conjecture}.

\begin{definition}
Let $S=(G,-1,q)$ be a quaternionic structure. 
We define the \emph{level} of $S$ as 
\[s(S) := \Inf\{ k \in \N \mid -1 \in D_S(k\times \qf{1})\}\]
and the \emph{Pythagoras number} of $S$ as
\[p(S):=\Inf\{k\in \N\mid D_S(k\times \qf{1})=D_S((k+1)\times \qf{1})\}\] 
where $\Inf(\emptyset)=\infty$ and we say that $S$ is \emph{pythagorean} if $p(S)=1$. 
Furthermore, we define the \emph{$u$-invariant} of $S$ as
\[u(S):=\Sup\{\dim(\varphi)\mid \varphi \text{ is an anisotropic torsion form over } S\}\]
where \emph{torsion form} means that there exists $n\in \N$ such that $n\times\varphi$ is hyperbolic. 
In the classical sense, this is often referred to as the \emph{general $u$-invariant} (cf. \cite[Chapter XI]{Lam_IntroQF}).
\end{definition}

If $F$ is a field, note that the values $s(S_F)$ and $p(S_F)$ are the level and Pythagoras number of $F$.
In particular, we point out that the level of $S$ is generally not the ring theoretic level of the associated Witt ring. 

In the ETC, it is conjectured that every finite quaternionic structure can be constructed from simple building blocks, leading to so called elementary type structures.
The building blocks are the structures of order $1$ as seen in \Cref{small quaternionic structures} and those with exactly $2$ quaternions that are nondegenerate:

\begin{definition}
    A finite quaternionic structure is said to be a \emph{local type structure} or simply \emph{of local type} if it is nondegenerate of order at least $3$ with exactly $2$ quaternions.
\end{definition}

\begin{remark}\label{rem:local_types}
    Local type structures are explored in detail in \cite[Chapter 5, Section 3]{Marshall_AbstractWittRings}.
    The results mentioned in the following can be verified by elementary linear algebra means as only basic properties of symmetric bilinear maps are required.
    We give some additional detail in \Cref{prop:nqm_local_types} later on when we discuss how to represent local type structures by matrices.
    
    If $S=(G,-1,q)$ is a local type structure of order $n\in \N$, it is immediate that $q$ can be viewed as a nondegenerate symmetric bilinear map $q:G\times G\to \F_2$ by \ref{QS2}, \ref{QS3} and the aligning definitions of the term 'nondegenerate'.
    Conversely, given any nondegenerate symmetric bilinear map $q:G\times G\to \F_2$ satisfying the additional condition \ref{QS1} imposed on the distinguished element $-1\in G$ defines a quaternionic structure as \ref{QS4} is always satisfied.
    Hence, in order to determine all local type structures of order $n$, it suffices to determine all such nondegenerate symmetric bilinear maps.
    It turns out that there are three kinds of such maps and thus three kinds of local type structures (also cf. \cite[section 3]{Marshall_ETC}).
    We denote them as $\LL_n$ (unique local type structure of odd order $n$, satisfies $-1\neq 1$), $\LL_{n,0}$ (unique local type structure of even order $n$ with $-1=1$) and $\LL_{n,1}$ (unique local type structure of even order $n$ with $-1\neq 1$).
    By \cite[VI.2.22 to 2.29]{Lam_IntroQF}, these can be obtained as the quaternionic structures of suitable extensions of the $2$-adic numbers $\Q_2$ because
    \begin{itemize}
        \item any finite extension of $\Q_2$ of odd degree $m$ gives rise to a local type structure of order $m+2$ and thus to $\LL_{m+2}$
        \item any finite extension of $\Q_2$ of even degree $m$ gives rise to a local type structure of order $m+2$ with $-1=1$ if and only if this extension contains $\Q_2(\sqrt{-1})$, and thus to either $\LL_{m+2,0}$ or $\LL_{m+2,1}$.
    \end{itemize}
    In particular, the local type structures are exactly the quaternionic structures of dyadic local fields.
    The level, pythagoras number and $u$-invariant of local type structures can now be calculated directly or taken as well-known facts of the classical theory (cf. \cite{Lam_IntroQF}):
    All local type structures have $u$-invariant $4$ (\cite[XI.6.2]{Lam_IntroQF}) and for $k\in \N$, $k\ge 2$, we have
    \begin{itemize}
        \item $s(\LL_{2k-1})=4$,\ $p(\LL_{2k-1})=4$
        \item $s(\LL_{2k,0})=1$,\ $p(\LL_{2k,0})=2$
        \item $s(\LL_{2k,1})=2$,\ $p(\LL_{2k,1})=3$
    \end{itemize}
\end{remark}

As we have established how the 'building blocks for quaternionic structures' look like, we now consider two ways to construct bigger quaternionic structures.
The first one corresponds to the Laurent series extension $F\laurent{x}$ of a field $F$ of characteristic different from $2$ (cf. \cite[VIII.4.11]{Lam_IntroQF}.
This gives us a way to extend a structure of order $n$ to one of order $n+1$:

\begin{proposition}\label{prop:group_extension}
    Let $S=(G,-1,q)$ be a quaternionic structure and $G_x$ the direct product of $G$ with a cyclic group $\{1,x\}$ of order $2$, written as an inner direct product. 
    The \emph{group extension} $S[x]:=(G_x,-1,q_x)$ with quaternions in $Q\times G$,
    \[q_x(ax^\alpha,bx^\beta)=(q(a,b),(-1)^{\alpha\beta}a^\beta b^\alpha)\]
    and zero quaternion $(0,1)$ is a quaternionic structure. 
    Its Witt ring is the group ring $W(S)[\{1,x\}]$.
    The value groups are $V_S(1) = G_x$, $V_S(a) = V(a)$ for all $a \in G \setminus \{1\}$ and $V_S(ax) = \{1, -ax\}$ for all $a \in G$.
    We often write $S\Delta := S[x]$.
    Constructing the group extension can be iterated and we write $S\Delta_n$ for the $n$-fold group extension of $S$ for any $n\in \N$. 
    
    If $S \cong S_F$ for a field $F$ of characteristic not $2$, then $S\Delta \cong S_{F\laurent{x}}$.
\end{proposition}

The explicit description of $q$ above was pointed out to the second author by Sven Wagner (unpublished) while the corresponding abstract Witt ring and value groups are presented in \cite[5.16]{Marshall_AbstractWittRings} and \cite[3.2]{Marshall_ETC}.
We omit the proof as it can be done via a straightforward computation using the fact that quaternionic structures can be described uniquely in terms of either their value groups or their abstract Witt rings.
These equivalent descriptions are dealt with in \cite{KSS_Equivalences}.

\begin{proposition}\cite[5.7]{Marshall_AbstractWittRings}\label{prop:direct_product}
    Let $S_1=(G_1,e_1,q_1)$ and $S_2=(G_2,e_2,q_2)$ be quaternionic structures with sets of quaternions $Q_1,Q_2$.
    The \emph{direct product} $S_1\times S_2:=(G_1\times G_2, (e_1,e_2), q_1\times q_2)$ with quaternions $Q_1\times Q_2$
    \[(q_1\times q_2)((a_1,a_2),(b_1,b_2)):=(q_1(a_1,b_1), q_2(a_2,b_2)) \]
    and zero-quaternion $(0,0)$ is a quaternionic structure with abstract Witt ring 
    \[W(S_1\times S_2)=(W(S_1)\times W(S_2))_\text{par}:=\{(\varphi,\psi)\mid \dim(\varphi)\equiv_2\dim(\psi)\}.\]
    We usually write $G_1\times G_2$ as an inner direct product.
    Its value groups are $V_{S_1\times S_2}(a_1a_2)=V_{S_1}(a_1)\times V_{S_2}(a_2)$ for $a_1\in G_1$, $a_2\in G_2$.
\end{proposition}

The following result of Kula shows that direct products 'behave well' considering fields. 

\begin{theorem}\label{thm:Kula_direct_product}\cite[6.5]{Kula_FieldsQFSchemes}
    If $E,F$ are fields of characteristic not $2$, then $S_E\times S_F$ is realized by a field of characteristic $0$.
\end{theorem}

\begin{example}\label{ex:ets_totdeg_and_small_group_ext}
\begin{itemize}
\item The totally degenerate quaternionic structures of order $n\in \N$ (cf. \Cref{small quaternionic structures}) are exactly $\LL_{1,0}^n$ (the $n$-fold direct product) and $\LL_{1,1}^n$.
This follows immediately from the fact that for $\epsilon\in \{0,1\}$, $\LL_{1,\epsilon}^n$ has $1^n=1$ quaternions by \Cref{prop:direct_product}.
\item We have $\LL_0\Delta\cong \LL_{1,0}$ which also follows from \Cref{small quaternionic structures} as there is only one quaternionic structure of order $2$ and $-1 = 1$.
\item We have $\LL_{1}\times \LL_{1}\cong \LL_{1}\Delta$. 
This is in fact the only \emph{decomposable} (meaning that it can be written as the direct product of two structures of order at least 1) group extension (cf. \cite[p. 8]{Marshall_ETC}).
\end{itemize}
\end{example}

Marshall introduced the following \emph{normalized decomposition}, uniquely decomposing quaternionic structures into indecomposable factors $S$, meaning they can only written as the trivial direct product $S\cong \LL_0\times S$ (\cite[5.12]{Marshall_AbstractWittRings} and \cite[4.1]{Marshall_ETC}).

\begin{proposition}\label{normalized decomposition}
Let $S$ be a finite quaternionic structure. 
Choose $\epsilon=0$ if $-1=1$ in $S$ and $\epsilon=1$ else. 
There exist indecomposable nondegenerate quaternionic structures $S_1,\ldots,S_k$ such that 
\[S\cong \LL_{1,\epsilon}^m\times S_1\times\ldots\times S_k\]
where $m=\dim(\rad(S))$. The $S_i$ are unique up to permutation.
\end{proposition}

\begin{definition}
Let $S$ be a finite quaternionic structure. 
We call $S$ of \emph{elementary type} if $S$ can be obtained from $\LL_0, \LL_{1,0}, \LL_{1,1}, \LL_1$ and local type structures applying only direct products and group extensions.
\end{definition}

The lack of counterexamples led to the following conjecture:

\begin{conjecture}[Elementary Type Conjecture (ETC)] \label{elementary type conjecture}
    Every finite quaternionic structure is of elementary type.
\end{conjecture}

It follows from \Cref{rem:local_types}, \Cref{prop:direct_product}, \Cref{thm:Kula_direct_product} along with the earlier \Cref{small quaternionic structures} that all elementary type structures are realized by fields.
In terms of the corresponding Witt rings, the Elementary Type Conjecture states that finitely generated Witt rings can be constructed from $W(\C) \cong \Z/2\Z, W(\F_3) \cong \Z/4\Z, W(\R) \cong \Z$ and Witt rings $W(F)$ of finite extensions $F$ of $\Q_2$ using group ring extensions with a group of order 2 and fiber products of the canonical homomorphisms $W(F) \to W(F) / I(F) \cong \Z/2\Z$ (which we explicitly described in \Cref{prop:direct_product}).

Since it is not known whether all quaternionic structures are in fact realized by fields, we have the following possibly weaker versions:

\begin{conjecture}[$AP_3$ Elementary Type Conjecture (3ETC)]\label{conj_AP3ETC}
    Every finite quaternionic structure satisfying $AP_3$ is of elementary type.
\end{conjecture}

\begin{conjecture}[weak Elementary Type Conjecture (wETC)] \label{conj_weakETC}
    Every finite quaternionic structure that is realized by a field is of elementary type.
\end{conjecture}

\begin{remark}\label{rem:comparing_conjectures}
    Clearly, we have ETC $\implies$ 3ETC.
    As pointed out above, all elementary types are realized by fields. 
    It follows from the Arason-Pfister theorem \cite[Hauptsatz]{ArasonPfister_Krull} that all quaternionic structures which are realized by fields - in particular all elementary types - satisfy $AP_3$.
    Hence, we also have 3ETC $\implies$ wETC.
    For the main result of this article, we work with elementary type structures satisfying $AP_3$.
    In order to verify 3ETC for a given order $n\in \N$, it suffices to show that the number of elementary type structures of order $n$ is an upper bound for the number of quaternionic structures of order $n$ satisfying $AP_3$.
    We deal with the question of how to count elementary type structures at the end of this section. 
    The results are summarized in \Cref{number of small ets}.
\end{remark}

As already pointed out in the introduction, the ETC gave rise to more general conjectures.
By \cite[Theorem 3.5]{JacobWareETCGalois}, the quaternionic structure associated to a field $F$ is of elementary type if and only if the Galois group $\mathsf{Gal}(F_2 / F)$ of its quadratic closure $F_2$ over $F$ can be constructed in finitely many steps from the Galois groups $\mathsf{Gal}(K_2/ K)$, where $K$ is either $\C, \R$, a finite field of characteristic not 2 or a finite extension of $\Q_2$, using certain semi-direct products with the 2-adic integers $\Z_2$ and free pro-2 products.
The wETC thus translates directly to a conjecture on the structure of the Galois groups $\mathsf{Gal}(F_2 / F)$ for fields with finite square class group $F^\ast / F^{\ast2}$.
Even more general, this may be further conjectured not only for the prime number 2, but also for the group $\mathsf{Gal}(F_p / F)$ for fields $F$ with $F^\ast / F^{\ast p}$ finite containing a $p$-th root of unity.
More details on these general elementary type conjectures can be found in Ido Efrats recent survey \cite{efrat2025elementarytypeconjecturemaximal}.

The ETC is also connected to a famous conjecture regarding the level of fields:

\begin{conjecture}[Level Conjecture (LC)] \label{level conjecture}
    All finite quaternionic structures (coming from fields) have level in $\{1,2,4,\infty\}$.
\end{conjecture}

\begin{example}\label{list of small ets with invariants}
With Marshall's normalized decomposition, it is possible to iteratively compute complete lists of the elementary types of a certain order. 
We give those for order up to 2, including the level $s$, Pythagoras number $p$, and $u$-invariant $u$ in \Cref{tbl_local_structures}.
As they will be important later on, we also include the $d_1:=\dim(V(-1))$ and the minimum $d_{\Min}$ over all $\dim(V(a))$ for $a\in G$. 
We also include the general values for direct products, group extensions, and local structures so that all the given invariants can be computed for all elementary types with this list. 
Verifying these values can be done in a straightforward manner applying the results mentioned in \Cref{rem:local_types}, \Cref{prop:group_extension}, \Cref{prop:direct_product}.
The invariants for direct products and group extensions were also computed explicitly in \cite[5.1]{Kula_FieldsQFSchemes}.
Note that we assume $S$ and $S'$ to be of elementary type in \Cref{tbl_local_structures} which means that the depicted value for the $u$-invariant of $S\times S'$ is correct as the other case of \cite[5.1(d)]{Kula_FieldsQFSchemes} does not occur for elementary types.

Similar to \Cref{small quaternionic structures}, it can be seen that the list contains all quaternionic structures of order up to $2$.
In the following \Cref{tbl_local_structures}, $k$ can be any integer greater than 1.
\begin{center}
\begin{table}[ht!]
\footnotesize
\begin{tabular}{l|p{0.09\textwidth}|p{0.11\textwidth}|p{0.13\textwidth}|p{0.11\textwidth}|p{0.11\textwidth}|p{0.11\textwidth}}
structure & order & $s$ & $p$ & $u$ & $ d_1$ & $d_{\Min}$\\ \hline
$\LL_0$ & 0& 1& 1& 1& 0&0\\
$\LL_{1,0}$ & 1& 1&2&2&1&1\\
$\LL_{1,1}$ & 1& 2&2&2&1&1\\
$\LL_1$& 1&$\infty$&1&0&0&0\\
$\LL_{1,0}^2$& 2&1&2&2&2&2\\
$\LL_{1,1}^2$& 2&2&2&2&2&2\\
$\LL_{1,1}\times \LL_1$&2&$\infty$&2&2&1&2\\
$\LL_1^2\cong \LL_1\Delta$&2&$\infty$&1&0&0&0\\
$\LL_{1,0}\Delta$&2&1&2&4&2&1\\
$\LL_{1,1}\Delta$&2&2&3&4&1&1\\
$\LL_{2k-1}$&$2k-1$&4&4&4&$2k-2$&$2k-2$\\
$\LL_{2k,0}$&$2k$&1&2&4&$2k$&$2k-1$\\
$\LL_{2k,1}$&$2k$&2&3&4&$2k-1$&$2k-1$\\
$S$&$n$&$s$&$p$&$u$&$d_1$&$d_{\Min}$\\
$S'$&$n'$&$s'$&$p'$&$u'$&$d_1'$&$d_{\Min}'$\\
$S\Delta$&$n+1$&$s$&$\begin{cases} p \text{ if } s=\infty\\ s+1 \text{ else}\end{cases}$&$2u$&$d_1$&1\\
$S\times S'$&$n+n'$&$\Max\{s,s'\}$&$\Max\{p,p'\}$&$\Max\{u,u'\}$&$d_1+d_1'$&$d_{\Min}\!+\!d_{\Min}'$
\end{tabular}
\caption{Elementary type structures and their invariants} \label{tbl_local_structures}
    \end{table}
\end{center}
\end{example}

The previous \Cref{list of small ets with invariants} (see \Cref{tbl_local_structures}) shows that every elementary type structure has level in $\{1,2,4,\infty\}$, Pythagoras number in $\{1,2,3,4,5\}$ and $u$-invariant in $\{0\} \cup \{2^m \mid m \in \N_0\}$ and all these values can be obtained (though not independently, since for example the Pythagoras number cannot exceed the level by more than $1$). 
In particular, the LC holds if the ETC or 3ETC or wETC does.
The LC is known to hold for Witt rings of fields (of characteristic not $2$) with up to $2^8$ square classes by a result of Becher in 2001 \cite{Becher_NoSquareClasses}.

We explore some existing results on the ETC and count elementary types of a given order in the remainder of this section.
Recall that for a quaternionic structure $S=(G,-1,q)$, we defined the value group of an element $a\in G$ as \[V_S(a)=D_S(\Pfister{a})=\{b\in G\mid q(a,b)=0\}\supset\{1,-a\}.\]

\begin{definition}
Let $S=(G,-1,q)$ be a quaternionic structure and $k\in \N_0$. 
We say that $a\in G$ is \emph{$k$-rigid} if its value group is $k$-dimensional, i.e. $\dim(V_S(a))=k$. 
If $G$ contains a $k$-rigid element, we say that $S$ is \emph{$k$-rigid}. 
If $S$ is $k$-rigid and furthermore either $k = 0$ or $G$ does not contain a $(k-1)$-rigid element, we say that $S$ is \emph{strictly $k$-rigid}.
We usually only say \emph{rigid} and \emph{strictly rigid} instead of $1$-rigid and strictly $1$-rigid.
\end{definition}

The rigidity of elementary type structures can be computed with the following observations.

\begin{lemma}\label{lem:rigidity_ets}
    Let $S=(G,-1,q)$ be a quaternionic structure of order $n$.
    \begin{enumerate}[label=(\alph*)]
        \item\label{lem:rigidity_ets1} 
        $S$ is $0$-rigid if and only if $S$ is strictly $0$-rigid if and only if $S$ is pythagorean if and only if $V(-1)=\{1\}$.
        \item\label{lem:rigidity_ets2} 
        $S\Delta$ is $1$-rigid. It is $0$-rigid if and only if $S$ is pythagorean and $n>0$.
        \item\label{lem:rigidity_ets3} 
        If $S$ is of local type, then $S$ is strictly $(n-1)$-rigid.
        More precisely, all elements of $G\setminus\{1\}$ are $(n-1)$-rigid.
        \item\label{lem:rigidity_ets4} 
        If $S$ is strictly $k$-rigid and $S'=(G',-1',q')$ is a strictly $k'$-rigid quaternionic structure, then $S\times S'$ is strictly $(k+k')$-rigid.
    \end{enumerate}
\end{lemma}
\begin{proof}
    First note that for all $a\in G$, we have $\{1,-a\}\subset V(a)$ by \ref{QS1} and \Cref{lem:basic_properties_qs} \ref{lem:basic_properties_qs_2}.
    Hence, \ref{lem:rigidity_ets1} follows from the fact that $\dim(V(a))\ge 1$ for all $a\in G\setminus\{-1\}$ and the definitions.
    
    \ref{lem:rigidity_ets2}: 
    Write $S\Delta=S[x]=(G\cup xG, -1,q_x)$ and recall that by definition $q_x(ax^\alpha,x)=0$ for $a\in G$, $\alpha\in \{0,1\}$ if and only if $(-1)^\alpha a=1$ if and only if $ax^\alpha\in \{1,-x\}$.
    Since $-x\neq 1$, this shows that $x$ is $1$-rigid.
    By \ref{lem:rigidity_ets1}, $S\Delta$ is $0$-rigid if and only if $V_{S[x]}(-1)=\{1\}$.
    For $a,\alpha$ as before, we have $q_x(-1,ax^\alpha)=0$ if and only if $q(-1,a)=0$ and $(-1)^\alpha=1$.
    If $-1\neq 1$, it follows that $V_{S[x]}(-1)=V_S(-1)$, so $S[x]$ is pythagorean if and only if $S$ is.
    If $-1=1$, we trivially have $V_S(-1)=G$ and $V_{S[x]}(-1)=G\cup xG$, so $S$ is only pythagorean if $n=0$ and $S[x]$ is not pythagorean.
    
    \ref{lem:rigidity_ets3}:
    Recall that local type structures are nondegenerate and have exactly $2$ quaternions by definition. 
    We write $Q=\{0,1\}$ for the quaternions of $S$.
    For $a\in G\setminus\{1\}$, choose $b\in G$ such that $q(a,b)=1$, making use of $S$ being nondegenerate.
    For all $c\in G$, we have $q(a,c)=0\iff q(a,bc)=1$ by \ref{QS3}, so it follows that we get a bijection from $\{c\in G\mid q(a,c)=1\}$ to $\{c\in G\mid q(a,c)=0\}=V_S(a)$.
    Also, we get $G=V_S(a)\cup bV_S(a)$, so $|V_S(a)|=2^{n-1}$ and the assertion follows.
    
    \ref{lem:rigidity_ets4}:
    This follows from \Cref{prop:direct_product} as $V_{S\times S'}(aa')=V_S(a)\times V_{S'}(a')$ for all $a\in G$, $a'\in G'$.
\end{proof}

\medskip

Among the following results on the ETC, especially those regarding $k$-rigid elements for small $k$ will be important for reducing the length of the computations in \Cref{sec_Computation}.

\begin{proposition}\label{results on the etc}
Let $S$ be a finite quaternionic structure of order $n\in \N$.
\begin{enumerate}[label=(\alph*)]
\item\label{ETC_pythagorean} Marshall in \cite[6.23]{Marshall_AbstractWittRings}: If $S$ is pythagorean, then every indecomposable factor of $S$ is isomorphic to $\LL_1$ or a group extension of $\LL_1$. 
In particular, $S$ is of elementary type.
\item Kula in \cite[3.12]{Kula_FinitelyGeneratedWittRings}: If $S$ has at most 12 quaternions, then $S$ is of elementary type.
\item\label{ETC5} Carson and Marshall in \cite[§ 5]{CarsonMarshall_Decomposition}: If $n\le 5$, then $S$ is of elementary type.
\item\label{ETC_iteration}
If $n\ge 4$, it is known that $S$ is decomposable or a group extension in the following cases:
\begin{enumerate}[label=(\roman*)]
    \item\label{ETC_iteration_Deg} $S$ is degenerate (Marshall, cf. normalized decomposition in \ref{normalized decomposition})
    \item\label{ETC_iteration_Rigid} $S$ is nondegenerate and rigid (Bos in \cite[2.10]{Bos_StructureWittRing})
    \item\label{ETC_iteration_s1_2rigid} $s(S)=1$ and $S$ is $2$-rigid (Carson and Marshall in \cite[3.10]{CarsonMarshall_Decomposition})
    \item\label{ETC_iteration_s2_2rigid} $s(S)=2$ and there is a $2$-rigid element in $D_S(\qf{1,1})$ (Cordes in \cite[3.1/4.1]{Cordes_BinaryRepr4})
    \item\label{ETC_iteration_s4_2rigid} $s(S)=4$ and $-1$ is a $2$-rigid element (Cordes in \cite[2.6]{Cordes_BinaryRepr4}).
\end{enumerate}
\end{enumerate}
\end{proposition}

The central idea behind our computational approach to verify the 3ETC in new cases, which will be explained in more detail in \Cref{sec_Computation}, is to find a strategy to efficiently create a list of all quaternionic structures satisfying $AP_3$ of a given order $n$.
Based on Marshall's normalized decomposition \Cref{normalized decomposition}, we can count the number of elementary type structures of order $n$ for any $n$, see \cite[4.5]{CarsonMarshall_Decomposition}.
We can then verify the 3ETC if we can show that the length of the list does not exceed the number of elementary type structures.
In fact, we do not have to compute the list of all $AP_3$ structures.
By \Cref{results on the etc}, we already know numerous situations in which the ETC and thus also the 3ETC holds, so we do not have to consider these cases.
This reduction of cases is considered in detail in \Cref{lem:reduction_technique}.

Our aim for the remainder of this section is thus to derive a more refined version of the aforementioned result \cite[4.5]{CarsonMarshall_Decomposition} which also gives us the number of quaternionic structures of elementary type of a given order that fulfill certain additional properties.

\begin{notation}
    We write $\mathcal{E}_k(n)$ for the set of isomorphism types of strictly $k$-rigid  nondegenerate elementary type structures of order $n$ and 
    \[\overline{\mathcal{E}}_k(n) := \{S \in \mathcal{E}_k(n) \mid S \text{ is indecomposable}\}.\] 
    Furthermore, we set $e_k(n) := |\mathcal{E}_k(n)|$ and $\overline{e}_k(n) := |\overline{\mathcal{E}}_k(n)|$. 
    Lastly, we write $\mathcal{E}(n)\ /\ \mathcal{E}'(n)\ /\ \mathcal{E}_\text{d}(n)\ /\ \mathcal{E}_\Delta(n)$ for the sets of elementary types of order $n \ /\ $ elementary types of order $n$ with $1=-1$ \ /\ degenerate elementary types of order $n\ /\ $  group extensions of order $n$ and respectively $e(n), \ e'(n), e_\text{d}(n), e_\Delta(n)$ for their cardinalities. 
\end{notation}

\begin{definition}
    Let $k\in \N_0$ and $n\in \N$. 
    We define the set $P_{k,n}$ as 
    \[P_{k,n}:=\{(m_{i,j})_{0\le i\le k,\ 1\le j\le n}\in \Mat_{(k+1)\times n}(\N_0)\mid \sum_{i,j} im_{i,j}=k,\ \ \sum_{i,j}jm_{i,j}=n\}.\]
    Note that $P_{0,n}$ is in bijection to the set of partitions of $n$. 
\end{definition}

\begin{lemma}\label{counting ets}
    Let $k,n\in \N_0$. 
    The following hold:
    \begin{enumerate}[label=(\alph*)]
    \item\label{counting1} $e_\text{d}(n)=e(n-1)+e'(n-1)-e'(n-2)$ for $n\ge 2$
    \item\label{counting2} $e_\Delta(n)=e(n-1)$ for $n>2$.
    \item\label{counting3} We have $e_0(0)=e_0(1)=e_0(2)=1=\overline{e}_0(1)$, $\overline{e}_0(0)=\overline{e}_0(2)=0$ and $e_1(2)=\overline{e}_1(2)=2$. 
    For $k\ge n$, we always have $e_k(n)=0$.
    \item\label{counting4} If $n>2$ and $k\ge 2$, we have
	\begin{enumerate}[label=(\roman*)]
		\item\label{counting41} $\overline{e}_0(n)=e_0(n-1)$ and $e_0(n)=\sum_{m\in P_{0,n}}\prod_{j=1}^n \binom{\overline{e}_0(j)+m_{0,j}-1}{m_{0,j}}$
		\item\label{counting42} $\overline{e}_1(n)=e_\Delta(n)-\overline{e}_0(n)$ and $e_1(n)=\sum_{j=2}^n \overline{e}_1(j)e_0(n-j)$
		\item\label{counting43} $\overline{e}_k(n)=\begin{cases} 2,&\text{ if } k+1=n \text{ is even}\\ 1,&\text{ if } k+1=n \text{ is odd} \\ 0,&\text{ else.}\end{cases}$
		\item\label{counting44} $e_k(n)=\sum_{m\in P_{k,n}} \prod_{0\le i\le k,\ 1\le j \le n} \binom{\overline{e}_i(j)+m_{i,j}-1}{m_{i,j}}$. 
	\end{enumerate}
    \end{enumerate}
\end{lemma}
\begin{proof}
Recall that by Marshall's normalized decomposition in \ref{normalized decomposition}, every quaternionic structure $S$ decomposes as $\LL_{1,\varepsilon}^k\times S_1\times\ldots\times S_\ell$ for unique nondegenerate indecomposable quaternionic structures $S_1,\ldots,S_\ell$ (up to permutation and isomorphism) and unique $k\in \N_0$.
Further, we have $\varepsilon=0$ if and only if $s(S)=1$ and $\varepsilon=1$ else.
    
\ref{counting1}:    
The degenerate structures are clearly the ones having $k>0$ in the above notation.
If $S$ is a degenerate structure with $1=-1$, it can uniquely be written as $S\cong\LL_{1,0}\times S'$ for $S'\in \mathcal{E}'_{n-1}$.
This shows that there are $e'(n-1)$ degenerate structures with $1=-1$ having order $n$.
If $S$ is a degenerate structure with $1\neq -1$, we uniquely write $S\cong \LL_{1,1}\times S''$ for a structure $S''$ of order $n-1$ that either also satisfies $1\neq -1$ or is nondegenerate with $1=-1$. 
Since the number of degenerate structures of order $n-1$ satisfying $1=-1$ is $e'(n-2)$ by the above argument, there are 
\[(e(n-1)-e'(n-1))+(e'(n-1)-e'(n-2))=e(n-1)-e'(n-2)\] 
choices for $S''$.

\ref{counting2}:    
Every group extension of order $n>2$ is indecomposable by \Cref{ex:ets_totdeg_and_small_group_ext}, so the result follows from the fact that $S\Delta\cong S'\Delta\iff S\cong S'$ for any quaternionic structures $S,S'$ (cf.\cite[5.20]{Marshall_AbstractWittRings}).
    
\ref{counting3}:    
This is clear from the definitions and the list of small quaternionic structures in \Cref{list of small ets with invariants} along with \Cref{lem:rigidity_ets} \ref{lem:rigidity_ets1} and \ref{lem:rigidity_ets2}.
    
\ref{counting4}
\ref{counting41}:
The set $\mathcal{E}_0(n)$ contains exactly the pythagorean elementary type structures of order $n$ by \Cref{lem:rigidity_ets} \ref{lem:rigidity_ets1}.
The only indecomposable elementary type structures of order $n$ are the local type structures and the group extensions.
By \Cref{lem:rigidity_ets}, local type structures are not pythagorean and since $n>2$, group extensions are pythagorean if and only if they are extensions of pythagorean structures of order $n-1$.
Hence, $\overline{e}_0(n)=e_0(n-1)$. 

For $e_0(n)$, we count the number of products of indecomposable pythagorean structures of smaller orders that are adding up to $n$, making use of \Cref{lem:rigidity_ets} \ref{lem:rigidity_ets4}.
In other words, we just count the number of possible products for each partition $(m_{0,1},\ldots,m_{0,n})\in P_{0,n}$ of the order $n$, we can choose $m_{0,i}$ factors from $\overline{\mathcal{E}_0}(i)$ for all $i$. 
This gives the desired equation.
    
\ref{counting42}:
By \Cref{lem:rigidity_ets} \ref{lem:rigidity_ets2}, group extensions of order $n>2$ are strictly $1$-rigid if and only if they are not pythagorean.
As in \ref{counting41} and taking into account \Cref{lem:rigidity_ets} \ref{lem:rigidity_ets3}, these are the only strictly $1$-rigid indecomposable elementary type structures, so the first part of the assertion follows.

For the second part, we proceed as in \ref{counting41} and note that by \Cref{lem:rigidity_ets} \ref{lem:rigidity_ets4}, $\mathcal{E}_1(n)$ consists of the nondegenerate elementary type structures that have exactly one strictly $1$-rigid factor in the normalized decomposition while all other factors are pythagorean.
Hence, such a structure can be uniquely written as the product of an indecomposable strictly $1$-rigid structure of order $j\in \{2,\ldots,n\}$ and a pythagorean structure of order $n-j$.
Note that we can assume $j\neq 1$ here as strictly $1$-rigid structures of order $1$ are degenerate and also that any pythagorean structure is nondegenerate since by \Cref{lem:rigidity_ets} \ref{lem:rigidity_ets4}, it cannot have any degenerate factor.
    
\ref{counting43}:
This is an immediate consequence of \Cref{lem:rigidity_ets} \ref{lem:rigidity_ets2} and \ref{lem:rigidity_ets3} as the only indecomposable elementary type structures of order $n>2$ are group extensions and local type structures.
    
\ref{counting44}:
By \Cref{lem:rigidity_ets} \ref{lem:rigidity_ets4}, for $S_\ell\in \mathcal{E}_\ell$ and $S_{\ell'}\in \mathcal{E}_{\ell'}$, we get $S_\ell\times S_{\ell'}\in \mathcal{E}_{\ell+\ell'}$. 
Hence, we have to count all products of nondegenerate indecomposable structures $S_{i,j}$ from $\mathcal{E}_i(j)$ such that $\sum_i i=k$ and $\sum_j j=n$. 
By definition of $P_{k,n}$, the desired equation follows in the same way as in \ref{counting41}.
\end{proof}

\begin{remark}\label{number of small ets}
    \Cref{counting ets} enables us to refine the list Carson and Marshall gave in \cite{CarsonMarshall_Decomposition} calculating $e(n)$ and $e'(n)$ for small values of $n$, see \Cref{tbl_e_n}. 
    These values will be used later on for the computations in \ref{sec_Computation}. 
    We get the following:
    \newcolumntype{x}[1]{%
    >{\raggedleft\hspace{0pt}}p{#1}}%
    \newcommand{\tn}{\tabularnewline}%
    \begin{center}
    \begin{table}[ht!]
\footnotesize
            \begin{tabular}{x{0.07\textwidth}|x{0.05\textwidth}|x{0.05\textwidth}|x{0.05\textwidth}|x{0.06\textwidth}|x{0.06\textwidth}|x{0.06\textwidth}|x{0.06\textwidth}|x{0.06\textwidth}|x{0.06\textwidth}|x{0.07\textwidth}}
            $n$&0&1&2&3&4&5&6&7&8&9\tn\hline
            $e(n)$&1&3&6&17&51&155&492&1600&5340&18150\tn
            $e'(n)$&1&1&2&4&10&22&54&130&328&832\tn
            $e_\text{d}(n)$&0&2&3&7&19&57&167&524&1676&5538\tn
            $e_0(n)$&1&1&1&2&4&8&17&36&79&175\tn
            $e_1(n)$&0&0&2&7&22&71&227&735&2414&8059\tn
            $e_2(n)$&0&0&0&1&4&14&60&234&903&3379\tn
            $e_3(n)$&0&0&0&0&2&4&13&45&173&667\tn
            $e_4(n)$&0&0&0&0&0&1&6&19&65&231\tn
            $e_5(n)$&0&0&0&0&0&0&2&6&19&65\tn
            $e_6(n)$&0&0&0&0&0&0&0&1&9&27\tn
            $e_7(n)$&0&0&0&0&0&0&0&0&2&8\tn
            $e_8(n)$&0&0&0&0&0&0&0&0&0&1
        \end{tabular}
        \caption{Values of $e(n), e'(n), e_d(n), e_k(n)$ for small $n$} \label{tbl_e_n}
    \end{table}
    \end{center}
\end{remark}

\section{The Abstract 2-Brauer Group}\label{sec_2Brauer}

So far, we have not imposed any additional structure on the set of abstract quaternions except distinguishing the split quaternion.
Classically, the quaternions generate the $2$-torsion part of the so called \emph{Brauer group}.
We briefly recall this in the following:

\begin{remark}\label{rem:classical_Brauer}
    Let $F$ be a field of characteristic not $2$ with Witt ring $W$ and fundamental ideal $I$.
    The \emph{Brauer group} $\text{Br}(F)$ of $F$ consists of equivalence classes of the central simple algebras over $F$.
    We do not give any details here, but refer to \cite[chapter IV]{Lam_IntroQF} for a thorough introduction.
    In 1981, Merkurjev \cite{Merkurjev_KTheory} showed that $I^2/I^3$ is isomorphic to the $2$-torsion part of $\text{Br}(F)$ which is generated by the equivalence classes of quaternion algebras.
    More precisely, an isomorphism is given by mapping $\Pfister{a,b}+I^3$ to the equivalence class of $(a,b)_F$, the quaternion algebra over $F$ with slots $a$ and $b$.
\end{remark}

As for example pointed out in \cite[p. 139]{Lam_IntroQF}, it is possible to describe the $2$-torsion part of the Brauer group of a field by generators and relations.
The generators are the (equivalence classes of) quaternion algebras and the relations are bimultiplicativity and so called Steinberg relations.
We replicate this in the following way.

\begin{definition}
	Let $S=(G,-1,q,0,Q)$ be a quaternionic structure.
	The \emph{abstract $2$-Brauer group of $S$} $(B(S),\ast_S)$ or simply $(B,\ast)$ is the free abelian group generated by $Q$ with 
	\begin{align}
	    q(a,b)\ast q(a,c)=q(a,bc) \quad \text{for any } a,b,c\in G \tag{LR}\label{eq:Linkage_Relation}
    \end{align}
    as its sole relations.
	We usually briefly call this group the \emph{2B-group of $S$}.
	and $\ast$ the \emph{quaternionic product}.
	We refer to \eqref{eq:Linkage_Relation} as the \emph{linkage relation}.
	For any $a,b\in G$, write $\overline{q}(a,b)\in B$ for the quaternion $q(a,b)$ modulo linkage relations.
	We say that $Q$ \emph{naturally embeds into $B$} if for all $a,b,c,d\in G$ $\overline{q}(a,b)=\overline{q}(c,d)\implies q(a,b)=q(c,d)$.
	In this case, we simply write $q(a,b)=\overline{q}(a,b)$ for all $a,b\in G$.
\end{definition}

In the following, let always $S=(G,-1,q)$ be a quaternionic structure with quaternions $Q$ and 2B-group $B$.

\begin{lemma}\label{lem:properties_2B} 
	The following hold:
	\begin{enumerate}[label=(\alph*)]
		\item\label{lem:properties_2B1} $B$ is boolean with neutral element $0=\overline{q}(1,1)$.
		\item\label{lem:properties_2B2} For any $\alpha\in \Z$ and $a,b\in G$, we have $\overline{q}(a^\alpha,b)=\overline{q}(a,b)^\alpha=\overline{q}(a,b^\alpha)$.
		\item\label{lem:properties_2B3} If $S$ is of order $n\in \N_0$, then 
		\[\dim(B)\le \begin{cases} \binom{n}{2} &\text{ if } 1=-1 \\ \binom{n}{2}+1&\text{ else.}\end{cases}\]
	\end{enumerate}
\end{lemma}
\begin{proof}
	\ref{lem:properties_2B1}:
	By definition, we have $\overline{q}(a,b)\ast \overline{q}(a,b)=\overline{q}(a,b^2)=\overline{q}(a,1)=0$ and $\overline{q}(a,b)\ast 0=\overline{q}(a,b)\ast \overline{q}(a,1)=\overline{q}(a,b)$, so $B$ is a boolean group with neutral element $0\in Q$. 
	
	\ref{lem:properties_2B2}:
	Since both $G$ and $B$ are boolean groups, we can reduce to the case $\alpha\in \{0,1\}$. 
	If $\alpha=0$, then all quaternions are trivial by \Cref{lem:basic_properties_qs}. 
	If $\alpha=1$, all exponents are trivial, so there is nothing to show.
	
	\ref{lem:properties_2B3}:
	Choose a basis $(a_0,\ldots,a_{n-1})$ of $G$ with $a_0=-1$ if $-1\neq 1\in G$. 
	$B$ is generated by all $\overline{q}(\prod_{i\in I} a_i,\prod_{j\in J} a_j)$ for subsets $I,J\subset \{0,\ldots,n-1\}$. 
	Iteratively applying \eqref{eq:Linkage_Relation}, we get
	\[\overline{q}(\prod_{i\in I} a_i,\prod_{j\in J} a_j)=\Ast_{i\in I} \overline{q}(a_i,\prod_{j\in J} a_j)=\Ast_{i\in I,j\in J} \overline{q}(a_i,a_j),\]
	so $B$ is generated by $\{\overline{q}(a_i,a_j)\mid i,j\in \{0,\ldots,n-1\}\}$. 
	For any $i,j\in \{0,\ldots,n-1\}$, we have $\overline{q}(a_i,a_j)=\overline{q}(a_j,a_i)$ by \ref{QS2} and, applying \Cref{lem:basic_properties_qs} \ref{lem:basic_properties_qs_1}, also
	\[\overline{q}(a_i,a_i)=\overline{q}(-1,a_i)=\begin{cases} \overline{q}(a_0,a_i) &\text{ if } -1\not=1\\ 0 &\text{ else.}\end{cases}\] 
	Hence, $B$ is generated by 
	$(\overline{q}(a_0,a_0),\overline{q}(a_i,a_j)\mid 0\le i<j\le n-1)$ and in the case $-1=1\in G$, we can further exclude $\overline{q}(a_0,a_0)=0$ from this set of generators.
	The assertion follows.
\end{proof}

In order to now prove that $AP_3$ is an equivalent condition to $Q$ embedding naturally into $B$, we follow the ideas of \cite{MarshallYucasAP3} which in turn largely follows the classical concepts (cf. \cite[Chapter V.3]{Lam_IntroQF}).

It was mentioned in \Cref{rem:Pfister_forms_and_fundamental_ideal} that the dimension (modulo $2$) and the discriminant uniquely describe quadratic forms modulo $I^2$.
For $I^3$, defining the proper invariant is slightly more involved.
We take the following intermediate step.
As before, all terms align with the classical versions. 

\begin{definition}
	Let $\varphi=\qf{a_1,\ldots,a_n}$ be a form over $S$.
	We define its \emph{Hasse invariant} as
	\[h(\varphi):=\Ast_{1\le i<j\le n} \overline{q}(a_i,a_j).\]
\end{definition}

\begin{lemma}\label{lem:Hasse_inv}
	Let $\varphi,\psi$ be forms over $S$ and choose $c\in G$. 
	The following hold:
	\begin{enumerate}[label=(\alph*)]
		\item\label{lem:Hasse_inv1}
		$\varphi\cong \psi\implies h(\varphi)=h(\psi)$
		\item\label{lem:Hasse_inv2}
		$h(\varphi\perp\psi)=h(\varphi)\ast h(\psi)\ast \overline{q}(\det(\varphi),\det(\psi))$
		\item\label{lem:Hasse_inv3}
		If $\dim(\varphi)=m$, then 
		\[h(c\varphi)=h(\varphi)\ast \overline{q}(c,c^{\frac{m(m-1)}{2}}\det(\varphi)^{m-1}).\]
		In particular, if $\dim(\varphi)\equiv_4 0$ and $\det(\varphi)=1$ or $\dim(\varphi)\equiv_4 1$, then $h(c\varphi)=h(\varphi)$.  
	\end{enumerate}
\end{lemma}
\begin{proof}
	We write $\varphi=\qf{a_1,\ldots,a_m}$ and $\psi=\qf{b_1,\ldots,b_n}$.
	\ref{lem:Hasse_inv1}: 
	In this case, we have $m=n$.
	If $m=1$, we get $h(\varphi)=0=h(\psi)$ as the empty quaternionic product.
	If $m=2$, we have $h(\varphi)=\overline{q}(a_1,a_2)=\overline{q}(b_1,b_2)=h(\psi)$ by definition of isometry. 
	Hence, it remains to inductively consider the case $m>2$:
	Choose $a,b\in G$ and a form $\chi$ of dimension $m-2$ such that
	\[\varphi':=\qf{a_2,\ldots,a_m}\cong \qf{a}\perp\chi,\ \psi':=\qf{b_2,\ldots,b_m}\cong \qf{b}\perp \chi,\ \qf{a_1,a}\cong \qf{b_1,b}.\]
	By definition of $h$ and the induction hypothesis,
	\begin{align*}
	    h(\varphi)&=h(\varphi')\ast \overline{q}(a_1,\det(\varphi'))\\
	    &=h(\qf{a}\perp\chi)\ast \overline{q}(a_1,\det(\varphi'))\\
	    &=h(\chi)\ast \overline{q}(a,\det(\chi))\ast \overline{q}(a_1,\det(\varphi')).
	\end{align*}
	Note that the determinant is invariant under isometry which means that $\det(\varphi')=a\det(\chi)$. 
	Hence, we can use the bimultiplicativity of $\ast$ to simplify
	\begin{align*}
	    h(\varphi)&=h(\chi)\ast \overline{q}(a,\det(\chi))\ast \overline{q}(a_1,a)\ast \overline{q}(a_1,\det(\chi))\\
	    &=h(\chi)\ast \overline{q}(a_1a,\det(\chi))\ast \overline{q}(a_1,a).
	\end{align*}
	Doing the same calculation for $\psi$ yields
	\[h(\psi)=h(\chi)\ast \overline{q}(b_1b,\det(\chi))\ast \overline{q}(b_1,b).\]
	Now, we also have $a_1a=b_1b$ and $q(a_1,a)=q(b_1,b)$ since $\qf{a_1,a}\cong\qf{b_1,b}$, so the assertion follows.
	
	\ref{lem:Hasse_inv2}:
	By definition, we have
	\[h(\varphi\perp\psi)=h(\varphi)\ast h(\psi)\ast \Ast_{i,j} \overline{q}(a_i,b_j).\]
	As in the proof of \Cref{lem:properties_2B} \ref{lem:properties_2B3}, the right hand side factors above can be factorized as $\overline{q}(\prod_{i=1}^m a_i,\prod_{j=1}^n b_j)=\overline{q}(\det(\varphi),\det(\psi))$, proving the assertion.
	
	\ref{lem:Hasse_inv3}:
	We get $h(c\varphi)=\Ast_{i<j} \overline{q}(ca_i,ca_j)$ and for any $i,j$, we apply the bimultiplicativity and symmetry to get
	\[\overline{q}(ca_i,ca_j)=\overline{q}(c,c)\ast \overline{q}(c,a_j)\ast \overline{q}(a_i,c)\ast \overline{q}(a_i,a_j)=\overline{q}(a_i,a_j)\ast \overline{q}(c,ca_ia_j).\]
	Thus, we get 
	\[h(c\varphi)=h(\varphi)\ast \Ast_{i<j} \overline{q}(c,ca_ia_j)=h(\varphi)\ast \overline{q}(c,\alpha)\]
	with 
	\[\alpha=\prod_{i<j} ca_ia_j=c^{\frac{m(m-1)}{2}} \prod_{i=1}^m a_i^{m-1}=c^{\frac{m(m-1)}{2}}\det(\varphi)^{m-1},\]
	since we have $\binom{m}{2}=\frac{m(m-1)}{2}$ factors and every $a_i$ appears exactly once in combination with every $a_j$ for $j\neq i$.
	In particular, if $m\equiv_4 0$ and $\det(\varphi)=1$ or $m\equiv_4 1$, we have $\alpha=1$, completing the proof.
\end{proof}

Note that the above particularly shows that $h$ is not necessarily an invariant for Witt equivalence classes since in general 
\[h(\varphi\perp \qf{1,-1})=h(\varphi)\ast \overline{q}(1,-1)\ast \overline{q}(\det(\varphi),-1)=h(\varphi)\ast \overline{q}(\det(\varphi),-1)\neq h(\varphi).\]
Hence, this invariant has to be tweaked a little bit.

\begin{definition}
	Let $\varphi$ be a form over $S$.
	The \emph{Witt invariant} of $\varphi$ is defined as $w(\varphi):=h(\varphi)\ast \overline{q}(-1,c_\varphi)$ with
	\[c_\varphi:=
	\begin{cases}
		1 &\text{, if } \dim(\varphi)\equiv_8 1,2\\
		-\det(\varphi) &\text{, if } \dim(\varphi)\equiv_8 3,4\\
		-1 &\text{, if } \dim(\varphi)\equiv_8 5,6\\
		\det(\varphi) &\text{, if } \dim(\varphi)\equiv_8 7,0\\
	\end{cases}\]
\end{definition}

\begin{lemma}\label{lem:Witt_inv_basics}
	Let $\varphi,\psi$ be forms over $S$. The following hold:
	\begin{enumerate}[label=(\alph*)]
		\item\label{lem:Witt_inv_basics1}
		If $\varphi\in I^2$, then $w(\varphi\perp\psi)=w(\varphi)\ast w(\psi)$.
		\item\label{lem:Witt_inv_basics2}
		The Witt invariant is an invariant for Witt equivalence classes.
		\item\label{lem:Witt_inv_basics3}
		$w(c\Pfister{a,b})=\overline{q}(a,b)$ for all $a,b,c\in G$.
	\end{enumerate}
\end{lemma}
\begin{proof}
	\ref{lem:Witt_inv_basics1}:
	Since forms in $I^2$ have trivial discriminant and even dimension by \Cref{rem:Pfister_forms_and_fundamental_ideal}, we get $\det(\varphi)=(-1)^k$ for $k\in \{0,1\}$ where $k=0$ if and only if $\dim(\varphi)\equiv_4 0$.
	It follows that 
	\begin{align*}
	    w(\varphi)\ast w(\psi)&=h(\varphi)\ast h(\psi)\ast \overline{q}(-1,c_\varphi)\ast \overline{q}(-1,c_\psi)\\
	    &=h(\varphi)\ast h(\psi)\ast \overline{q}(-1,c_\varphi\cdot c_\psi)
	\end{align*}
	and, also applying \Cref{lem:Hasse_inv} \ref{lem:Hasse_inv2},
	\[w(\varphi\perp\psi)=h(\varphi\perp\psi)\ast \overline{q}(-1,c_{\varphi\perp\psi})=h(\varphi)\ast h(\psi)\ast \overline{q}((-1)^k,\det(\psi)).\]
	Since $G$ is a boolean group and $\overline{q}(1,a)=0$ for all $a\in G$, we have 
	\[\overline{q}((-1)^k,\det(\psi))=\overline{q}(-1,\det(\psi)^k)\] and it suffices to show that
	\begin{equation}\label{eq:cases_Witt_inv}
		c_\varphi\cdot c_\psi=c_{\varphi\perp\psi}\cdot \det(\psi)^k.
	\end{equation}
	If $\dim(\varphi)\equiv_8 2$, then $k=1$ and thus $\det(\varphi)=-1$.
	One checks that we get the following cases depending on $\dim(\psi)$, making use of $\det(\varphi\perp\psi)=\det(\varphi)\det(\psi)=-\det(\psi)$:
	\[c_\psi=
	\begin{cases}
		1,\\
		-\det(\psi),\\
		-1,\\
		\det(\psi)
	\end{cases}
	\quad\text{and}\quad
	c_{\varphi\perp\psi}=
	\begin{cases}
		\det(\psi), &\text{ if } \dim(\psi)\equiv_8 1,2\\
		-1, &\text{ if } \dim(\psi)\equiv_8 3,4\\
		-\det(\psi) &\text{ if } \dim(\psi)\equiv_8 5,6\\
		1, &\text{ if } \dim(\psi)\equiv_8 7,0.
	\end{cases}
	\]
	Clearly, \eqref{eq:cases_Witt_inv} holds in these cases.
	The other cases $\dim(\varphi)\equiv_8 4,6,0$ are verified in the same way.
	
	\ref{lem:Witt_inv_basics2}:
	Since $w(\mathcal{H})=h(\mathcal{H})=\overline{q}(1,-1)=0$ and $\mathcal{H}=0\in I^2$, this follows immediately from \ref{lem:Witt_inv_basics1}.
	
	\ref{lem:Witt_inv_basics3}:
	Take $a,b,c\in G$. 
	We have 
	\[w(c\Pfister{a,b})=w(\qf{c,-ac,-bc,abc})=h(\qf{c,-ac,-bc,abc})\ast \overline{q}(-1,-1)\]
	by definition as (scaled) Pfister forms have determinant $1$.
	We can apply \Cref{lem:Hasse_inv} \ref{lem:Hasse_inv3} for the same reason to get $h(c\Pfister{a,b})=h(\Pfister{a,b})$.
	Applying the definition of the Hasse invariant and the bimultiplicativity of $\overline{q}$ along with \Cref{lem:basic_properties_qs}, we have
	\begin{align*}
		h(\qf{1,-a,-b,ab})&=\overline{q}(-a,(-b)ab)\ast \overline{q}(-b,ab)\\
		&=\overline{q}(-a,-a)\ast \overline{q}(-b,a)\\
		&=\overline{q}(-1,-a)\ast \overline{q}(-b,a).
	\end{align*}
	All in all, we get 
	\begin{align*}
		w(c\Pfister{a,b})&=\overline{q}(-1,-1)\ast \overline{q}(-1,-a)\ast \overline{q}(-b,a)\\
		&=\overline{q}(-1,a)\ast \overline{q}(-b,a)\\
		&=\overline{q}(b,a)=\overline{q}(a,b).
	\end{align*}
\end{proof}

We can thus recreate the following result \cite[Thm. 3.3]{MarshallYucasAP3} of Marshall and Yucas in our setting:

\begin{proposition}\label{thm:AP_3}
	The Witt invariant defines a surjective group homomorphism $w:I^2\to B$ to the 2B-group of $S$ with $I^3\subset \Ker{w}$.
\end{proposition}
\begin{proof}
	The previous \Cref{lem:Witt_inv_basics} \ref{lem:Witt_inv_basics1} and \ref{lem:Witt_inv_basics2} shows that the restriction of $w$ to $I^2$ is a well-defined group homomorphism. 
	The surjectivity follows from \Cref{lem:Witt_inv_basics} \ref{lem:Witt_inv_basics3} because $\{\overline{q}(a,b)\mid a,b\in G\}$ generates $B$.
	Furthermore, any $3$-fold Pfister form
	\[\Pfister{a,b,c}=\Pfister{a,b}\perp (-c)\Pfister{a,b}\] 
	is mapped to $0$, again by \Cref{lem:Witt_inv_basics} \ref{lem:Witt_inv_basics3}.
	As $3$-fold Pfister forms generate $I^3$, this shows $I^3\subset\Ker{w}$.
\end{proof}

We can now show that the abstract $2$-Brauer group is rightfully named in light of Merkurjev's theorem as mentioned in \Cref{rem:classical_Brauer}.

\begin{corollary}\label{cor:2B_is_I2I3}
	We have $B\cong I^2/I^3$ as groups.
\end{corollary}
\begin{proof}
	We have seen in the previous \Cref{thm:AP_3} that the Witt invariant defines a group epimorphism 
	$\overline{w}:I^2/I^3\to B$ satisfying $\overline{w}(\Pfister{a,b}+I^3)=\overline{q}(a,b)$.
	Hence, it remains to show that this is also injective.
	
	Take $a,b,c\in G$. 
	We have $\Pfister{a,b}+\Pfister{a,b}=\Pfister{a,b,-1}\in I^3$ and
	\begin{align*}
		\Pfister{a,b}+\Pfister{a,c}-\Pfister{a,bc}&=\qf{1,-a}\cdot\qf{1,-b,1,-c,-1,bc}\\
		&=\qf{1,-a}\cdot\qf{1,-b,-c,bc}\\
		&=\Pfister{a,b,c}\in I^3.
	\end{align*}
	Thus, $I^2/I^3$ is a boolean group satisfying 
	\[\Pfister{a,b}+I^3+\Pfister{a,c}+I^3=\Pfister{a,bc}+I^3.\] 
	
	Since $I^2$ is generated by the $2$-fold Pfister forms as an additive group, the factor group $I^2/I^3$ is generated by $\{\Pfister{a,b}+I^3\mid a,b\in G\}$.
	As $B$ is generated by $Q$ with only the linkage relations \eqref{eq:Linkage_Relation} which are also satisfied by the corresponding generators of $I^2/I^3$, it follows that $\overline{w}$ is injective. 
\end{proof}

\begin{remark}\label{rem:relations_in_I2I3}
	The previous \Cref{cor:2B_is_I2I3} actually shows that the generators $\Pfister{a,b}+I^3$ of $I^2/I^3$ have the linkage relations as their only relations.
\end{remark}

Analogously to \cite[Prop. 3.23]{Marshall_AbstractWittRings}, we can state an equivalent condition to $AP_3$:

\begin{proposition}\label{prop:AP_3_equiv_embedding}
	$S$ satisfies $AP_3$ if and only if $Q$ naturally embeds into the 2B-group $B$.
	In particular, $Q$ embeds into $B$ if $S$ is realized by a field.
\end{proposition}
\begin{proof}
	First assume that $S$ satisfies $AP_3$ and take $a,b,c,d\in G$.
	If $\overline{q}(a,b)=\overline{q}(c,d)$, then it follows from \Cref{cor:2B_is_I2I3} that $\Pfister{a,b}=\Pfister{c,d}$ modulo $I^3$.
	But then
	\[\Pfister{a,b}-\Pfister{c,d}=\qf{1,-a,-b,ab,-1,c,d,-cd}=\qf{-a,-b,ab,c,d,-cd}\in I^3\]
 	and $AP_3$ show that this form of dimension $6$ is hyperbolic which means that $\Pfister{a,b}=\Pfister{c,d}\in W(S)$.
	Hence, $q(a,b)=q(c,d)$ which shows that $Q$ naturally embeds into $B$.
	
	If conversely $Q$ naturally embeds into $B$, take $\varphi\in I^3$ with $\dim(\varphi)<8$. 
	Since $\dim(\varphi)$ is even, we can assume $\dim(\varphi)=6$, adding hyperbolic planes if necessary.
	We have to show that $\varphi$ is hyperbolic.
	Write $\varphi=\qf{a,b,c,a',b',c'}$.
	Without loss of generality - if necessary, scale by $abc$ - we may assume $c=ab$.
	Since $\varphi\in I^2$, we have $\disc(\varphi)=1$ and thus $\det(\varphi)=-1$ and $c'=-a'b'$.
	It follows that up to Witt equivalence
	\[\varphi=\qf{a,b,ab,a',b',-a'b'}=\Pfister{-a,-b}-\Pfister{a',b'}.\]
	Since $\varphi\in I^3=\Ker(w)$, we get $w(\Pfister{-a,-b})=w(-\Pfister{a',b'})=w(\Pfister{a',b'})$ and hence $\overline{q}(-a,-b)=\overline{q}(a',b')$ by \Cref{lem:Witt_inv_basics} \ref{lem:Witt_inv_basics3}.
	Making use of the assumption, this shows that $q(-a,-b)=q(a',b')$.
	Since we can identify quaternions with $2$-fold Pfister forms, $\varphi=\Pfister{-a,-b}-\Pfister{a',b'}$ is hyperbolic.
\end{proof}

We slightly generalize Marshall's results from \cite[proof of 5.7, 5.16, 5.17]{Marshall_AbstractWittRings} on how the Arason-Pfister property \ref{eq_APk} behaves for direct products and group extensions.

\begin{proposition}\label{prop:AP_k_direct_prod_and_group_ext}
Let $S=(G,-1,q)$ and $S'=(G',-1',q')$ be quaternionic structures and $k\ge 2$
\begin{enumerate}[label=(\alph*)]
\item\label{prop:AP_k_direct_prop_and_group_ext1} $S\times S'$ satisfies $AP_k$ if and only if $S$ and $S'$ both satisfy $AP_k$.
\item\label{prop:AP_k_direct_prop_and_group_ext2} $S\Delta$ satisfies $AP_k$ if and only if $S$ satisfies both $AP_k$ and $AP_{k-1}$.
\end{enumerate}
\end{proposition}
\begin{proof}
    \ref{prop:AP_k_direct_prop_and_group_ext1}:
    By \Cref{prop:direct_product}, every element in $\phi\in W(S\times S')$ can be written as $(\varphi,\varphi')\in W(S)\times W(S')$ with $\dim(\varphi)\equiv_2 \dim(\varphi')$ and $\phi=0\in W(S\times S')$ if and only if $\varphi=0\in W(S)$ and $\varphi'=0\in W(S')$.
    More precisely, if $\phi=\qf{a_1a_1',\ldots,a_na_n'}$ with $a_i\in G$ and $a_i'\in G'$ for all $i\in \{1,\ldots,n\}$, we can write $\varphi=\qf{a_1,\ldots,a_n}$ and $\varphi'=\qf{a_1',\ldots,a_n'}$. 
    In particular, every $k$-fold Pfister form in $W(S_1\times S_2)$ can be written as $(\pi,\pi')\in W(S)\times W(S')$ with $\pi,\pi'$ being $k$-fold Pfister forms in $W(S)$ and $W(S')$.
    Thus, $AP_k$ for $W(S\times S')$ is an immediate consequence of $AP_k$ for $W(S)$ and $W(S')$.
    The other implication follows from $W(S)\times \{0\}$ and $\{0\}\times W(S')$ being subrings of $W(S\times S')$.
    
    \ref{prop:AP_k_direct_prop_and_group_ext2}:
    Write $S[x]:=S\Delta$.
    By \Cref{prop:group_extension}, every form in $W(S[x])$ can be written as $\varphi+ x\psi$ for forms $\varphi,\psi\in W(S)$ where the anisotropic part of $\varphi+x\psi$ over $S[x]$ is the sum of the anisotropic parts of $\varphi$ and $\psi$ over $S$, the latter scaled by $x$.
    Since we can add a zero to get $\varphi+ x\psi=\varphi+\psi-(1-x)\psi$, it follows that 
    \[I(S[x])=I(S)+(1-x)W(S)=I(S)+(1-x)I(S)^0.\]
    Note that 
    \[(1-x)^2=\Pfister{x,x}=\Pfister{-1,x}=(1+1)(1-x)\in (1-x)I(S)\] 
    by \Cref{lem:basic_properties_qs} \ref{lem:basic_properties_qs_1} and the identification of quaternions and $2$-fold Pfister forms.
    Thus, it follows iteratively that \[I(S[x])^k=(I(S)+(1-x)I(S)^0)^k=I(S)^k+(1-x)I(S)^{k-1}.\]
    To prove $AP_k$ for $S[x]$, take any form $\varphi\in I(S[x])^k$ with $\dim(\varphi)<2^k$.
    We write $\varphi=\varphi_1+(1-x)\varphi_2=\varphi_1+\varphi_2-x\varphi_2$ with $\varphi_1\in I(S)^k$ and $\varphi_2\in I(S)^{k-1}$.
    If the anisotropic part of $\varphi_2$ over $S$ has dimension less than $2^{k-1}$, we can apply $AP_{k-1}$ for $S$ to get $\varphi_2=0$ and it follows from $AP_k$ for $S$ that $\varphi=\varphi_1=0$.
    Else, the anisotropic part of $\varphi_1+\varphi_2$ over $S$ has dimension less than $2^{k-1}$ and $AP_{k-1}$ for $S$ shows that $\varphi_1+\varphi_2=0$.
    Then $\varphi=x\varphi_2$ and $\varphi_2=-\varphi_1\in I(S)^k$, so applying $AP_k$ for $S$ shows that $\varphi=0$.
    
    For the other implication, assume that $S[x]$ satisfies $AP_k$.
    As in \ref{prop:AP_k_direct_prop_and_group_ext1}, since $W(S)$ is a subring of $W(S[x])$, $AP_k$ for $W(S)$ follows immediately.
    For $AP_{k-1}$, take $\varphi\in I(S)^{k-1}$ with $\dim(\varphi)<2^{k-1}$.
    Then $\varphi+ x\varphi=\Pfister{-x}\varphi\in I(S[x])^k$ with dimension less than $2^k$, so $AP_k$ for $S[x]$ shows that $\varphi+ x\varphi=0\in W(S[x])$ which proves that $\varphi=0\in W(S)$.
\end{proof}

\begin{proposition}\label{prop:2B_of_product_and_extension}
    Let $S=(G,-1,q)$ and $S'=(G',-1',q')$ be quaternionic structures with 2B-groups $B:=B(S)$ and $B':=B(S')$.
	The following hold:
	\begin{enumerate}[label=(\alph*)]
		\item\label{prop:2B_of_product_and_extension1} 
		$B(S\times S')\cong B\times B'$.
		\item\label{prop:2B_of_product_and_extension2} 
		If $S$ satisfies $AP_3$, then $B(S\Delta)\cong B\times G$. 
	\end{enumerate}
\end{proposition}
\begin{proof}
	\ref{prop:2B_of_product_and_extension1}:
	By \Cref{prop:direct_product}, the set of quaternions of $S\times S'$ is $Q\times Q'$. 
	The linkage relations for $S\times S'$ are thus exactly the linkage relations for $S$ and $S'$ and it follows that the 2B-group of $S\times S'$ is isomorphic to $B\times B'$.
	
	\ref{prop:2B_of_product_and_extension2}: 
	Write $S\Delta=S[x]=(G\cup xG, -1, q_x)$ and $B_x$ for its 2B-group.
	By \Cref{prop:group_extension}, the quaternions $Q_x$ of $S\Delta$ can be viewed as a subset of $Q\times G$ with zero-element $(0,1)$, identifying $q_x(ax^\alpha,bx^\beta)$ with $(q(a,b),(-1)^{\alpha\beta}a^\beta b^\alpha)$. 
	Note that $S$ satisfies $AP_2$ by \Cref{rem:Pfister_forms_and_fundamental_ideal}.
	Hence, \Cref{prop:AP_k_direct_prod_and_group_ext} \ref{prop:AP_k_direct_prop_and_group_ext2} shows that $S\Delta$ also satisfies $AP_3$ and by \Cref{prop:AP_3_equiv_embedding}, we can naturally identify $Q$ and $Q_x$ with subsets of $B$ and $B_x$.
	
	Take $a,b\in G$ and $\alpha,\beta\in \{0,1\}$. 
	Making use of \Cref{lem:basic_properties_qs} \ref{lem:basic_properties_qs_2} and \Cref{lem:properties_2B} \ref{lem:properties_2B3} along with the linkage relations for $q_x$, we also get
	\begin{align*}
		q_x(ax^\alpha,bx^\beta)&=q_x(a,b)\ast q_x(a,x^\beta)\ast q_x(b,x^\alpha)\ast q_x(x^\alpha,x^\beta)\\
		&=q_x(a,b)\ast q_x(a^\beta b^\alpha,x)\ast q_x(x,x)^{\alpha\beta}\\
		&=q_x(a,b)\ast q_x(a^\beta b^\alpha,x)\ast q_x(-1,x)^{\alpha\beta}\\
		&=q_x(a,b)\ast q_x((-1)^{\alpha\beta}a^\beta b^\alpha,x).
	\end{align*}
	Hence, $B_x$ is generated by $\{q_x(a,b),q_x(a,x)\mid a,b\in G\}$. 
	We define
	\[B':=\spn(q_x(a,b)\mid a,b\in G) \text{ and } G':=\spn(q_x(a,x)\mid a\in G).\] 
	Since $q_x(a,x)\ast q_x(b,x)=q_x(ab,x)$ for all $a,b\in G$, we have $G'\cong G$. 
	It remains to show that $B'\cong B$ and $B'\cap G'$ is trivial. 
	For the former, take $a,b,c,d,e,f\in G$ and $\alpha,\beta\in \{0,1\}$. 
	Since we only have \eqref{eq:Linkage_Relation} as relations, we have
	\begin{align}
		\quad\quad q_x(a,b)\ast q_x(c,d) = q_x(ex^\alpha,fx^\beta) \label{eq:2B_group_ext}
	\end{align}
	if and only if there exist $r,s,t\in G$ and $\rho,\sigma,\tau\in \{0,1\}$ such that
	\[q_x(a,b)=q_x(rx^{\rho},tx^{\tau}),\  q_x(c,d)=q_x(sx^{\sigma},tx^{\tau}),\  q_x(ex^\alpha,fx^\beta)=q_x(rsx^{\rho+\sigma},tx^{\tau}).\]
	Using the identification from \Cref{prop:group_extension}, this is equivalent to the existence of such $r,s,t,\rho,\sigma,\gamma$ satisfying
	\[q(a,b)=q(r,t),\ \ q(c,d)=q(s,t),\ \ q(e,f)=q(rs,t) \text{ and}\]
	\[(-1)^{\rho\tau}r^\tau t^\rho=1,\ \ (-1)^{\sigma\tau}s^\tau t^\sigma=1,\ \ (-1)^{(\rho+\sigma)\tau}(rs)^\tau t^{\rho+\sigma}=(-1)^{\alpha\beta}e^\beta f^\alpha.\]
	The term on the right hand side is the product of the two previous ones, so $(-1)^{\alpha\beta}e^\beta f^\alpha=1$. 
	If $\alpha=1$, this means that $f=(-e)^\beta$, so either $\beta=0$ and $f=1$, meaning $q_x(ex^\alpha,fx^\beta)=q_x(ex,1)$, or $\beta=1$ and $q_x(ex^\alpha,fx^\beta)=q_x(ex,-ex)$. 
	In both cases $q_x(ex^\alpha,fx^\beta)$ is the trivial quaternion $q_x(1,1)$. 
	All in all, this shows that \eqref{eq:2B_group_ext} holds if and only if
	\[q(a,b)\ast q(c,d)=q(e,f) \text{ and } (\alpha=\beta=0 \text{ or } q_x(a,b)=q_x(c,d)).\]
	Again by \Cref{prop:group_extension}, we have $q_x(a,b)=q_x(c,d)\iff q(a,b)=q(c,d)$.
	Hence, it follows that $B'$ is generated by the set $\{q_x(a,b)\mid a,b\in G\}$ with the exactly same (linkage) relations as $B$ with generators $\{q(a,b)\mid a,b\in G\}$ and thus $B\cong B'$. 
	Additionally, the above calculations show that the only element of $G'$ that is generated by $\{q_x(a,b)\mid a,b\in G\}$ is the trivial quaternion, so $B'\cap G'$ is trivial which remained to prove.
\end{proof}

\begin{corollary}\label{cor:2B_max_dim}
	Both upper bounds in \Cref{lem:properties_2B} \ref{lem:properties_2B3} are attained by the iterated group extensions $\LL_0\Delta_n$ and $\LL_1\Delta_{n-1}$.
\end{corollary}
\begin{proof}
    We can apply \Cref{prop:2B_of_product_and_extension} \ref{prop:2B_of_product_and_extension2} since $\LL_0\Delta_{n}$ and $\LL_1\Delta_{n-1}$ satisfy $AP_3$ - this follows as these structures are elementary types or alternatively via an iterated application of \Cref{prop:AP_k_direct_prod_and_group_ext} \ref{prop:AP_k_direct_prop_and_group_ext2}.
	Clearly, the 2B-group of $\LL_0$ has dimension 0, so it follows from \Cref{prop:2B_of_product_and_extension} \ref{prop:2B_of_product_and_extension2} that the 2B-group of $\LL_0\Delta$ has dimension $0+0=0$, that of $\LL_0\Delta_2$ has dimension $(0+0)+1$ and iteratively, that of $\LL_0\Delta_n$ has dimension
	\[0+0+1+\ldots+(n-1)=\binom{n}{2}.\]
	The 2B-group of $\LL_1$ has dimension 1 as the unique nonzero quaternion necessarily must form a basis. 
	Hence, with the same iteration as before, the 2B-group of $\LL_1\Delta_{n-1}$ has dimension
	\[1+1+\ldots+(n-1)=\binom{n}{2}+1.\qedhere\] 
\end{proof}

\section{Quaternionic Matrices}\label{sec_QM}

\begin{definition}\label{def:AP3_structure_and_Qdim}
	Let $S$ be a quaternionic structure. 
	We refer to the \emph{Q-dimension} of $S$, written $\Qdim(S)$ as the dimension of its 2B-group. 
	If $S$ is of finite order and satisfies $AP_3$, we say that $S$ is an \emph{$AP_3$ structure}.
\end{definition}

In the following, we usually only consider $AP_3$ structures which means that the quaternions naturally embed into the abstract 2B-group by \Cref{prop:AP_3_equiv_embedding}.
As discussed in \Cref{rem:comparing_conjectures}, this includes quaternionic structures realized by fields with finitely many square classes and thus in particular elementary type structures. 
By \Cref{lem:properties_2B} \ref{lem:properties_2B3}, $AP_3$ structures also have finite $Q$-dimension and are uniquely described by the quaternions $q(a_i,a_j)$ generating the 2B-group for a basis $(a_0,\ldots,a_{n-1})$ of the square class group.
We want to make use of that fact to describe $AP_3$ structures by $(n\times n)$-matrices with quaternions as entries.

This will be the foundation for our computations in \Cref{sec_Computation}.
As before, write $S=(G,-1,q)$ for a fixed $AP_3$ structure of order $n\in \N_0$ with quaternions $Q$ and 2B-group $B(S)\cong I(S)^2/I(S)^3$ where $I(S)$ is the fundamental ideal of the Witt ring $W(S)$.

\begin{definition}
	For $k\in \N_0$, we take $I_k\subset \N_0$ such that $k=\sum_{i\in I_k} 2^{i}$ is the binary representation of $k$. 
	Then the binary operation $\ast_{\N_0}$ on $\N_0$ defined as
	\[k\ast_{\N_0} l:= \sum_{i\in I_k\triangle I_l} 2^{i}\]
    yields a boolean group $(\N_0,\ast_{\N_0})$. 
    For $n\in \N$, the set $\{0,1,\ldots,2^n-1\}$ defines a subgroup with basis $(2^k\mid 0\le k\le n-1)$.
	We call $\ast_{\N_0}$ the \emph{quaternionic product on $\N_0$}.
	Note that $\ast_{\N_0}$ is simply the xor operation in binary expansion.
	We often simply write $\ast$ instead of $\ast_{\N_0}$.
\end{definition}

Recall that we assume $S$ to be an $AP_3$ structure which means that the 2B-group $B$ is generated by $Q$.

\begin{definition}
	We say that a basis $\mathcal{A}=(a_0,\ldots,a_{n-1})$ of $G$ is a \emph{basis of S} if $-1=1$ or $a_0=-1$. 
	Given a basis $\mathcal{B}=(q_0,\ldots,q_{m-1})$ of $B$, we define the group isomorphism 
	\[\phi_\mathcal{B}: (B,\ast)\to (\{0,\ldots,2^m-1\},\ast)\quad\text{with}\quad q_k\mapsto 2^k.\]
	If $n\neq 0$, we define the \emph{quaternionic matrix of $S$ via $\mathcal{A},\mathcal{B}$} as 
	\[M_{\mathcal{A},\mathcal{B}}(S):=(\phi_\mathcal{B}(q(a_i,a_j)))_{0\le i,j\le n-1}\in \Mat_n(\N_0).\]
	If $\mathcal{B}$ is fixed, we often identify $(B,\ast)$ with $(\{0,\ldots,2^m-1\},\ast)$, writing $q_\mathcal{B}(a,b)$ or -- by abuse of notation -- simply $q(a,b)$ instead of $\phi_\mathcal{B}(q(a,b))$.	
\end{definition}

\begin{example}\label{ex:qm_basic_examples}
	\begin{enumerate}[label=(\alph*)]
		\item If $S$ is totally degenerate of order $n\in \N$, all quaternions are split, so $M_{\mathcal{A},\mathcal{B}}(S)$ is the $n\times n$-matrix consisting only of zeros for any choice of bases $\mathcal{A}$ of $S$ and $\mathcal{B}$ of $B(S)$.
		\item If $|Q|=2$, the matrix $M_{\mathcal{A},\mathcal{B}}(S)$ only depends on the choice of the basis $\mathcal{A}$ of $S$ since all split quaternions are associated with the entry $0$ and all others with the entry $1$. 
		More precisely, $M_{\mathcal{A},\mathcal{B}}(S)$ is just the Gram matrix of the symmetric bilinear map $q:G\times G\to \F_2$ for the basis $\mathcal{A}$.
		As Gram matrices, they depend heavily on the choice of the basis of $S$.
		We already point out here that the choice of bases is slightly reduced as we always take $-1$ as the first basis element whenever $-1\neq 1$.
		\item If $S$ is of order $1$, there is only one choice of basis of $S$, consisting of the nontrivial element of $G$, so we get a unique quaternionic matrix for each such structure.
		The totally degenerate $\LL_{1,0}$ and $\LL_{1,1}$ both yield the $(1\times 1)$-matrix $(0)$ and $\LL_1$ the matrix $(1)$.
		
	\end{enumerate}
\end{example}

In general, the matrix $M_{\mathcal{A},\mathcal{B}}$ is heavily dependent on the choice of the bases $\mathcal{A}$ and $\mathcal{B}$.
We compare matrices in the following way:

\begin{notation}\label{not:lex_order}
	For $n\in \N$, we bestow the set $\Mat_n(\N_0)$ with the lexicographical order $<_\text{lex}$, applied row by row, and extending the usual order on $\N_0$. 
	We simply say that a matrix is 'smaller/greater/...' than another if it is lexicographically smaller/greater/...
\end{notation}

\begin{definition}
	Let $\mathcal{A}$ be a basis of $S$ (recall that this means that $1=-1$ or $-1$ is the first basis element).
	Then the \emph{reduced quaternionic matrix of $S$ via $\mathcal{A}$} is defined as 
	\[M_\mathcal{A}(S):=\text{min}\{M_{\mathcal{A},\mathcal{B}}(S)\mid \mathcal{B}\text{ is a basis of } B(S)\}\]
	with regard to the order introduced above in \Cref{not:lex_order}.
	In the same way, the \emph{normal quaternionic matrix of $S$} is defined as the minimum of all reduced quaternionic matrices of $S$ and denoted by $M(S)$.
\end{definition}

Clearly, this is well defined as we assume $S$ to be finite, so there is only a finite number of quaternionic matrices of $S$.

We collect some properties of quaternionic matrices in order to later determine when a matrix determines a quaternionic structure -- or more precisely an $AP_3$ structure.
We can use the quaternionic product on $\N_0$ to calculate all quaternions starting from a quaternionic matrix in the following way:

\begin{definition}
	Given $n\in \N$ and a matrix $M\in M_n(\N_0)$, we define the \emph{completion of $M$} as the matrix 
	\[\overline{M}:=(\overline{M}_{k,\ell})_{0\le k,\ell\le 2^n-1}\in\Mat_{2^n}(\N_0)\quad\text{with}\quad
	\overline{M}_{k,\ell}=\Ast_{(i,j)\in I_k\times I_\ell} M_{i,j}.\]
	Note that we usually write $\overline{M}_{I_k,I_\ell}:=\overline{M}_{k,\ell}$, using subsets of $\{0,\ldots,n-1\}$ as indices. 
	For $k=0$ or $\ell=0$, the above definition gives an empty product, meaning that the leading row and column of $\overline{M}$ consist only of zeros. 
\end{definition}

\begin{lemma}\label{lem:properties_qm}
	Let $\mathcal{A}=(a_0,\ldots,a_{n-1})$ be a basis of $S$ and $\mathcal{B}$ a basis of $B$. 
	Set $M:=M_{\mathcal{A},\mathcal{B}}(S)$. 
	The following hold:
	\begin{enumerate}[label=(\alph*)]
		\item\label{lem:properties_qm1} For any $I,J\subset \{0,\ldots,n-1\}$, we have $\overline{M}_{I,J}=q_\mathcal{B}(\prod_{i\in I} a_i,\prod_{j\in J} a_j)$.
		\item\label{lem:properties_qm2} $M$ satisfies the following:
		\begin{enumerate}[label=(M\arabic*)]
			\item\label{M1} $M_{i,i}=M_{0,i}$ for all $i\in \{0,\ldots,n-1\}$ \emph{ or } $M_{i,i}=0$ for all $i\in \{0,\ldots,n-1\}$
			\item\label{M2} $M$ is symmetric
			\item\label{M3} $\overline{M}_{I,J}=\overline{M}_{K,L}\implies \exists X\!\subset\! \{0,\ldots,n-1\}:\  \overline{M}_{I,J}=\overline{M}_{I,X}=\overline{M}_{K,X}$\\
			for all $I,J,K,L\subset \{0,\ldots,n-1\}$
		\end{enumerate}
		\item\label{lem:properties_qm3} The basis $\mathcal{B}$ can be chosen such that $M$ also satisfies 
		\begin{enumerate}[label=(red)]
			\item\label{red} $M_{i,j}\le \Min\{2^k\mid M_{i',j'}<2^k \ \forall\  (i',j')<_{\text{lex}} (i,j)\})$ for all $i,j\in \{0,\ldots,n-1\}$.
		\end{enumerate}
		\ref{red} holds if and only if $M$ is reduced, meaning $M=M_\mathcal{A}(S)$.
	\end{enumerate}
\end{lemma}
\begin{proof}
	Write $m:=\Qdim(S)$.
	
	\ref{lem:properties_qm1}:
	Since $\phi_\mathcal{B}$ is an isomorphism, we have 
	\[q_\mathcal{B}\left(\prod_{i\in I} a_i, \prod_{j\in J} a_j\right)=\Ast_{(i,j)\in I\times J} q_\mathcal{B}(a_i,a_j)\]
	and since $q_\mathcal{B}(a_i,a_j)=M_{i,j}$, the above equals $\overline{M}_{I,J}$ by definition.
	
	\ref{lem:properties_qm2}:
	The properties \ref{M2} and \ref{M3} follow directly from \ref{QS2} and \ref{QS4}. 
	For \ref{M1}, it suffices to see that $q(a,a)=q(a,-1)$ for all $a\in G$ by \Cref{lem:basic_properties_qs} \ref{lem:basic_properties_qs_2} and this equals $0$ if $1=-1$ in $G$. 
	Since we have chosen $a_0=-1$ in the other case, \ref{M1} follows. 
	
	\ref{lem:properties_qm3}:
	Since the $q(a_i,a_j)$ generate $B$ by \Cref{lem:properties_2B} \ref{lem:properties_2B3} as $Q$ embeds naturally into $B$ by general assumption, we can choose $\mathcal{B}$ from these quaternions. 
	Take $k\in \N_0$. 
	If $q_0,\ldots,q_{k-1}$ are defined and $k<m$, set $q_k:=q(a_{i_k},a_{j_k})$ such that $(i_k,j_k)$ is lexicographically minimal such that $q(a_{i_k},a_{j_k})\not\in \spn(q_0,\ldots,q_{k-1})$. 
	This iteratively defines a basis $\mathcal{B}=(q_0,\ldots,q_{m-1})$ of $B$. 
	To show \ref{red}, take $i,j\in \{0,\ldots,n-1\}$. 
	If $(i,j)<_{\text{lex}}(i_0,j_0)$, then $q(a_i,a_j)=0$ and thus also $M_{i,j}=0$. 
	If $(i,j)>_\text{lex}(i_{m-1},j_{m-1})$, there is nothing to show because all entries of $M$ are in $\{0,\ldots,2^m-1\}=\spn(q_0,\ldots,q_{m-1})$.
	
	Otherwise, we can take $k\in \{0,\ldots,m-1\}$ minimal such that $(i_k,j_k)\ge_{\text{lex}} (i,j)$. 
	Then for $(i',j')<_{\text{lex}} (i,j)$, we have $q(a_{i'},a_{j'})\in \spn(q_0,\ldots,q_{k-1})$ by construction of $\mathcal B$ and thus 
	\[M_{i',j'}=q_\mathcal{B}(a_{i'},a_{j'})\in \spn(1,2,\ldots,2^{k-1})=\{0,\ldots,2^{k}-1\}.\]
	If $(i,j)=(i_k,j_k)$, then $M_{i,j}=2^k$. 
	Otherwise it follows as for $(i', j')$ above that $M_{i,j}<2^k$. 
	If $k\neq 0$, then $2^{k-1}=M_{i_{k-1},j_{k-1}}$ and $i_{k-1},j_{k-1}<(i,j)$, so $k$ is the minimum as described in \ref{red}.
	
	It follows from the construction that this choice of $\mathcal{B}$ gives the lexicographical minimal $M_{\mathcal{A},\mathcal{B}}(S)$ for fixed $\mathcal{A}$.
\end{proof}

\begin{definition}
	Take $n\in \N$ and $M\in \Mat_n(\N_0)$. 
	We call $M$ a \emph{prequaternionic matrix} if it satisfies \ref{M1}, \ref{M2}, \ref{M3}. 
	We say that $M$ is \emph{reduced} if $M$ also satisfies \ref{red}. 
	If $M$ is a prequaternionic matrix, we say that $M$ is of \emph{type 1} if the first condition of \ref{M1} is satisfied and of \emph{type 0} if the second condition of \ref{M1} is satisfied.
	
	Note that the conditions for type $0$ and type $1$ are not mutually exclusive.
	Whenever a matrix satisfies both conditions, we assign a type to it.
	As it will be possible to distinguish the types for the relevant cases later on (see \Cref{cor:nondegenerate_nqm_unique}), we do not put much focus on this.
\end{definition}

\Cref{lem:properties_qm} shows that in the notation used there, $M_{\mathcal{A},\mathcal{B}}(S)$ is a prequaternionic matrix and $M_\mathcal{A}(S)$ is a reduced prequaternionic matrix. 

Prequaternionic matrices already define unique $AP_3$ structures as we establish in the following. 
However, we will see later on that in order to properly reflect the properties of the 2B-groups in terms of $\ast_{\N_0}$, we need another technical detail to define \emph{quaternionic matrices}.
This will be pointed out in \Cref{rem:necessity_linkage_matrices} and \Cref{ex:linkage_matrices}.

\begin{lemma}\label{lem:properties_completion}
	Let $M\in \Mat_{n}(\N_0)$ be a prequaternionic matrix of type $\epsilon\in \{0,1\}$ with completion $\overline{M}$. 
	Then the following hold:
	\begin{enumerate}[label=(\alph*)]
		\item\label{lem:properties_completion1} 
		$\overline{M}$ is symmetric.
		\item\label{lem:properties_completion2} 
		If $\epsilon=0$, then $\overline{M}_{I,I}=0$ for all $I\subset \{0,\ldots,n-1\}$.
		
		If $\epsilon=1$, then $\overline{M}_{I,I}=\overline{M}_{\{0\},I}$ for all $I\subset \{0,\ldots,n-1\}$.
		\item\label{lem:properties_completion3} 
		$\overline{M}_{I,J}\ast_{\N_0}\overline{M}_{I,K}=\overline{M}_{I,J\triangle K}$ for all $I,J,K\subset \{0,\ldots,n-1\}$.
		\item\label{lem:properties_completion4} 
		$\overline{M}_{\{i\},\{j\}}=M_{i,j}$ for all $i,j\in \{0,\ldots,n-1\}$.
	\end{enumerate}
\end{lemma}
\begin{proof}
    By abuse of notation, we write $\ast$ for $\ast_{\N_0}$.
	\ref{lem:properties_completion1} is an immediate consequence of \ref{M2} and the definition of $\overline{M}$. 
	Now take $I,J,K\subset\{0,\ldots,n-1\}$. 
	We get \ref{lem:properties_completion2} by
	\begin{align*}
		\overline{M}_{I,I}
		&=&\Ast_{(i,j)\in I\times I} M_{i,j}
		&=&\Ast_{i\in I} M_{i,i}\ast \Ast_{(i,j)\in I\times I, i<j} \underbrace{(M_{k,l}\ast M_{l,k})}_{\tcap{\ref{M2}}{=}0}\\
		&=&\Ast_{i\in I} M_{i,i}
		&=&\begin{cases} \Ast_{i\in I} 0=0 &\text{if } \epsilon=0 \\ \Ast_{i\in I} M_{0,i}=\overline{M}_{\{0\},I} &\text{if } \epsilon=1.\end{cases}
	\end{align*}
	In a similar fashion, we get \ref{lem:properties_completion3} by computing
	\begin{align*}
		\overline{M}_{I,J}\ast\overline{M}_{I,K}
		&=&\Ast_{(i,j)\in I\times J} M_{i,j}\ast \Ast_{(i,k)\in I\times K} M_{i,k}\\
		&=& \Ast_{(i,j)\in I\times (J\triangle K)} M_{i,j}=\overline{M}_{I,J\triangle K}.
	\end{align*}
	Lastly, \ref{lem:properties_completion4} follows directly from the definition.
\end{proof}

The following shows that any prequaternionic matrix defines an $AP_3$ structure from which we can then recover the zeros in the completion of the matrix.

\begin{proposition}\label{prop:prequaternionic_matrices}
	Let $M\in \Mat_n(\N_0)$ be a prequaternionic matrix of type $\epsilon\in \{0,1\}$. 
	Let $G$ be a boolean group with neutral element $1$ and basis $\mathcal{A}:=(a_0, \ldots, a_{n-1})$. 
	If $\epsilon=0$, choose $-1:=1\in G$, else $-1:=a_0$. 
	Define $Q$ as the set of entries of $\overline{M}$, naturally containing 0. 
	Then
	\[q:G\times G\to Q,\quad \left(\prod_{i\in I} a_i, \prod_{j\in J} a_j \right)\mapsto \overline{M}_{I,J}\]
	defines an $AP_3$ structure $S(M)=(G, -1, q)$ with $s(S)=1$ if $\epsilon=0$ and $s(S)>1$ else. 
	For any basis $\mathcal{B}$ of the 2B-group of $S(M)$ and $I,J\subset\{0,\ldots,n-1\}$, it satisfies
	\[\overline{M}_{I,J}=0 \iff \overline{M_{\mathcal{A},\mathcal{B}}(S(M))}_{I,J}=0\]
\end{proposition}
\begin{proof}
	The property \ref{QS2} follows from  \Cref{lem:properties_completion} \ref{lem:properties_completion1}. 
	\ref{QS1} follows from \Cref{lem:properties_completion} \ref{lem:properties_completion2}, in the case $\epsilon=1$ also making use of part \ref{lem:properties_completion3}:
	\[q\left(\prod_{i\in I} a_i, -\prod_{i\in I} a_i\right) = \overline{M}_{I, I \triangle \{0\}} = \overline{M}_{I,I} \ast_{\N_0} \overline{M}_{I, \{0\}} = 0.\]
	\ref{QS3} is another consequence of \Cref{lem:properties_completion} \ref{lem:properties_completion3}, applying the fact that $(\N_0,\ast_{\N_0})$ is boolean. 
	Lastly, \ref{QS4} follows from \ref{M3}. 
	
	Recall that $S$ has finite order. 
	In light of \Cref{prop:AP_3_equiv_embedding}, we have to show that $Q$ naturally embeds into the 2B-group of $S(M)$ in order to prove that $S(M)$ is an $AP_3$ structure.
	We know that $Q=\{\overline{M}_{I,J}\mid I,J\subset \{0,\ldots,n-1\}\}$ generates a subgroup $B'$ of $(\N_0,\ast_{\N_0})$.
	It follows from \Cref{lem:properties_completion} \ref{lem:properties_completion3} and the definition of the symmetric difference that the linkage relations for $Q$ are satisfied in $B'$.
	Since $Q$ trivially embeds into $B'$, elements in $Q$ are equal if and only if they are equal modulo linkage relations.
	Thus, $Q$ naturally embeds into the 2B-group of $S(M)$.
	
	For the fixed basis $\mathcal{A}$, the zeros in the completion of a quaternionic matrix $M_{\mathcal{A},\mathcal{B}}(S(M))$ clearly do not depend on the choice of the basis $\mathcal{B}$ of the 2B-group.
	Since we can naturally identify quaternions and elements of the 2B-group, we get the asserted equivalence.
\end{proof}

\begin{definition}
	We say that a prequaternionic matrix $M$ is \emph{degenerate}, \emph{pythagorean}, etc. if the $AP_3$ structure $S(M)$ has the according property.
\end{definition}

\begin{corollary}\label{cor:isomorphic_qs_from_matrices}
	\begin{enumerate}[label=(\alph*)]
		\item\label{cor:isomorphic_qs_from_matrices1} 
		Let $M,M'\in \Mat_n(\N_0)$ be prequaternionic matrices of the same type such that $\overline{M}$ and $\overline{M'}$ have zeros in the same entries. 
		Then $S(M)\cong S(M')$ as quaternionic structures.
		\item\label{cor:isomorphic_qs_from_matrices2} 
		Let $\mathcal{A}$ be a basis of $S$ and $\mathcal{B}$ a basis of the 2B-group $B(S)$. 
		Then $S\cong S(M_{\mathcal{A},\mathcal{B}}(S))$.
	\end{enumerate}
\end{corollary}
\begin{proof}
	\ref{cor:isomorphic_qs_from_matrices1}:
	Write $\mathcal{A}$ and $\mathcal{A}'$ for the bases of $S$ and $S'$ as in \Cref{prop:prequaternionic_matrices}. 
	It is clear that the basis transformation $\mathcal{A}\mapsto \mathcal{A}'$ defines an isomorphism of quaternionic structures as it preserves the $(-1)$-element ($M,M'$ are of the same type) and the zero quaternions.
	
	\ref{cor:isomorphic_qs_from_matrices2}:
	We can define an isomorphism of the square class groups by mapping $\mathcal{A}$ to the basis defined in the construction in \Cref{prop:prequaternionic_matrices}.
	Since $S$ and $S(M_{\mathcal{A},\mathcal{B}}(S))$ have the same type, this isomorphism preserves $-1$.
	By \Cref{lem:properties_qm} \ref{lem:properties_qm1} and by the construction of \Cref{prop:prequaternionic_matrices}, this isomorphism also preserves zero quaternions.
\end{proof}

At this point, we are at a good position to compute the normal quaternionic matrices of local type structures.
We leave some linear algebra details (which appear similarly in \cite[Chapter 5, Section 2 and 3]{Marshall_AbstractWittRings}) to the reader.

\begin{proposition}\label{prop:nqm_local_types}
	For $k\in \N$, $k\ge 2$, we have $M(\LL_{2k,0})=E_{2k}$,
	\[
	M(\LL_{2k,1})=\begin{pmatrix}
	0&0&\ldots&0&1\\
	0&&&&0\\
	\vdots&&E_{2k-2}&&\vdots\\
	0&&&&0\\
	1&0&\ldots&0&1
	\end{pmatrix},\ 
	M(\LL_{2k-1})=\begin{pmatrix}
		1&0&\ldots&0\\
		0&&&\\
		\vdots&&E_{2k-2}&\\
		0&&&
	\end{pmatrix}\]
	where $E_n\in \Mat_n(\N_0)$ is the matrix with entry $1$ on the anti-diagonal and entry $0$ else.
\end{proposition}
\begin{proof}
    It was noted in \Cref{rem:local_types} that any nondegenerate symmetric bilinear map $G\times G\to \F_2$ uniquely defines a local type structure with square class group $G$.
    In \Cref{ex:qm_basic_examples} we saw that the quaternionic matrices are Gram matrices of these maps as we only have $2$ quaternions.
    Clearly, the matrices in the assertion above are invertible and symmetric, so they are Gram matrices of nondegenerate symmetric bilinear maps for some choice of basis of $G$. 
    The only limitation for quaternionic matrices in contrast to these Gram matrices is that we fix $-1$ as the first basis element of the square class group whenever $-1\neq 1$.
    Hence, for $\LL_{2k,0}$ with square class group $G$, the only local type structure with $-1=1$, the quaternionic matrices are exactly the Gram matrices of nondegenerate symmetric bilinear maps $q:G\times G\to \F_2$ such that $q(a,a)=0$ for all $a\in G$ since $-1=1$ implies that $q(a,a)=q(a,-1)=0$ by \Cref{lem:basic_properties_qs}.
    These are exactly the Gram matrices with zeros on the diagonal by \Cref{lem:properties_completion} \ref{lem:properties_completion2}.
    It is now an easy linear algebra exercise to check that $E_{2k}$ is the minimal symmetric and invertible $2k\times 2k$ matrix with entries from $\{0,1\}$, completing the proof for $\LL_{2k,0}$.
    
    In the other cases, we have $1\neq -1$, so we have to choose a basis $(-1,a_1,\ldots)$ of $S$.
    Writing $\{0,1\}$ for the set of quaternions, we do not have to distinguish between different choices for the basis of the 2B-group as seen in \Cref{ex:qm_basic_examples}.
    
    For $\LL_{2k,1}=(G,-1,q)$, take $a_{2k-1}$ such that $q(-1,a_{2k-1})=1$.
    Since $q(-1,-1)=0$, the elements $-1$ and $a_{2k-1}$ span a $2$-dimensional subspace $H$ of $G$.
    By definition, the orthogonal complement $H^\perp$ of $H$ is in $V(-1)$, so $q(a,a)=q(a,-1)=0$ for all $a\in H^\perp$ by \Cref{lem:basic_properties_qs} \ref{lem:basic_properties_qs_2}.
    One readily checks that the restriction of $q$ to $H^\perp\times H^\perp$ is nondegenerate.
    Hence, we can choose a basis $(a_1,\ldots,a_{2k-2})$ of $H$ such that we get $E_{2k-2}$ as the Gram matrix of the restriction of $q$ to $H^\perp\times H^\perp$ in the same way as for $\LL_{2k,0}$.
    This shows that the asserted matrix is indeed a quaternionic matrix of $\LL_{2k,1}$.
    Since $\LL_{2k,1}$ is nondegenerate and thus its quaternionic matrices are invertible, it quickly follows that the given matrix is the minimal choice and thus the normal quaternionic matrix of $\LL_{2k,1}$.
    
    For $\LL_{2k-1}$ we can proceed very similarly. 
    Now, since $q(-1,-1)=1$, we can split off $\spn(-1)$ orthogonally.
\end{proof}

\begin{remark}\label{rem:necessity_linkage_matrices}
	At this point, we could simply assign a unique prequaternionic matrix to any $AP_3$ structure $S$ by choosing the lexicographically minimal one giving rise to a structure isomorphic to $S$ by the construction of \Cref{prop:prequaternionic_matrices}. 
	This would give a bijective mapping from the isomorphism types of $AP_3$ structures of order $n$ to a certain subset of $\Mat_n(\N_0)$. 
	As it turns out, this does not properly display the structure of the 2B-group of $S$: In general, even if additionally requiring \ref{red}, there can be prequaternionic matrices of $S$ that cannot be written as the (reduced) quaternionic matrix of $S$ via a basis of $S$. 
	This is the case if there are relations of $\ast_{\N_0}$ appearing in the prequaternionic matrix that are not induced by linkage relations. 
	We illustrate this in the following \Cref{ex:linkage_matrices}.
\end{remark}

\begin{example}\label{ex:linkage_matrices}
	To make notation clearer in this example, we use some coloring to point out which elements belong to which 'part' of the following structure:
	We consider the elementary type structure 
	\[(\LL_1\times (\textcolor{green}{\LL_1}\times \textcolor{blue}{\LL_1})\textcolor{red}{\Delta})\textcolor{orange}{\Delta}=:S=:(G,-1,q)\]
	with square class group $G=\{1,a_1\}\times \textcolor{green}{\{1,b_0\}}\times \textcolor{blue}{\{1,a_2\}}\times \textcolor{red}{\{1,a_3\}}\times \textcolor{orange}{\{1,a_4\}}$ according to construction and coloring. 
	Then $a_0:=-1=a_1\textcolor{green}{b_0}\textcolor{blue}{a_2}$. 
	We compute $M:=M_\mathcal{A}(S)$ for the basis $\mathcal{A}=(a_0,\ldots,a_4)$.
	We have $q(a_0,a)\not=0$ for all $a\in G\setminus\{1\}$ ($S$ is pythagorean), so all entries of the first row of $M$ must be independent, meaning that by \ref{red} along with \Cref{lem:rigidity_ets} \ref{lem:rigidity_ets1}, we have 
	the first row $(1\ 2\ 4\ 8\ 16)$ (a more general result on the structure of the first row can be found in \Cref{first row nqm} later on). 
	Furthermore, by the properties of the direct product $\LL_1\times (\textcolor{green}{\LL_1}\times\textcolor{blue}{\LL_1})\textcolor{red}{\Delta}$, we get $q(a_1,\textcolor{blue}{a_2})=0=q(a_1,\textcolor{red}{a_3})$. 
	Applying \ref{M1} and \ref{M2} there are $4$ more values $x_1,\ldots,x_4$ to compute:
	\[M=\begin{pmatrix} 1&2&4&8&16\\ 2&2&0&0&x_1\\4&0&4&x_2&x_3\\8&0&x_2&8&x_4\\16&x_1&x_3&x_4&16\end{pmatrix}\]
	Iteratively applying \Cref{prop:2B_of_product_and_extension} 
	for the 2B-group, it follows that 
	\[\Qdim(S)=1+((\textcolor{green}{1}+\textcolor{blue}{1})+\textcolor{red}{2})+\textcolor{orange}{4}=9.\]
	Hence, we must have $x_1=2^5$, $x_2=2^6$, $x_3=2^7$, $x_4=2^8$ and thus
	\[M=\begin{pmatrix} 1&2&4&8&16\\ 2&2&0&0&32\\4&0&4&64&128\\8&0&64&8&256\\16&32&128&256&16\end{pmatrix}.\]
	It is clear from the uniqueness of the first row and the position of the zero entries that this is the lexicographically minimal quaternionic matrix of $S$ via a basis of $S$. 
	However, it can be verified with an elementary yet lengthy computation that the matrix 
	\[M':=\begin{pmatrix} 1&2&4&8&16\\ 2&2&0&0&32\\4&0&4&33&64\\8&0&33&8&128\\16&32&64&128&16\end{pmatrix}\]
	also satisfies \ref{red}, $M'<_{\text{lex}}M$ and we have $\overline{M}_{I,J}=0\iff \overline{M'}_{I,J}=0$ for all $I,J\subset \{0,\ldots, 4\}$. 
	Hence, by \Cref{cor:isomorphic_qs_from_matrices}, we have $S\cong S(M')$. 
	
	Here, the 'problem' lies in the fact that $33=32\ast 1$ is no relation of the 2B-group of $S$ where the according values are independent, as $32\ast 1=33$ is not induced by linkage relations.
\end{example}

As illustrated by the previous example, it is also necessary to include the linkage relations in a sensible definition of quaternionic matrices.

\begin{lemma}\label{lem:pqm_dimension_Brauer}
	Let $M\in \Mat_n(\N_0)$ be a prequaternionic matrix.
	Then $\Qdim(S(M))\ge \dim(\spn_{\ast_{\N_0}}(M_{i,j}\mid 0\le i,j\le n-1))$. 
\end{lemma}
\begin{proof}
	Write $S(M)=(G_M,-1,q_M)$ with basis $(a_0,\ldots,a_{n-1})$ as in \Cref{prop:prequaternionic_matrices}.
	Also write $B$ for the 2B-group of $S(M)$ and $B'$ for the subgroup of $(\N_0,\ast_{\N_0})$ generated by the entries of $M$.
	
	$B$ is generated by the $q_M(a_i,a_j)=M_{i,j}$ with linkage relations \ref{eq:Linkage_Relation} as its sole relations.
	As the entries of $M$ resp. $\overline{M}$ satisfy these relations by \ref{M3} regarding the quaternionic product on $\N_0$, we get a well-defined surjective homomorphism $B\to B'$ by mapping $q_M(a_i,a_j)$ to $M_{i,j}$ for all $i,j$.
	The assertion follows since $\Qdim(S(M))=\dim(B)$ by definition.
\end{proof}

\begin{definition}
	Let $M\in \Mat_n(\N_0)$ be a (reduced) prequaternionic matrix of type $\epsilon\in \{0,1\}$.
	We say that $M$ is a \emph{(reduced) quaternionic matrix of type $\epsilon$} if we have equality in \Cref{lem:pqm_dimension_Brauer}.
\end{definition}

\begin{remark}\label{rem:quaternionic_matrices_equivalences}
	By \Cref{prop:prequaternionic_matrices}, a prequaternionic matrix $M\in \Mat_n(\N_0)$ 
	together with a boolean group $G$ of dimension $n$
	defines an $AP_3$ structure $S(M)=(G,-1,q)$ 
	with quaternions $M_{i,j}\in \N_0$.
	Modulo linkage relations (which are also satisfied in regard to $\ast_{\N_0}$), the quaternions generate the 2B-group $B(S(M))$.
	As seen in the proof of \Cref{lem:pqm_dimension_Brauer}, we get a surjective homomorphism from $B':=\spn_{\ast_{\N_0}}(M_{i,j}\mid 0\le i,j\le n)$ to $B(S(M))$.
	By definition, this is an isomorphism if and only if $M$ is a quaternionic matrix. 
	In this case, we can thus identify $B(S(M))$ with $B'$.
	This 'justifies' calling $\ast_{\N_0}$ the \emph{quaternionic product} whenever we have a quaternionic matrix at hand. 
\end{remark}

\begin{lemma}\label{brauer group in matrix}
	If $M\in \Mat_n(\N_0)$ is a reduced quaternionic matrix and $S(M)$ is defined as in \Cref{prop:prequaternionic_matrices}, then the different $2$-powers appearing as entries of $M$ form a basis of the 2B-group of $S(M)$. 
	In particular, if $M$ has a nonzero entry, then
	\[\Qdim(S(M))=\Max\{k+1\mid \exists\ i,j\in \{0,\ldots,n-1\}:\ M_{i,j}=2^k\}.\]
\end{lemma}
\begin{proof}
	By \ref{red}, there is $m\in \N$ such that $\mathcal{B}:=(1,2,\ldots,2^{m-1})$ consists of entries of $M$ and all entries of $\overline{M}$ are in $\{0,\ldots,2^m-1\}$. 
	Hence, $\mathcal{B}$ is a basis of $\spn_{\ast_{\N_0}}(M_{i,j}\mid 0\le i,j\le n-1)$ and thus of the 2B-group of $S(M)$ by \Cref{rem:quaternionic_matrices_equivalences}.
\end{proof}

\begin{definition}
	We call two quaternionic matrices of the same type \emph{isomorphic} if their quaternionic structures constructed as in \Cref{prop:prequaternionic_matrices} are isomorphic. 
	Let $M$ be a quaternionic matrix of type $\epsilon\in \{0,1\}$. We say that $M$ is a \emph{normal quaternionic matrix} if for every quaternionic matrix $M'$ isomorphic to $M$, we have $M\le_{\text{lex}} M'$.
\end{definition}

By this definition and \Cref{brauer group in matrix}, we immediately get the following:

\begin{corollary}\label{cor:nqm_unique}
	Let $\epsilon\in \{0,1\}$. For any $n\in \N$, we have a bijection between the set of normal quaternionic matrices of type $\epsilon$ in $\Mat_n(\N_0)$ and the isomorphism types of $AP_3$ structures $S$ of order $n$ with $-1=1$ (if $\epsilon=0$) or $-1\not=1$ (if $\epsilon=1$), mapping $M$ to $S(M)$. 
\end{corollary}

Restricting ourselves to nondegenerate $AP_3$ structures, we can handle both types simultaneously by \ref{M1}.

\begin{corollary}\label{cor:nondegenerate_nqm_unique}
	For any $n\in \N$, we have a bijection between the set of nondegenerate normal quaternionic matrices in $\Mat_n(\N_0)$ and the isomorphism types of nondegenerate $AP_3$ structures of order $n$.
\end{corollary}
\begin{proof}
	Clearly, if $M$ is a nondegenerate quaternionic matrix, then the first row of $M$ has a nonzero entry. 
	Thus, $M$ cannot satisfy both conditions in \ref{M1} simultaneously, and the type of $M$ is uniquely determined.
\end{proof}

\begin{definition}
	Let $M\in \Mat_n(\N_0)$ be a quaternionic matrix of type $\epsilon\in \{0,1\}$. 
	We say that $M'\in \Mat_n(\N_0)$ is a \emph{simple transformation} of $M$ if there are $i,j\in \{0,\ldots,n-1\}$, $i\not=j$ such that
	\begin{itemize}
		\item We have $\epsilon=0$ or $i,j\not=0$ and:\\
		$M'$ is obtained from $M$ by swapping the $i$-th row of $M$ with the $j$-th row and the $i$-th column with the $j$-th column
		\item We have $\epsilon=0$ or $i\not=0$ and:\\
		$M'_{i,k}=M_{i,k}\ast M_{j,k}=M'_{k,i}$ for all $k\not=i$, $M'_{i,i}=M_{i,i}\ast M_{j,j}$ and $M'_{k,\ell}=M_{k,\ell}$ for all $k,\ell\in \{0,\ldots,n-1\}\setminus\{i\}$
	\end{itemize}
	The simple transformations above will briefly be denoted as $(i \leftrightarrow j)$ and $(i\mapsto i\ast j)$ respectively.
	We say that $M'$ is a \emph{transformation} of $M$ if it can be obtained from $M$ using a finite number of simple transformations.
\end{definition}

The simple transformations display basis transformations of $AP_3$ structures in the following way: Take an $AP_3$ structure $S=(G,-1,q)$ with basis $\mathcal{A}=(a_0,\ldots,a_{n-1})$ and basis $\mathcal{B}$ of the 2B-group. 
As $G$ is an $\F_2$ vector space, we cannot scale $\mathcal{A}$ nontrivially, so if $\mathcal{A}'$ is another basis of $S$, we can obtain it from $\mathcal{A}$ by an iteration of swapping elements $a_i \leftrightarrow a_j$ or replacing $a_i$ by $a_ia_j$. 
The following lemma shows that the matrix transformations defined above translate the effects of these basis transformations to the quaternionic matrices, going from $M_{\mathcal{A},\mathcal{B}}$ to $M_{\mathcal{A}',\mathcal{B}}$ by an iteration of simple transformations. 
For the canonical choice of $\mathcal{B}$ (cf. \Cref{lem:properties_qm} \ref{lem:properties_qm3}) in this setting, we have to consider reduced quaternionic matrices:

\begin{lemma}\label{lem:reduced_qm}
	Let $M, M'\in \Mat_n(\N_0)$ be quaternionic matrices of the same type. 
	The following hold:
	\begin{enumerate}[label=(\alph*)]
		\item\label{lem:reduced_qm1} 
		There is a unique reduced quaternionic matrix $M_\text{red}$ such that:
		\[\forall I,J\subset \{0,\ldots,n-1\}:\ \overline{M}_{I,J}=0 \iff \overline{M_\text{red}}_{I,J}=0\]
		\item\label{lem:reduced_qm2} If $M$ is a normal quaternionic matrix, then $M$ is a reduced quaternionic matrix.
		\item\label{lem:reduced_qm3} 
		$M$ and $M'$ are isomorphic if and only if there is a transformation $M''$ of $M$ such that $M''_\text{red}=M'_\text{red}$.
	\end{enumerate}
\end{lemma}
\begin{proof}
	The entries of $M$ span the 2B-group $B$ of $S(M)$ by definition. 
	Set $m:=\Qdim(S(M))=\dim(B)$.
	Then we can write $B=\{0,\ldots,2^{m-1}-1\}$ as a subgroup of $(\N_0,\ast_{\N_0})$ by \Cref{rem:quaternionic_matrices_equivalences}.
	
	\ref{lem:reduced_qm1}:
	We can choose a basis $\mathcal{B}:=(M_{i_0,j_0},\ldots,M_{i_{m-1},j_{m-1}})$ of $B$ in the same way as in \Cref{lem:properties_qm} \ref{lem:properties_qm3}. 
	Then, the group isomorphism 
	\[B\to \{0,\ldots,2^m-1\},\ \mathcal{B}\mapsto (1,2,\ldots,2^{m-1})\]
	applied on the entries of $M$ defines a matrix $M_\text{red}$. 
	For $I,J,K,L\subset \{0,\ldots,n-1\}$, we have 
	\[\overline{M}_{I,J}=\overline{M}_{K,L}\iff \overline{M_\text{red}}_{I,J}=\overline{M_\text{red}}_{K,L}.\]
	This immediately shows that $M_\text{red}$ is a prequaternionic matrix. 
	The axiom \ref{red} follows as in the proof of \Cref{lem:properties_qm} \ref{lem:properties_qm3}. 
	Since the zero entries in the completions are the same, we have $S(M)\cong S(M_\text{red})$ and in particular $\Qdim(S(M_\text{red})=m$ which then shows that $M_\text{red}$ is a reduced quaternionic matrix. 
	The uniqueness follows in the same way as in \Cref{lem:properties_qm} \ref{lem:properties_qm3} again.
	
	\ref{lem:reduced_qm2}:
	This follows immediately from \ref{lem:reduced_qm1} since by \Cref{cor:isomorphic_qs_from_matrices} \ref{cor:isomorphic_qs_from_matrices1}, we get $S(M)\cong S(M_\text{red})$ and clearly $M_\text{red}\le_\text{lex} M$ (cf. \Cref{lem:properties_qm} \ref{lem:properties_qm3}).
	
	\ref{lem:reduced_qm3}:
	First, we point out that applying simple transformations $(i\leftrightarrow j)$ or $(i\mapsto i\ast j)$ just means swapping the $i$-th and $j$-th basis element in the induced basis of $S(M)$ or replacing $a_i$ by $a_ia_j$. 
	Both are basis transformations of $S(M)$ as by the choice of $i,j$, the element $-1$ is fixed. 
	In particular, simple transformations and by iteration also general transformations of $M$ are again quaternionic matrices of the same type that are isomorphic to $M$. 
	By \ref{lem:reduced_qm1}, reduction of quaternionic matrices also does not change the isomorphism type, so the 'if' part of the statement holds.
	
	Conversely, assume $S:=S(M)=(G,-1,q)$, $S':=S(M')=(G',-1',q')$ are isomorphic with bases $\mathcal{A}=(a_0,\ldots,a_{n-1})$ and $\mathcal{A}'=(a_0',\ldots,a_{n-1}')$ as in \Cref{prop:prequaternionic_matrices}. 
	Consider the isomorphism $\phi:G\to G'$ of $S$ and $S'$ and the preimages $a_i'':=\phi^{-1}(a_i')$ for $i=0,\ldots,n-1$. They form a basis $\mathcal{A}'':=(a_0'',\ldots,a_{n-1}'')$ of $S$.
	Hence, we can obtain $\mathcal{A}''$ from $\mathcal{A}$ by a finite number of elementary basis transformations. 
	Write $M:=M_{\mathcal{A},\mathcal{B}}(S)$ for a basis $\mathcal{B}$ of the 2B-group of $S$. 
	Then $M'':=M_{\mathcal{A}'',\mathcal{B}}$ is a transformation of $M$. 
	Since $\phi$ is an isomorphism mapping $\mathcal{A}''$ to $\mathcal{A}'$, the completions $\overline{M''}$ and $\overline{M'}$ have zeros at the same entries by \Cref{prop:prequaternionic_matrices}.
	By \ref{lem:reduced_qm1}, this shows that $M''_\text{red}=M'_\text{red}$.
\end{proof}

The previous statement points out that in order to compute the normal quaternionic matrix of a given quaternionic matrix $M$, we have to compute the lexicographic minimum of all reduced (in the sense of \Cref{lem:reduced_qm} \ref{lem:reduced_qm1}) transformations of $M$.

\section{Normal Quaternionic Matrices}\label{sec_NQM}
The main result of this section will be determining what the first two rows of normal quaternionic matrices look like. 
As already seen (for example in \Cref{prop:nqm_local_types}), choosing $-1$ as the first element of any basis of an $AP_3$ structure with $-1\neq 1$ clearly impacts the structure of the first row.

\begin{proposition}\label{first row nqm}
Let $M=(M_{i,j})_{0\le i,j\le n-1}$ be a normal quaternionic matrix of type $\epsilon\in \{0,1\}$ with associated $AP_3$ structure $S:=S(M):=(G,-1,q)$. 
Set $x:=\dim(V(-1))$ if $-1\not=1\in G$ and $x:=\Max\{\dim(V(a))\mid a\in G\setminus \{1\}\}$ otherwise. 
We write $s:=s(S)$. 
Then
\[(M_{0,0},M_{0,1},\ldots,M_{0,n-1})=\begin{cases} (0,\ldots,0,1,2,4,\ldots,2^{n-x-1}), &\text{ if } s\le 2 \\ (1,0,\ldots,0,2,4,\ldots,2^{n-x-1}), &\text{ else}.\end{cases}\]
\end{proposition}
\begin{proof}
Let $\mathcal{A}=(a_0,\ldots,a_{n-1})$ be a basis of $S$ such that $M_\mathcal{A}(S)=M$. 
First note that $m_{0,0}=0$ if $s\le 2$ (which means that $q(a_0,a_0)=0$) and $m_{0,0}=1$ else. 
By definition of the normal quaternionic matrix, $\mathcal{A}$ is chosen such that we get the maximum number of 'leading zeros', i.e. $x'\in\N_0$ maximal such that $q(a_0,a_i)=0$ for all 
\[i\in \begin{cases} \{0,\ldots,x'-1\} & \text{if } s\le 2\\
\{1,\ldots,x'\} & \text{if } s>2.\end{cases}\] 
This means that $a_0,\ldots,a_{x'-1}$ resp. $a_1,\ldots,a_{x'}$ form a basis of $V(a_0)$. 
By definition, we get $x'=x$. 
Define 
\[J=\begin{cases} \{x,x+1,\ldots,n-1\} &\text{if } s\le 2\\ 
\{0,x+1,x+2,\ldots,n-1\} & \text{if } s>2.\end{cases}\] 
The values $(q(a_0,a_j)\mid j\in J)$ must be independent for else there would be a subset $\emptyset\not= I\subset J$ such that 
\[0=\Ast_{i\in I} q(a_0,a_i)=q(a_0,\prod_{i\in I} a_i)\] 
and thus $\prod_{i\in I}a_i\in V(a_0)$, a contradiction. 
Hence, the first row of $M$ is of the described form.
\end{proof}

As an immediate consequence of the above proof, we can specify exactly which bases of $S$ produce a minimal first row:

\begin{corollary}\label{minimal_first_row}
Let $M$ be a normal quaternionic matrix of order $n$ and type $\epsilon\in \{0,1\}$ with associated $AP_3$ structure $S:=S(M)$. Define $x$ as in \Cref{first row nqm}. Let $\mathcal{A}=(a_0,\ldots,a_{n-1})$ be any basis of $S$. The first row of $M_\mathcal{A}(S)$ is identical to that of $M$ if and only if
\begin{itemize}
    \item $s(S)\le 2$ and $V(a_0)=\spn(a_0,\ldots,a_{x-1})$ \textbf{or}
    \item $s(S)>2$ and $V(a_0)=\spn(a_1,\ldots,a_x)$
\end{itemize}
\end{corollary}

\begin{example}\label{example first row} \Cref{first row nqm} directly produces a list of 'potential' normal quaternionic matrices of order $1$ and $2$ and the different matrices are clearly not isomorphic as the dimension of the value sets $V(a)$ is invariant under isomorphism. 
For type $0$, we get exactly the matrices
\[\begin{pmatrix} 0\end{pmatrix},\quad \begin{pmatrix} 0&0\\0&0\end{pmatrix},\quad \begin{pmatrix} 0&1\\1&0\end{pmatrix}.\] 
which are the normal quaternionic matrices of $\LL_{1,0}$, $\LL_{1,0}^2$ and $\LL_{1,0}\Delta$.\\
For type 1, we get the matrices
\[\begin{pmatrix} 0 \end{pmatrix},\quad\begin{pmatrix} 1\end{pmatrix},\quad \begin{pmatrix} 0&0\\0&0\end{pmatrix},\quad \begin{pmatrix} 0&1\\1&1\end{pmatrix},\quad \begin{pmatrix} 1&0\\0&0\end{pmatrix},\quad \begin{pmatrix} 1&2\\2&2\end{pmatrix}.\]
One readily checks that these are the normal quaternionic matrices of (in the same order) 
\[\LL_{1,1},\quad \LL_1,\quad \LL_{1,1}^2,\quad \LL_{1,1}\Delta,\quad \LL_{1,1}\times \LL_{1},\quad \LL_{1}^2(\cong \LL_1\Delta).\] 
Comparing with the number of elementary type structures of order $1$ and $2$ given in \Cref{tbl_e_n}, this verifies the 3ETC for structures of order $1$ and $2$.
\end{example}

To also characterize the structure of the second row, we introduce a few definitions and a somewhat technical lemma that will also be of use for applying our algorithm in the final section. 
Recall that we bestow the lexicographic order on matrices row by row. 
We are often interested in all entries of a normal quaternionic matrix up to (with respect to that order) a certain row/column index.
In the computations later on, a significant amount of time could be saved by identifying situations in which we can apply a basis transformation to replace an entry $a$ by $a\ast 2^k$ for some $k\in \N_0$ without changing any previous entries. 
In these cases, we can immediately conclude that $2^k$ cannot appear in the binary expansion of $a$.
This leads to the following definition:

\begin{definition}
Given $x, y\in \N$ we say that \emph{$y$ is a factor of $x$} if $x*y<x$.\\ 
Let $M=(M_{i,j})_{0\le i,j\le n-1}$ be a quaternionic matrix of order $n\in \N$.
Take $i,j,k,\ell\in \{0,\ldots,n-1\}$ with $(i,j)\le_\text{lex}(k,\ell)$. 
We say that $M_{i,j}$ is a \emph{unique value (of $M$) up to $(k,\ell)$} if $M_{i,j}$ is a power of 2 that is no factor of any of the elements of $\{M_{i',j'}\mid i'<j',\ (i',j')\not=(i,j) \text{ and } (i',j')<_\text{lex} (k,\ell)\}$ 
\end{definition}

\begin{lemma}\label{exclude factors}
Let $M$ be a normal quaternionic matrix of order $n$ and take $i,j\in \{0,\ldots,n-1\}$ with $i<j$. 
The following hold:
\begin{enumerate}[label=(\alph*)]
\item\label{exclude factors 1}
If there is $k<i$ such that $M_{k,j}\not=0$ and for all $\ell<j$ with $\ell\not=i$, we have $M_{k, \ell}=0$ or $M_{i,\ell}$ is unique up to $(i,j)$, then $M_{k,j}$ is no factor of $M_{i,j}$.
\item\label{exclude factors 2} 
If there is $\ell <j$ such that $M_{i,\ell}\not=0$ and for all $k<i$, we have $M_{k,\ell}=0$ or $M_{k,j}$ is unique up to $(i,j)$, then $M_{i,\ell}$ is no factor of $M_{i,j}$.
\item\label{exclude factors 4} 
If $j+1>i$ and $M_{k,j}=0$ for all $k<i$, then $M_{i,j}\ge M_{i,j-1}$.
\item\label{exclude factors 3} 
If there is $k<i$ such that for all $\ell\ge j$ the entry $M_{k,\ell}$ is not in $\spn(M_{k,0},\ldots,M_{k,\ell-1})$, the following hold:
\begin{enumerate}[label=(\roman*)]
\item\label{exclude factors 31}
If $M_{i,j}\in \spn(M_{k,j'}\mid j'<j)$, then $M_{i,j}\in\spn(M_{i,j'}\mid j'<j)$
\item\label{exclude factors 32}
If $M_{k,j}\in \spn(M_{i,j'}\mid j'<j)$, then $M_{i,j}\in \spn(M_{i,j'}\mid j'<j)$. 
\end{enumerate}
\end{enumerate}
\end{lemma}
\begin{proof}
\ref{exclude factors 1}:
We apply the transformation $(i\to k\ast i)$ to obtain the quaternionic matrix $M'$ of $S$ and set $M'':=M'_\text{red}$.
We clearly have $M'_{i',j'}=M_{i',j'}$ if $i\not\in \{i',j'\}$. 
For $\ell\not=i$, $\ell <j$, we get 
\[M'_{i,\ell}=M_{i,\ell}\ast M_{k,\ell}.\]
If $M_{k,\ell}=0$, then $M'_{i,\ell}=M_{i,\ell}$. 
Otherwise, $M_{i,\ell}$ is unique up to $(i,j)$ and thus, we still get 
\[M'_{i,\ell} \not \in \spn(M'_{i', j'} \mid (i',j')<_\text{lex} (i,\ell)\}.\] 
Hence, we have $M''_{i,\ell}=M_{i,\ell}$ in either case. 

Although we transformed the (implicitly employed) quaternionic basis by considering the reduced matrix $M''=M'_\text{red}$, we still get $M''_{i',j'}=M_{i',j'}$ for all $(i',j')<_\text{lex} (i,j)$ since the exchanged basis quaternions are no factors of any other entries of $M$. 
As $M$ is normal, it follows that $M''_{i,j}\ge M_{i,j}$ and since $M'_{i,j}=M_{i,j}\ast M_{k,j}$, we have that $M_{k,j}$ is no factor of $M_{i,j}$.

\ref{exclude factors 2}:
This follows in the same way as \ref{exclude factors 1}, now considering the transformation $(j\to j\ast \ell)$.

\ref{exclude factors 4}:
This follows similarly to \ref{exclude factors 1} and \ref{exclude factors 2}, without the need to consider unique values,  by the transformation $(j-1\leftrightarrow j)$.

\ref{exclude factors 3}:
It follows immediately from the given condition that $\overline{M}_{2^k, x}\not=\overline{M}_{2^k,y}$ for $x< 2^j\le y$. 
By \ref{M3} and the definition of the completion, if
\[M_{i,j}\in \spn(M_{k,j'}\mid j'<j)=\{\overline{M}_{2^k,x}\mid x\in \{0,\ldots,2^j-1\}\}\] 
then we already have $M_{i,j}\in \spn(M_{i,j'}\mid j'<j)$. 
This shows \ref{exclude factors 31}.

If now $M_{k,j}\in \spn(M_{i,j'}\mid j'<j)$, we can apply \ref{M3} again, making use of the fact that $\overline{M}_{2^k,x}\not=M_{k,j}$ for $x<2^j$ and also for $x\ge 2^{j+1}$ since $M_{k,j+1},\ldots,M_{k,n-1}$ are all independent from the previous entries of the row. 
Hence, there must be $y$ with $2^j\le y<2^{j+1}$ such that $\overline{M}_{2^i,y}=M_{k,j}$ which proves \ref{exclude factors 32}.
\end{proof}

\begin{proposition}\label{second row nqm}
Let $n\in \N$, $n\ge 3$ and $M\in \Mat_n(\N_0)$ be a normal quaternionic matrix with level $s:=s(S(M))$.
Then there are unique $d,r,t,u,v\in \N_0$ satisfying $t\le u\le d\le v$ such that the first two rows of $M$ can be written as
\[\left(\begin{array}{ccccccc} \begin{matrix} M_{0,0} &M_{0,1}\\ M_{1,0}& M_{1,1} \end{matrix} & B_1 & B_2 & B_3 & B_4 & B_5 & B_6\end{array}\right)\]
with matrices $B_i\in \Mat_{2\times k_i}(\N_0)$ and 
\[(k_1,\ldots,k_6)=(n-r-v-2, r, v-d, t, u-t, d-u),\] 
the case $k_i=0$ being allowed and interpreted as a non existent block. 
If $s\in \{1,2\}$ and $M_{0,1}=0$, these matrices can be written as the following:
\[B_1=\begin{pmatrix} 0&\ldots &0\\0&\ldots &0\end{pmatrix},\ 
B_2=\begin{pmatrix} 0&\ldots&0\\ 2^0&\ldots&2^{r-1}\end{pmatrix},\  
B_3=\begin{pmatrix} 0&\ldots &0\\ 2^d&\ldots&2^{v-1}\end{pmatrix}\]
\[B_4=\begin{pmatrix} 2^0&\ldots &2^{t-1}\\0&\ldots&0\end{pmatrix},\  
B_5=\begin{pmatrix} 2^t&\ldots&2^{u-1}\\ 2^t&\ldots&2^{u-1}\end{pmatrix},\ 
B_6=\begin{pmatrix} 2^u&\ldots& 2^{d-1}\\ 2^{v}&\ldots&2^{v+d-u-1}\end{pmatrix}\]
If $s>2$ and $M_{0,1}=0$, the entries of $B_3,\ldots,B_6$ are multiplied by $2$ and the entries of $B_2$ are multiplied by a factor $f\in \{1,2\}$.

If $M_{0,1}\not=0$, we have $k_1=k_2=k_3=0$ while the entries of $B_4, B_5$ and $B_6$ are multiplied by $2M_{0,1}$.
\end{proposition}
\begin{proof}
By \Cref{first row nqm}, the first row of $M$ has the asserted shape, setting $d:=n-x$ if $s\le 2$ and $d:=n-x-1$ otherwise. 

We write $S:=S(M)=(G,-1,q)$ and fix a basis $\mathcal{A}=(a_0,\ldots,a_{n-1})$ of $S$ such that $M_\mathcal{A}(S)=M$. 
To keep notation brief, we usually only write $q$ for the quaternionic map instead of $q_\mathcal{B}$ for a basis $\mathcal{B}$ of the full 2B-group $B$ group whenever no confusion is imminent. 


Let now $k \in \{2, \ldots, n-1\}$ be such that $M_{1, 0}, \ldots, M_{1, {k-1}}$ are as claimed.
Set 
\[\alpha:=\Min\{2^x\mid 2^x> M_{i,j} \text{ for all } (i,j)<_\text{lex} (1,k)\}.\] 
We have $M_{1,k}=\alpha$ if and only if $M_{1,k}$ is independent from the previous entries of $M$. 
Otherwise, $M_{1,k}$ is the quaternionic product of certain previously appearing entries. 
By \Cref{exclude factors} \ref{exclude factors 2} and induction hypothesis, the nonzero elements among $M_{1,2},\ldots,M_{1,k-1}$ are no factors of $M_{1,k}$. 
Furthermore, the entries $M_{1,0}$ and $M_{1,1}$ naturally appear in the first row. 
Hence, in this case, we can write 
\[M_{1,k}=\Ast_{i\in I} M_{0,i}\] 
for some set $I\subset \{0,n-d,n-d+1,\ldots,n-1\}$.  

In the following, we distinguish cases regarding in which of the blocks from $B_1,\ldots,B_6$ the entry $M_{1,k-1}$ lies. 
We then show that in every case, $M_{1,k}$ either \emph{continues} this block in the claimed way or \emph{begins} a block with higher index. 
We first only consider the case $M_{0,1}=0$ although it does not make much difference for the arguments, as pointed out at the end of the proof. 

\underline{Case 1}: $M_{0,k}=0$

\underline{Case 1.1}: $k=2$ or $M_{1,k-1}$ lies in $B_1$ (i.e. $M_{0,k-1}=M_{1,k-1}=0$)

If $M_{1,k}=0$ or $M_{1,k}=\alpha$, we have nothing to show as $M_{1,k}$ continues $B_1$ or begins $B_3$. 
Hence, we can assume $M_{1,k}=\Ast_{i\in I} M_{0,i}$ with $I\not=\emptyset$. 
If $s>2$ and $I=\{0\}$, we get $M_{1,k}=1$, beginning $B_2$ with $f=1$. 
Otherwise, $I$ contains an element $j\ge n-d$, so 
\[\left(a_0,\ldots,a_{n-d-1},\prod_{i\in I} a_i\right)\] 
is independent and can be extended to a basis $\mathcal{A}'$ of $S$. 
By \ref{minimal_first_row}, its quaternionic matrix $M'$ has the same first row as $M$ and $M'_{1,2}=\ldots=M'_{1,k-1}=0$. 
By the minimality of $M$, we now get 
\[0<M_{1,k}\le M'_{1,k}=q\left(a_0, \prod_{i\in I} a_i\right) = M'_{0,n-d} = M_{0,n-d}\]
and thus $M_{1,k}=1$ if $s\le 2$ and $M_{1,k}=2$ else.
Note that in the latter case, $M_{1,k}$ begins $B_2$ with $f=2$ while the former case satisfies $f=1$ as asserted.

\underline{Case 1.2}: $k>2$ and $M_{1,k-1}$ lies in $B_2$ or $B_3$

By \Cref{exclude factors} \ref{exclude factors 4}, we have $M_{1,k}\ge M_{1,k-1}>0$ and as $M_{1,k-1}$ is a power of 2 and no factor of $M_{1,k}$, even $M_{1,k}\ge 2\cdot M_{1,k-1}$. 
If $M_{1,k}=\alpha$ we are done as in the previous case. In particular, this occurs if $M_{1,k-1}$ lies in $B_3$ since then $\alpha=2M_{1,k-1}$. 

Otherwise, since $M_{1,k}\ge 2M_{1,k-1}$, there is $n-d\le i\in I$ such that $M_{0,i}\not\in \{M_{1,k'}\mid k'<k\}$. 
Choose $\ell$ such that $M_{0,\ell}=2M_{1,k-1}$. 
Then $\ell\ge n-d$ is minimal such that $M_{0,\ell}\notin \{M_{1,k'}\mid k'<k\}$.
This means that in particular $i\ge \ell$ and 
\[\left(a_0, \ldots, a_{\ell-1}, \prod_{i\in I} a_i\right)\] 
is independent and can be extended to a basis $\mathcal{A}'$ of $S$. 
By \ref{minimal_first_row}, the matrix $M':=M_{\mathcal{A}'}(S)$ has the same first row as $M$ and since $(a_0,\ldots,a_{\ell-1})$ were not changed in the basis, we get $M_{1,k'}=M'_{1,k'}$ for all $k'<k$:
This is clear for entries in $B_1$ or $B_3$, while for $B_2$ we point out that we have $M_{1,k'}=M_{0,i}$ for some $i<\ell$ by the way we chose $\ell$.
By the minimality of $M$, we get
\[M_{1,k}\le M'_{1,k}=M'_{0,\ell}=M_{0,\ell}=2M_{1,k-1},\]
completing the case.

\underline{Case 2}: $M_{0,k}\not=0$

If $M_{1,k}=\alpha$, we begin or continue the last block $B_6$ and have nothing to show. 
Hence, we again assume $M_{1,k}=\Ast_{i\in I} M_{0,i}$. 
By \Cref{first row nqm}, the entries $M_{0,k},\ldots,M_{0,n-1}$ are all independent from the previous entries of the first row and we can apply \Cref{exclude factors} \ref{exclude factors 3} \ref{exclude factors 31}:
If $I\subset \{0,\ldots,k-1\}$, meaning that $M_{1,k}\in \spn(M_{0,j}\mid j<k)$, then we must have $M_{1,k}\in \spn(M_{1,j}\mid j<j)$.
However, by \Cref{exclude factors} \ref{exclude factors 2} and the fact that by induction hypothesis, $M_{0,k}$ is unique up to $(1,k)$, no $M_{1,j}$ for $j<k$ is a factor of $M_{1,k}$, meaning that $M_{1,k}=0$ and $I=\emptyset$.
We can conclude that we always get
\[I = \emptyset \quad \text{or} \quad I \not \subset \{0, \ldots, k - 1 \}.\] 
Considering the transformation $(k-1 \leftrightarrow k)$ in a similar way as in \Cref{exclude factors} \ref{exclude factors 4}, the case $I=\emptyset$ (beginning or continuing $B_4$) can only occur if $M_{1,k-1}$ does not lie in $B_5$ or $B_6$.

It remains to consider the case $I\not\subset \{0,\ldots,k-1\}$, meaning that 
\[\left(a_0,\ldots,a_{k-1},\prod_{i\in I} a_i\right)\] 
is independent and can be extended to a basis of $S$. 
As $M_{0,k},\ldots,M_{0,n-1}$ are unique up to $(1,k)$ by induction hypothesis, it now follows just like in \emph{Case 1.2} that $M_{1,k}=M_{0,k}$, completing \emph{Case 2}.

Lastly, if $M_{0,1}\not=0$, by \Cref{first row nqm} we have 
\[M_{0,1} = \begin{cases}1, & \text{if $s\le 2$}\\ 2, &\text{if $s>2$.} \end{cases}\] 
In either case, we get $M_{0,k}\not=0$ for all $k\ge 2$, so the blocks $B_1,B_2,B_3$ do not appear and we only have to consider case 2 above which works in the same way here. 
The only difference is that the first appearing block starts with entry $M_{0,2}=2M_{0,1}\in \{2,4\}$ as its first entry which carries forward as a factor for all other entries.

Note that the uniqueness of the parameters $d,f,r,t,u,v$ follows immediately from the uniqueness of the normal quaternionic matrix.
\end{proof}

We note the following specifications:

\begin{corollary}\label{specifications parameters}
Let $M\in \Mat_n(\N_0)$ be a normal quaternionic matrix and define $s,d,f,r,t,u,v$ as in \Cref{second row nqm}. 
The following hold:
\begin{enumerate}[label=(\alph*)]
\item\label{specifications parameters 1} If $s\le 2$ or $f=2$, we have $t\ge r$. 
Otherwise, we have $t\ge r-1$.
\item\label{specifications parameters 2} If $s\le 2$, we have $2t\ge r+u$.
\item\label{specifications parameters 3} If $s=1$ and $d\not=n-1$, we have $t\le r+v-d$.
\end{enumerate}
\end{corollary}
\begin{proof}
We keep the notation from the previous proof.

\ref{specifications parameters 1}: 
If for $k\ge 2$, the entry $0\not=M_{0,k}$ appears as $M_{1,\ell}$ for $\ell<k$, we get $M_{1,k}\in \spn(M_{1,k'}\mid k'<k)$ by \Cref{exclude factors} \ref{exclude factors 3} \ref{exclude factors 31}. 
Applying \Cref{exclude factors} \ref{exclude factors 2}, we can exclude $M_{1,\ell}=M_{0,k}$ as a factor. Hence, $M_{1,k}=0$ by \Cref{second row nqm}.

As this happens for all second row entries of $B_2$ with the exception of $1$ if $s>2$ and $f=1$, the claim follows.

\ref{specifications parameters 2}: 
Consider the transformation $(2\mapsto 1\ast 2)$: 
This does not change any of the entries at indices $(0,0)$, $(0,1)$, $(1,0)$, $(1,1)$ as $s\le 2$ and clearly also not the entries in $B_1$, $B_2$ and $B_3$ while $B_4$ and $B_5$ are transformed into
\[\begin{pmatrix} 2^0&\ldots&2^{t-1}\\ 2^0&\ldots&2^{t-1}\end{pmatrix}\quad\text{and}\quad\begin{pmatrix} 2^t&\ldots&2^{u-1}\\0&\ldots&0\end{pmatrix}.\]
Applying \Cref{exclude factors} \ref{exclude factors 2} with the second row entries of $B_2$, we can further transform $B_4$ to 
\[\begin{pmatrix} 2^0&\ldots &2^{r-1}&2^r&\ldots&2^{t-1}\\0&\ldots&0&2^r&\ldots&2^{t-1}\end{pmatrix}.\]
Permuting the $u-t$ columns in front of the last $t-r$ columns of $B_4$, we can thus attain a block of the shape
\[\begin{pmatrix} 2^0&\ldots&2^{r+u-t}\\0&\ldots&0\end{pmatrix}\]
instead of $B_4$ without changing the first row (cf. \ref{minimal_first_row}) or the first three blocks as seen above. 
By the minimality of $M$, we must have $t\ge r+u-t$.

\ref{specifications parameters 3}: 
By \Cref{first row nqm}, the dimension $\dim(V(a_0))=n-d$ is maximal. 
If $d\not=n-1$, it follows that $M_{0,0}=M_{1,0}=M_{0,1}=M_{1,1}=0$ and counting the number of zeros in the second row with \Cref{second row nqm}, we get ${\dim(V(a_1))=n-r-v+t}$. 
We thus have $n-r-v+t\le n-d$ and the claim follows.
\end{proof}

\begin{example}\label{ex_nqm of order 3}
Any normal quaternionic matrix of order up to 3 is uniquely defined by its first two rows and its type. 
There are exactly 17 isomorphism types of elementary type structures of order $3$ as can be seen in \Cref{tbl_e_n}. 
The previous result gives rise to 4 options when the type is 0 and 14 options when the type is $1$. 
One of the latter does not define a quaternionic matrix as \ref{M3} is violated for  
\[M:=\begin{pmatrix} 0&1&2\\1&1&0\\2&0&2\end{pmatrix}.\] 
This can be seen as the rows with index $3$ resp. $5$ of $\overline{M}$ are $(0\ 1\ 0\ 1\ 2\ 3\ 2\ 3)$ and $(0\ 2\ 1\ 3\ 0\ 2\ 1\ 3)$, so the entries $1$ and $2$ appear in both rows but do not appear in common columns.

The other 17 candidates from \Cref{second row nqm} thus have to be the normal quaternionic matrices of the 17 different isomorphism types of $AP_3$ structures of order 3.
For brevity, we only display the 10 nondegenerate cases:

Type 0:
\[
\LL_{1,0}^2\Delta:\begin{pmatrix}0&0&1\\0&0&2\\1&2&0\end{pmatrix}, \ 
\LL_{1,0}\Delta_2:\begin{pmatrix} 0&1&2\\1&0&4\\2&4&0\end{pmatrix}\]
Type 1:
\[
\LL_{1,1}^2\Delta:\begin{pmatrix} 0&0&1\\0&0&2\\1&2&1\end{pmatrix}, \ 
\LL_{1,1}\Delta_2: \begin{pmatrix} 0&1&2\\1&1&4\\2&4&2\end{pmatrix}, \ 
\LL_3:\begin{pmatrix} 1&0&0\\0&0&1\\0&1&0\end{pmatrix}, \ 
\]
\[
\LL_1\times\LL_{1,0}\Delta:\begin{pmatrix} 1&0&0\\0&0&2\\0&2&0\end{pmatrix}, \ 
\LL_1\times \LL_{1,1}\Delta:\begin{pmatrix} 1&0&2\\0&0&2\\2&2&2\end{pmatrix},
\]
\[
(\LL_{1,1}\times\LL_1)\Delta:\begin{pmatrix}1&0&2\\0&0&4\\2&4&2\end{pmatrix}, \ 
\LL_1^3:\begin{pmatrix}1&2&4\\2&2&0\\4&0&4\end{pmatrix}, \ 
\LL_1^2\Delta:\begin{pmatrix}1&2&4\\2&2&8\\4&8&4\end{pmatrix}.
\]
\end{example}

\begin{remark}\label{rem_nqm of elementary types}
Given $n\in \N$, the number $e(n)$ of elementary type structures of order $n$ (cf. \Cref{counting ets}) gives a lower bound on the number of $AP_3$ structures since elementary type structures satisfy $AP_3$.
In order to show that the 3ETC holds for order $n$, it thus suffices to show that $e(n)$ also is an upper bound for the number of different normal quaternionic matrices of order $n$.
Hence, we do not focus on giving a detailed and general description of normal quaternionic matrices of elementary type structures here.
However, we point out a few details without proof.
Furthermore, many examples for normal quaternionic matrices of order up to $6$ are included in \ref{sec_ListNQM}.
\begin{enumerate}[label=(\alph*)]
\item We have seen the normal quaternionic matrices of $\LL_{1,0},\LL_{1,1},\LL_1$ in \Cref{ex:qm_basic_examples}.
\item The normal quaternionic matrices of local type structures have been computed in \Cref{prop:nqm_local_types}.
\item One can obtain the normal quaternionic matrix of a group extension $S\Delta$ from a normal quaternionic matrix of $S$ by adding an independent column and reducing. 
For example consider $S=\LL_{1}\times(\LL_{1}\times \LL_{1})\Delta$ and $S\Delta$ with (cf. \Cref{ex:linkage_matrices})
\[M(S)=\begin{pmatrix} 1&2&4&8\\2&2&0&0\\4&0&4&16\\8&0&16&8\end{pmatrix}\text{ and } M(S\Delta)=\begin{pmatrix}1&2&4&8&16\\2&2&0&0&32\\4&0&4&64&128\\8&0&64&8&256\\16&32&128&256&16\end{pmatrix}.\]
\item\label{nqm_direct_product} For direct products, the situation is more involved. 
Take structures $S=(G,-1,q)$ and $S'=(G',-1,q')$ with normal quaternionic matrices $M=M_\mathcal{A}(S)$ and $M'=M_{\mathcal{A}'}(S')$ for bases $\mathcal{A}=(a_0,\ldots,a_{n-1})$ of $S$ and $\mathcal{A}'=(a_0',\ldots,a_{n'-1}')$ of $S'$. 
We conjecture that we get $M(S\times S')=M_{\mathcal{C}}(S\times S')$ for a basis $\mathcal{C}=(c_0,\ldots,c_{n+n'-1})$ that is close to a union of $\mathcal{A}$ and $\mathcal{A'}$ in the following way: Consider $G$ and $G'$ as subsets of $G\times G'$ canonically. 
Then $c_0\in \{a_0,a_0',a_0a_0'\}$ depending on the levels of $S$ and $S'$ and 
\[\{c_1,\ldots,c_{n+n'-1}\}\subset \{a_i,a_j'\mid 0\le i< n,\ 0\le j< n'\}.\]
\end{enumerate}
\end{remark}

To close this section, we point out the following example to illustrate how the situation of direct products becomes slightly more involved as mentioned in \Cref{rem_nqm of elementary types} \ref{nqm_direct_product}, especially when multiple factors have level different from one.

\begin{example}
The situation for direct products also gives rise to quaternionic matrices of elementary type structures where nonzero entries appear that are no powers of 2. 
This is due to the fact that we have to fix $-1$ as the first basis element. 
For example, we get
\[M(\LL_3)=\begin{pmatrix}1&0&0\\0&0&1\\0&1&0\end{pmatrix}\quad\text{and}\quad 
M(\LL_3^2)=\begin{pmatrix} 1&0&0&0&0&2\\0&0&0&0&2&0\\0&0&0&3&0&0\\0&0&3&0&0&0\\0&2&0&0&0&0\\2&0&0&0&0&2\end{pmatrix}.\]
To obtain this matrix, take the bases $(a_0,a_1,a_2)$ and $(a'_0,a'_1,a'_2)$ of the two copies of $\LL_3$ defining the normal quaternionic matrix of $\LL_3$ and consider the basis $(a_0a_0',a_1,a'_1,a'_2,a_2,a_0)$ of $\LL_3^2$.

For considering direct products where at least one factor has level $1$, the situation usually becomes simpler.
For example consider the structure $\LL_{1,0}\Delta\times \LL_3$ with
\[M_{\LL_{1,0}\Delta}=\begin{pmatrix} 0&1\\1&0\end{pmatrix} \quad\text{and}\quad 
M_{\LL_{1,0}\Delta\times \LL_3}=\begin{pmatrix} 1&0&0&0&0\\0&0&0&0&1\\0&0&0&2&0\\0&0&2&0&0\\0&1&0&0&0\end{pmatrix}.\]
\end{example}

\section{The Computation}\label{sec_Computation}

In this section, we present the algorithm applied to compute normal quaternionic matrices along with runtimes for some cases. 
We start with the following reduction technique.

\begin{lemma}\label{lem:reduction_technique}
    Assume that the 3ETC holds for order up to $n-1$.
    Then the 3ETC holds for order $n$ if it holds for all $AP_3$ structures $S=(G,-1,q)$ of order $n$ such that $S$ does \emph{not} satisfy one of the following:
    \begin{itemize}
        \item $S$ is degenerate, pythagorean or rigid
        \item $s(S)=1$ and $S$ is $2$-rigid
        \item $s(S)=2$ and $S$ there is a $2$-rigid element in $D_S(\qf{1,1})$
        \item $s(S)=4$ and $-1\in G$ is $2$-rigid
    \end{itemize}
\end{lemma}
\begin{proof}
    We have to show that an $AP_3$ structure $S$ of order $n$ is of elementary type if it has one of the listed properties.
    If $S$ is pythagorean, this follows from Marshall's result \Cref{results on the etc} \ref{ETC_pythagorean}.
    In the other cases, it follows from the results in \Cref{results on the etc} \ref{ETC_iteration} that $S$ is decomposable or a group extension.
    By \Cref{prop:AP_k_direct_prod_and_group_ext}, $S$ can thus be constructed from $AP_3$ structures of order at most $n-1$ by direct products and group extensions.
    Since by assumption all $AP_3$ structures of order up to $n-1$ are of elementary type, the assertion follows.
\end{proof}

To make use of the above, we include the following results that show how to read the necessary properties off normal quaternionic matrices.

\begin{corollary}\label{specifications nqm}
Let $S$ be an $AP_3$ structure of order $n$ with normal quaternionic matrix $M$. Then $S$ is nondegenerate and neither rigid nor pythagorean if and only if every row of $\overline{M}$ has at least $4$ and at most $2^{n-1}$ zeros.

In this case, we have $M_{0,1}=M_{1,0}=M_{1,1}=0$.
\end{corollary}
\begin{proof}
The first part follows immediately from the definitions as the zeros in a row of $\overline{M}$ correspond $1-1$ to the elements of $V(a)$ for some $a$.
Considering \Cref{first row nqm} along with \ref{M1} and \ref{M2}, we get the second part of the assertion. 
\end{proof}

\begin{lemma}\label{level_nqm}
Let $S=(G,-1,q)$ be a nondegenerate $AP_3$ structure of order $n$ with normal quaternionic matrix $M$. 
We have $s(S)=1$ if and only if $M$ has only zeros on the diagonal.
Otherwise we have $s(S) = 2$ if and only if $M_{0,0}=0$.
Lastly, we have $s(S)=4$ if and only if $M_{0,0} = 1 = \overline{M}_{I, J}$ for some $I, J \subset \{0,\ldots, n - 1\}$ such that $\overline{M}_{\{0\}, I} = \overline{M}_{\{0\}, J} = 0$.
\end{lemma}
\begin{proof}
The statements about level $\le 2$ are an immediate consequence of \ref{M1} and the fact that 
\[-1\in D(\langle 1,1\rangle)\iff q(-1,-1)=0\iff M_{0,0}=0.\]
We have $s(S)\ge 4$ if and only if $M_{0,0}=1$ as seen in \Cref{first row nqm}.
Hence, we have $s(S)=4$ if and only if $M_{0,0}=1$ and $-1\in D(\qf{1,1,1,1})$.
By a classical property which also holds in the abstract setting (cf. \cite[Proposition 2.10]{Marshall_AbstractWittRings}), we have
\[D(\qf{1,1,1,1})=\bigcup_{a\in D(\qf{1,1,1})} V(-a)=\bigcup_{a\in D(\qf{1,1,1})} \{x\mid q(-a,x)=0\},\]
\[D(\qf{1,1,1})=\bigcup_{b\in D(\qf{1,1})} V(-b)=\bigcup_{\{b\mid q(-1,b)=0\}} \{y\mid q(-b,y)=0\}.\]
This means that $-1\in D(\qf{1,1,1,1})$ if and only if there are $a,b\in G$ such that $-1\in V(-a)$, $a\in V(-b)$ and $b\in V(-1)$, i.e. 
\[q(-1,-a)=q(a,-b)=q(-1,b)=0.\]
Note that for any $a,b\in G$, we have 
\[q(-a,b)=q(-1,-a)\ast q(a,-b)\ast q(-1,b)\ast q(-1,-1)\] 
by the bimultiplicativity and symmetry of $q$.
Hence, we get $s(S)=4$ if and only if
\[q(-1,-1)=1 \text{ and } \exists a,b\in G:\ q(-1,-a)=q(a,-b)=q(-1,b)=0\]
if and only if
\[q(-1,-1)=1 \text{ and } \exists a,b \in G:\ q(-1,-a)=q(-1,b)=0,\ q(-a,b)=1.\]
In the construction of \Cref{prop:prequaternionic_matrices}, we get a basis $(a_0,\ldots,a_{n-1})$ of $G$ (with $a_0=-1$) such that the existence of $I,J$ as asserted is equivalent to the existence of $I,J$ such that $q(-1,\prod_{i\in I} a_i)=0=q(-1,\prod_{j\in J} a_j)$ and $q(\prod_{i\in I} a_i,\prod_{j\in J} a_j)=1$.
Thus, the assertion follows from setting $-a=\prod_{i\in I} a_i$ and $b=\prod_{j\in J} a_j$ above.
\end{proof}

\begin{remark}\label{excluded parameters}
The previous two results enable us to apply \Cref{lem:reduction_technique} to exclude some cases from our consideration.
We explicitly collect these here for a better overview, also referencing the associated results in \Cref{results on the etc}.
To confirm the 3ETC for order $n$, whenever we know that the 3ETC holds for order up to $n - 1$, we can exclude $AP_3$ structures whose normal quaternionic matrix $M$ has one of the following properties:
\begin{itemize}
\item There is a row of the completion $\overline{M}$ that has only zeros or fewer than $4$ zeros, cf. \ref{results on the etc} \ref{ETC_pythagorean}, \ref{ETC_iteration} \ref{ETC_iteration_Deg} and \ref{ETC_iteration_Rigid}.
\item We have level $1$ and any row of the completion has fewer than $8$ zeros, cf. \ref{results on the etc} \ref{ETC_iteration} \ref{ETC_iteration_s1_2rigid}.
\item We have $M_{0,0}=0$ (i.e. level at most $2$) and any row of the completion starting with $0$ has fewer than $8$ zeros, following from \ref{results on the etc} \ref{ETC_iteration} \ref{ETC_iteration_s2_2rigid}.
\item We have $M_{0,0}=M_{1,2}=1$ and $M_{0,3}=2$ or $n=3$. This means that $-1$ is $2$-rigid and we have level $4$, cf. \ref{results on the etc} \ref{ETC_iteration} \ref{ETC_iteration_s4_2rigid}.
This corresponds to the parameters (cf. \Cref{second row nqm}) $d=n-3$, $f=1$, $r=1$ with arbitrary $v,t,u$. 
Note that in this case, we always have $v=n-2$ by the restrictions presented in \Cref{specifications parameters}.
\end{itemize}
\end{remark}

Our computations prove our main result as follows. For computation times, we refer to \Cref{computation time}.

\begin{theorem}\label{ETC}
All $AP_3$ structures of order up to $7$ are of elementary type.
In particular, all Witt rings of fields with up to $2^7$ square classes are of elementary type.
\end{theorem}
\begin{proof}
By the results of Carson and Marshall as cited in \Cref{results on the etc} \ref{ETC5}, every quaternionic structure of order up to $5$ is of elementary type.
In particular, every $AP_3$ structure of order up to $5$ is of elementary type.
Hence, to prove the 3ETC for order $6$, it suffices to compute the correct upper bound for the number of normal quaternionic matrices of order $6$ satisfying the additional properties from \Cref{excluded parameters} as shown in \Cref{lem:reduction_technique}.
By \Cref{number of small ets} there are $81=\sum_{k>1} e_k(6)$ elementary types of order $6$ that are nondegenerate and neither rigid nor pythagorean.
One readily checks that among those, there are $10$ of level $1$, $28$ of level $2$ and $43$ of level at least $4$, see \ref{sec_ListNQM}.
    
By \Cref{lem:reduction_technique} and \Cref{excluded parameters}, we can further exclude all $2$-rigid structures of level $1$, all structures of level $2$ where an element in $D(\qf{1,1})$ is $2$-rigid and all structures of level $4$ where $-1$ is a $2$-rigid element.
Regarding level $1$, there are $7$ strictly $2$-rigid nondegenerate structures since these are exactly the direct products of two group extensions, making use of \Cref{lem:rigidity_ets} and the normalized decomposition \Cref{normalized decomposition}.
With similar counting methods, one obtains that among the 28 structures with level $2$, nine have a $2$-rigid element in $D(\qf{1,1})$.
Regarding level $4$, there are no further exclusions since the elementary type structures of level $4$ where $-1$ is $2$-rigid are exactly $\mathbb{L}_3$ of order $3$ and its iterated group extensions which are rigid by \Cref{lem:rigidity_ets} \ref{lem:rigidity_ets2}.
    
Applying this, we verified the 3ETC for order $6$ by running \Cref{alg:ClassificationAlgorithm} with the above restrictions, computing that there are at most $3$ resp. $19$ resp. $43$ such normal quaternionic matrices of order $6$ and level $1$ resp. $2$ resp. at least $4$.
    
The argument for $AP_3$ structures of order $7$ is the same, building further on the previously stated results for order $6$. 
Here, the lower bound on the number of normal quaternionic matrices given by the elementary type structures further fulfilling the restrictions coming from \Cref{results on the etc} is $4$ resp. $65$ resp. $199$ for level $1$ resp. $2$ resp. at least $4$. 
Again, these numbers are obtained as upper bounds by \Cref{alg:ClassificationAlgorithm}.

As mentioned in \Cref{rem:comparing_conjectures}, all quaternionic structures realized by fields satisfy $AP_3$.
Hence, it follows that all Witt rings of fields with up to $2^7$ square classes are of elementary type.
\end{proof}

We briefly outline some pseudo code below in \Cref{alg:ClassificationAlgorithm} to classify normal quaternionic matrices and give some background information in the following paragraphs.
The complete Python script along with a demo notebook, documentation, and exemplary output can be found in the Zenodo repository \cite{Repository}.


In the program, quaternionic matrices are created row by row and via depth search in order to minimize the data that has to be stored simultaneously.
To keep track of the current object, we save our current data in an $n \times n$ matrix $M$ whose entries are tuples.
The first entry of the tuple in $M(i, j)$ is a partial quaternionic matrix $A$ with entries up to position $(i, j)$ and the second entry of $M(i, j)$ is a list of possible extensions at index $(i, j+1)$ of $A$ that still have to be handled.
To slim down the possible extensions of $A$, we make use of \Cref{exclude factors}.
We apply functions that increase resp. decrease $(i, j)$ to the next resp. last index, taking into considerations that elements of a quaternionic matrix up to the diagonal of a given row are already determined by previous entries.
These are called \emph{Append} resp. \emph{DeleteLastEntry} in \Cref{alg:ClassificationAlgorithm}.

Quaternionic matrices are saved as a class object, among their attributes are a partial completion, the unique values up to the last computed element and their parameters. 
Whenever a transformation is found that produces a smaller reduced matrix, or a violation of the common slot property \ref{M3} or another assumption (in particular: rigidity) is detected, the current class object is discarded. 
This will be called \emph{validity check} in the following pseudo code, a function returning either \emph{true} or \emph{false} depending on whether the object under consideration is valid or not respectively. 
Whenever an object is discarded, the program continues with the last object that has not yet been discarded. 
Note that we only consider selected transformations as the exponentially growing number of possible transformations takes up too much computation time to be practical.
As explained before in the proof of \Cref{ETC}, it suffices that the resulting list contains all normal quaternionic matrices of the considered properties and the number of those matrices is exactly the number of elementary types with those properties.
We distinguish the cases of quaternionic matrices with level $1$, level $2$, and level at least $4$. 
To properly deal with this distinction, for $s_\text{range}\subseteq \{1,2,4\}$, we define 
\[\overline{s_\text{range}}:=
\begin{cases} 
s_\text{range} & \text{ if } 4\not\in s_\text{range}\\ s_\text{range}\cup\{2^k\mid k\ge 2\}\cup\{\infty\} & \text{ else.} \end{cases}\]
Given $s \in s_\text{range}$, we consider \emph{suitable parameters} as sextuples $(d,f,r,v,t,u)$ with notation and restrictions as in \Cref{second row nqm} and \Cref{specifications parameters} to determine the first two rows of the partial quaternionic matrices to compute.

\begin{algorithm}
    \caption{Classification of $AP_3$ Structures}\label{alg:ClassificationAlgorithm}
    \begin{algorithmic}
        \Require $n \in \N_{\geq 2}$, $s_\text{range} \subseteq \{1, 2, 4\}$
        \Ensure List $L$ containing all normal quaternionic matrices of order $n$ and level in $\overline{s_\text{range}}$ excluding those mentioned in \Cref{excluded parameters}. 
        \\
        \Function{Classify Quaternionic Structures}{$n, s_\text{range}$}
            \State{$L = \emptyset$}
            \ForAll{$s \in s_\text{range}$}
                \ForAll{suitable parameters $p = (d,f,r,v,t,u)$}
                    \State{$A \gets$ partial quaternionic matrix in $M_{2, n}(\N_0)$ for $p$}
                    \State{$X \gets \{x \in \N_0 \mid A \text{ can be extended with } x$\}}
                    \While{\Call{Validity Check}{$A$} = \True}
                        \If{$A \in M_{n \times n}(\N_0)$}
                            \State{$L \gets L \cup \{A\}$}
                    
                        \ElsIf{$X \neq \emptyset$}
                            \State{take $x \in X$; $A \gets \Call{Append}{A, x}$; $X \gets X \setminus \{a\}$}
                        \Else
                            \State{$A \gets \Call{DeleteLastEntry}{A}$}
                        \EndIf
                    \EndWhile
                \EndFor
            \EndFor
            \State{$L \gets \{A \in L \mid A \text{ is reduced}\}$}
            
            \State{\Return{$L$}}
        \EndFunction
    \end{algorithmic}
\end{algorithm}

\begin{remark}
    In this remark, we highlight some computational aspects.
    \begin{itemize}
        \item As mentioned before, it was not possible to check all applicable transformations during the validity check to assure the minimality of the constructed quaternion matrix a priori.
        But since we have a lower bound on the number of abstract Witt rings of a given order by \Cref{number of small ets}, we could conclude the minimality of the quaternion matrices a posteriori by showing the effectiveness of the lower bound, implying that the computed matrices are in fact normal quaternion matrices.
        \item While developing the code, we constantly extended the set of considered transformations by inspecting provisional results of earlier versions and identifying matrices in the output that are not normal quaternionic matrices.
        We further considered partial quaternionic matrices appearing as intermediate results to find additional transformations.
        By doing so, we established a list of transformations mainly consisting of products of few simple transformations.
        \item Instead of applying transformations explicitly, we aimed to use transformations implicitly in order to exclude possible extensions of a given partial matrix as soon as possible, leading to a better performance.
        The theoretical background is given in \Cref{exclude factors}, mainly in parts \ref{exclude factors 1} and \ref{exclude factors 2}.
        As another example of implicit transformations, we implemented \emph{first row checks} and \emph{second row checks}.
        Since the parameters introduced in \Cref{second row nqm} are unique, we can check whether permuting the basis element corresponding to the current row to the beginning of the basis would result in another set of parameters yielding a lexicographically smaller matrix.
        \item The common slot property \ref{M3} can in general be checked only for completed rows.
        Nevertheless, we implemented some rules to exclude possible next entries or to discard matrices early.
        These work along the lines of \Cref{exclude factors} \ref{exclude factors 3}.
        \item The (partial) completion was not saved as a matrix but as a list of dictionaries, in order to check the common slot property \ref{M3} efficiently.
        We save a dictionary for each row, and the entry for the key $k$ in dictionary $i$ contains a list of those $j$, for which $\overline M_{i, j} = k$.
        The common slot check for rows $i, i'$ then boils down to showing that for each common key $k$ in the dictionaries belonging to $i, i'$, the corresponding lists have nonempty intersection.
    \end{itemize}
\end{remark}

In the following remark, we discuss the runtime and our approaches to shorten the runtime, which are based on parallelization techniques.

\begin{remark}\label{computation time}
    In order to tap the full potential of a modern computer architecture with multiple cores, we parallelized some of our calculations.
    This is not yet necessary for the case of order $6$ which was computed in about 14 seconds without applying any parallelization.
    However, dealing with the case of order $7$, the runtimes get tremendously large as can be expected due to the exponential growth both of the number of cases to be considered and the size of the completion of the quaternionic matrices.

    To have an overview consisting of comparable results, we include the runtimes for various cases in \Cref{tab_runtime}, where those for $n = 7, s \in \{2, 4\}$ are not obtained from a single run but as the sum of all partial executions, further assuming that no computations have been executed in parallel.

    Our actual implementation of \Cref{alg:ClassificationAlgorithm} was written in Python 3.12.4.
    The cases in which $n \leq 6$ or $s < 2$ as well as most of the parameters with $n = 7$ and $s = 4$ were run on a MacBook Pro with Apple M4 Pro Chip and 48 GB RAM.
    The computations for the input parameters that used the most time were computed using the HPC cluster Elysium of the Ruhr-Universität Bochum.

    \begin{table}[ht!]
        \renewcommand{\arraystretch}{1.25} 
        \centering
        \begin{tabular}{|c||c|c|c|c|}
            \hline
            \backslashbox{$s$}{$n$} & 4 & 5 & 6 & 7\\
            \hline \hline
            1 & $0\, \second$ & $0\, \second$ & $0\, \second$ & $12\, \minute\,  36\, \second$\\
            \hline
            2 & $0\, \second$ & $0\, \second$ & $2\, \second$ & $3$ days $15\, \hour$ \\
            \hline
            4 & $0\, \second$ & $0\, \second$ & $12\, \second$ & $204\, \text{days} $ \\
            \hline
        \end{tabular}
        \caption{runtimes}
        \label{tab_runtime}
    \end{table}

    We now point out the parallelization strategies used to deal with the case $n = 7$. 
    In a first step, for each choice of $s \in \{1, 2, 4\}$, we applied the algorithm parallelized for the different choices of parameters $s, d, f, r, t, u, v$, i.e. the different choices of the first two rows of the normal quaternionic matrices, as presented in \Cref{second row nqm}.
    For level $s = 1$, there were $41$ parameters to consider, the most expensive one of them having a computation time of roughly $7$ minutes.

    For level $s = 2$ resp. $s \ge 4$, the number of parameters to consider is $72$ resp. $196$. 
    Among these, there were $3$ resp. $12$ with a computation time of at least $5\, \hour$ whose isolated runtimes (again assumed to have run on a single core) can be found in \Cref{table_runtime_difficult_parameters}.
    This illustrates how most of the computation time focuses on a very small number of parameters.

    \begin{table}[ht!]
\begin{centering}
    \begin{tabular}{|l|r|}
        \hline
        parameter & runtime\\
        \hline \hline
        $s = 2$:&\\
        $(4, 1, 1, 4, 1, 1)$ & $2$ days, $11\hour$\\
        $(4, 1, 0, 4, 1, 1)$ & $15\hour$\\
        $(4, 1, 0, 4, 2, 2)$ & $11\hour$\\
        remaining $69$ parameters & $2\hour$\\
        \hline \hline 
        $s \ge 4$:&\\
        $(4, 1, 0, 4, 1, 1)$ & $85$ days, $5\hour$\\  
        $(4, 2, 1, 4, 1, 1)$ & $71$ days, $16\hour$\\
        $(4, 1, 0, 4, 2, 2)$ & $40$ days, $7\hour$\\
        $(4, 1, 0, 4, 2, 3)$ & $1$ day, $15\hour$\\
        $(3, 1, 0, 3, 0, 0)$ & $1$ day, $9\hour$\\
        $(3, 1, 0, 3, 1, 1)$ & $21\hour$\\
        $(4, 1, 0, 4, 3, 3)$ & $10\hour$\\
        $(3, 1, 0, 3, 0, 1)$ & $9\hour$\\
        $(4, 1, 0, 4, 0, 1)$ & $9\hour$\\
        $(4, 1, 0, 4, 0, 0)$ & $8\hour$\\
        $(2, 1, 0, 3, 0, 0)$ & $7\hour$\\
        $(4, 1, 0, 4, 1, 2)$ & $5\hour$\\
        remaining $184$ parameters & $19\hour$\\
        \hline
    \end{tabular} 
    \caption{parameters with runtimes of at least $5 \hour$ for $n = 7$ and $s = 2$, $s\ge 4$}
    \label{table_runtime_difficult_parameters}
\end{centering}
    \end{table}

    For these parameters, we applied further parallelization techniques which are explained in detail in \Cref{Ex_parallelization_parameter}. 
\end{remark}

\begin{example}\label{Ex_parallelization_parameter}
    We use the case $n = 7$, $s \ge 4$ with parameters $(4, 1, 0, 4, 1, 1)$ to illustrate our parallelization strategies for those parameters, for which the computation of the associated normal quaternionic matrices needed the most time.
    This strategy was applied for all parameters with runtimes of more than $24$ hours and - during the process of adapting the algorithm - also for some of the other parameters and for which we also established unparallelized results. 
    A detailed overview of this process can also be found in the repository \cite{Repository}.
    
    In the implementation of \Cref{alg:ClassificationAlgorithm}, we added the option to have a list of partial quaternionic matrices (with all their class attributes) as an input from which we can resume the computation to complete the given partial matrices.
    Further, built up from auxiliary functions used for debugging, we implemented the option to only compute all quaternionic matrices for a given input (either a set of parameters of a list of incomplete quaternionic matrices as just explained) just up to a certain entry $(i, j)$ specified in the code.
    The advantage of this is that it is possible to partition the case into many sub cases which can be computed on different kernels as they have no interdependence.
    In practice, it was possible for us to compute up to approx. 60000 cases in one run, resulting in .obj-files of about 3 GB, serving as the input for a future run.
    
    For the example under consideration, we get the first two rows of the normal quaternionic matrix as follows:
    \[\begin{pmatrix}
        1 & 0 & 0 & 2 & 4 & 8 & 16 \\
        0 & 0 & 0 & 0 & 32 & 64 & 128
    \end{pmatrix}\]
    In a first step, $16$ partial quaternionic matrices up to entry $(2,3)$ which is the first entry not determined by the symmetry and the diagonal conditions were calculated.
    While $11$ of those cases could be computed entirely in a few seconds, the other $5$ cases, namely
    \[\begin{pmatrix}
        1 & 0 & 0 & 2 & 4 & 8 & 16 \\
        0 & 0 & 0 & 0 & 32 & 64 & 128 \\
        0 & 0 & 0 & x &&&
    \end{pmatrix}\]
    for $x\in \{0, 2, 33, 40, 256\}$, were partially completed up to a point to reach a number of cases that is high enough to ensure efficient parallelization but without exceeding sensible limits for storage.
    
    Once again, we only illustrate the most complicated case which proved to be $x = 256$ above; some discussions on which (sub) cases are the most complex and how it might be possible to reduce computation times with more theoretic input can be found in the following \Cref{rem_computation_difficulties}.
    
    The algorithm computed $41172$ cases up to entry $(3, 4)$.
    All but $10$ of those were computed without any further intermediate steps, most of them with case times of under a minute, but some with case times of a few hours as well.
    The remaining $10$ cases were again only partially computed up to entry $(3, 5)$ or $(4, 4)$, each dividing into $300$ to $60000$ subcases which could all be computed in a few hours at most.
    
    For each case, a result file was produced, including the runtime of the case, the partial quaternionic matric used as the input, and the number of results of quaternionic matrices.
    In the given example, a simple routine checked all the result files to add the number of quaternionic matrices computed and also the different case times.
    As conjectured in the 3ETC, no quaternionic matrices were computed while the computation times for the parameter considered added up to $85$ days and $5$ hours, see \Cref{table_runtime_difficult_parameters}.
\end{example}

\begin{remark}\label{rem_computation_difficulties}
    While executing the computations for $AP_3$ structures of order $n = 7$, we reached the limits of what is computable in reasonable time.
    We identified the following problems:
    \begin{itemize}
        \item 
            In most of the cases, high powers of 2 led to big additional effort.
            This comes from the obvious fact that given a partial quaternionic matrix in which $2^0, \ldots 2^k$ occur, the next entry has to be chosen in $\{0, \ldots, 2^{k+1}\}$.
            Hence, introducing the next power of 2 results in doubling the possibilities for all upcoming entries of the matrix.
            It would be thus desirable to have stronger conditions on possible next entries, especially with respect to excluding 2-powers in the sense of \Cref{exclude factors}.
        \item 
            So far, we do not have criteria to identify nonreduced matrices while computing them.
            Instead, we excluded them when already having complete matrices, so after the computational effort is already done.
            E.g. for $n = 7, s = 4$, we computed the following matrices, both of them determining the same quaternionic structure $\mathbb{L}_{1,0}\Delta\times (\mathbb{L}_1\times \mathbb{L}_{1,1}\Delta)\Delta_2$, but only the second one is reduced (also cf. \Cref{ex:linkage_matrices}):
            \[\begin{pmatrix} 
                1&0&0&0&2&4&8\\ 
                0&0&0&0&2&16&32\\
                0&0&0&64&0&0&0\\
                0&0&64&0&0&0&0\\
                2&2&0&0&2&33&128\\
                4&16&0&0&33&4&256\\
                8&32&0&0&128&256&8
                \end{pmatrix}, \  
            \begin{pmatrix}
                1&0&0&0&2&4&8\\
                0&0&0&0&2&16&32\\
                0&0&0&64&0&0&0\\
                0&0&64&0&0&0&0\\
                2&2&0&0&2&128&256\\
                4&16&0&0&128&4&512\\
                8&32&0&0&256&512&8
            \end{pmatrix}.\]
            Even though the computation time for the nonreduced matrix is less than the time for the reduced matrix (mostly because it consists of smaller entries), it costs a lot of effort which would be desirable to exclude beforehand using a theoretical argument. 
            These 'copies' of isomorphic structures also seem to appear in intermediate results in a great number, even if they cannot be extended to a normal quaternionic matrix. 
            This appears to be a major reason for the extreme computation times in certain cases.
        \item 
            From a purely computational point of view, it does not seem to be obvious a priori which parameters (or subcases thereof) need the most runtime.
            It appears that the parameters for which the computation of associated normal quaternionic matrices take the most time usually have large $d$ and small values for $v-d$ but some parameters that satisfy these conditions could be handled with surprising ease.
            Additionally, while the 'critical' parameters usually satisfy $r=0$, there was one exception for order $7$ and level $4$, namely $(4, 2, 1, 4, 1, 1)$, which proved to be the second most time consuming parameter.
            Criteria to identify these critical parameters and their most time consuming subcases would enable us to apply automated strategies for parallelization of calculations, instead of relying on manual inspections for each case.
        \item 
            Aside from the limitations due to the common slot property, the only condition that enables us to exclude new powers of $2$ as the next entry comes from excluding rigid and in some cases $2$-rigid elements.
            Some test runs indicate that excluding $2$-rigid elements entirely would massively reduce the computation time for $n = 7$, $s \ge 4$ to about $1$ day.
            Hence, it would be of great help to generalize the results obtained by Carson and Marshall resp. Cordes on $2$-rigid elements, which are mentioned in \Cref{results on the etc}.
    \end{itemize}
\end{remark}

We conclude this article with a comparison to the most recent classification result.

\begin{remark}
As mentioned before, Carson and Marshall first computationally verified the ETC for order $5$ in their 1982 article \cite{CarsonMarshall_Decomposition}.
To compute all abstract Witt rings for a given square class group $G$, they consider the value groups $V(a)$ for $a\in G$.
There are some similarities to the approach presented here, as these value groups can be recovered by considering the zeros in the rows of the completion $\overline{M}$ of a given quaternionic matrix in the following way:
if $a\in G$ is associated with row and column $k_a$ in the completion, the set $V(a)$ consists of all $b\in G$ such that $\overline{M}_{k_a, k_b}=0$ for $b$ associated with row and column $k_b$.

While we were not able to recreate the computations done by Carson and Marshall due to sparse documentation, it appears that our possibilities for early \emph{validity checks} as described before enhance efficiency. 
However, we cannot present any quantitative comparisons with the computation times achieved by Carson and Marshall.
Furthermore, as pointed out before, our computations only work under the additional assumption of $AP_3$.
\end{remark}

\section*{Acknowledgments}
Significant parts of the calculations for this publication were performed on the HPC cluster Elysium of the Ruhr University Bochum, subsidised by the DFG (INST 213/1055-1).
The authors would like to thank the anonymous referee who provided useful and detailed comments on a previous version of the manuscript.

\appendix\label{appendix}
\section{Lists of Normal Quaternionic Matrices}\label{sec_ListNQM}

We include a list of all normal quaternionic matrices of all $AP_3$ structures of order $6$ that are nondegenerate, nonpythagorean and nonrigid. 
Instead of displaying the entire matrices, we only include the upper right triangle (excluding the diagonal) as this uniquely describes the quaternionic matrix (along with the level) and indicate line breaks by a semicolon. 
For the same level, the matrices are sorted lexicographically.

In \Cref{results on the etc}, we listed several criteria for a quaternionic structure to be of elementary type, leading to additional restrictions for those normal quaternionic matrices that have to be classified, see \Cref{excluded parameters}.
As pointed out above, \Cref{alg:ClassificationAlgorithm} did not compute the quaternionic matrices not fulfilling those additional properties in order to reduce the runtime.
This implies that the following list contains normal quaternionic matrices that are not computed by \Cref{alg:ClassificationAlgorithm}; those are marked with an asterisk.

\begin{center}
\footnotesize
\emph{normal quaternionic matrices of order 6, level 1}\\
(nondegenerate, nonpythagorean, nonrigid)
\begin{tabular}{|lr|}
$\LL_{1,0}\Delta\times\LL_{1,0}^3\Delta$&(0,0,0,0,1; 0,0,0,2; 0,0,4; 8,0; 0)$^\ast$ \\
$\LL_{1,0}^2\Delta^2$&(0,0,0,0,1; 0,0,0,2; 0,4,0; 8,0; 0)$^\ast$\\
$\LL_{1,0}^2\Delta\times\LL_{1,0}\Delta_2$&(0,0,0,0,1; 0,0,0,2; 4,8,0; 16,0; 0)$^\ast$\\
$\LL_{6,0}$&(0,0,0,0,1; 0,0,1,0; 1,0,0; 0,0; 0)\\
$\LL_{1,0}\Delta\times \LL_{4,0}$&(0,0,0,0,1; 0,0,1,0; 2,0,0; 0,0; 0)\\
$\LL_{1,0}\Delta^3$&(0,0,0,0,1; 0,0,2,0; 4,0,0; 0,0; 0)\\
$\LL_{1,0}\Delta\times(\LL_{1,0}\times\LL_{1,0}\Delta)\Delta$&(0,0,0,0,1; 0,0,2,0; 4,0,8; 0,16; 0)$^\ast$\\
$\LL_{1,0}\Delta\times\LL_{1,0}^2\Delta_2$&(0,0,0,0,1; 0,2,4,0; 8,16,0; 32,0; 0)$^\ast$\\
$\LL_{1,0}\Delta\times\LL_{1,0}\Delta_3$&(0,0,0,0,1; 2,4,8,0; 16,32,0; 64,0; 0)$^\ast$\\
$\LL_{1,0}\Delta_2^2$&(0,0,0,1,2; 4,8,0,0; 16,0,0; 0,0; 32)$^\ast$
\end{tabular}
\newpage
\emph{normal quaternionic matrices of order 6, level 2}\\
(nondegenerate, nonpythagorean, nonrigid)
\begin{tabular}{|lr|}
$\LL_{1,0}\Delta\times\LL_{1,1}^3\Delta$&(0,0,0,0,1; 0,0,0,2; 0,0,4; 8,0; 0)\\
$\LL_{1,0}^2\Delta\times\LL_{1,1}^2\Delta$&(0,0,0,0,1; 0,0,0,2; 0,4,0; 8,0; 0)\\
$\LL_{1,0}\Delta_2\times\LL_{1,1}^2\Delta$&(0,0,0,0,1; 0,0,0,2; 4,8,0; 16,0; 0)\\
$\LL_{6,1}$&(0,0,0,0,1; 0,0,1,0; 1,0,0; 0,0; 0)\\
$\LL_{1,0}\Delta\times\LL_{4,1}$&(0,0,0,0,1; 0,0,1,0; 2,0,0; 0,0; 0)\\
$\LL_{1,1}\Delta\times\LL_{1,0}^3\Delta$&(0,0,0,0,1; 0,0,2,0; 0,4,0; 8,0; 0)$^\ast$\\
$\LL_{1,1}\Delta\times\LL_{4,0}$&(0,0,0,0,1; 0,0,2,0; 2,0,0; 0,0; 0)\\
$\LL_{1,0}\Delta^2\times\LL_{1,1}\Delta$&(0,0,0,0,1; 0,0,2,0; 4,0,0; 0,0; 0)\\
$\LL_{1,0}\Delta\times(\LL_{1,1}\times\LL_{1,0}\Delta)\Delta$&(0,0,0,0,1; 0,0,2,0; 4,0,8; 0,16; 0)\\
$\LL_{1,1}\Delta\times(\LL_{1,0}\times\LL_{1,0}\Delta)\Delta$&(0,0,0,0,1; 0,0,2,0; 4,8,0; 16,0; 0)$^\ast$\\
$\LL_{1,1}\Delta\times\LL_{1,0}^2\Delta_2$&(0,0,0,0,1; 0,2,4,0; 8,16,0; 32,0; 0)$^\ast$\\
$\LL_{1,1}\Delta\times\LL_{1,0}\Delta_3$&(0,0,0,0,1; 2,4,8,0; 16,32,0; 64,0; 0)$^\ast$\\
$\LL_{1,1}\Delta\times\LL_{1,1}^3\Delta$&(0,0,0,1,2; 0,0,0,2; 0,0,4; 0,8; 0)\\
$\LL_{1,1}^2\Delta^2$&(0,0,0,1,2; 0,0,0,2; 0,0,4; 8,0; 0)\\
$\LL_{1,1}\Delta\times\LL_{4,1}$&(0,0,0,1,2; 0,0,0,2; 1,0,0; 0,0; 0)\\
$\LL_{1,0}\Delta\times\LL_{1,1}\Delta^2$&(0,0,0,1,2; 0,0,0,2; 4,0,0; 0,0; 0)\\
$\LL_{1,1}\Delta\times(\LL_{1,1}\times \LL_{1,0}\Delta)\Delta$&(0,0,0,1,2; 0,0,0,2; 4,0,8; 0,16; 0)\\
$\LL_{1,0}\Delta\times(\LL_{1,1}\times\LL_{1,1}\Delta)\Delta$&(0,0,0,1,2; 0,0,0,4; 8,0,0; 0,0; 16)\\
$\LL_{1,0}\Delta\times\LL_{1,1}^2\Delta_2$&(0,0,0,1,2; 0,0,4,8; 16,0,0; 0,0; 32)\\
$\LL_{1,0}^2\Delta\times\LL_{1,1}\Delta_2$&(0,0,0,1,2; 0,4,0,0; 8,0,0; 0,0; 16)$^\ast$\\
$\LL_{1,0}\Delta_2\times\LL_{1,1}\Delta_2$&(0,0,0,1,2; 4,8,0,0; 16,0,0; 0,0; 32)$^\ast$\\
$\LL_{1,1}^2\Delta\times\LL_{1,1}\Delta_2$&(0,0,1,2,4; 0,0,0,4; 0,0,8; 16,0; 0)\\
$\LL_{1,1}\Delta^3$&(0,0,1,2,4; 0,0,0,4; 0,2,0; 0,0; 0)\\
$\LL_{1,1}\Delta\times(\LL_{1,1}\times\LL_{1,1}\Delta)\Delta$&(0,0,1,2,4; 0,0,0,4; 0,8,0; 16,0; 0)\\
$\LL_{1,1}\Delta\times\LL_{1,1}^2\Delta_2$&(0,0,1,2,4; 0,0,0,4; 8,16,0; 32,0; 0)\\
$\LL_{1,0}\Delta\times\LL_{1,1}\Delta_3$&(0,0,1,2,4; 8,0,0,0; 0, 0,0; 16, 32; 64)$^\ast$\\
$\LL_{1,1}\Delta\times\LL_{1,1}\Delta_3$&(0,1,2,4,8; 0,0,0,8; 16,32,0; 64,0; 0)$^\ast$\\
$\LL_{1,1}\Delta_2^2$&(0,1,2,4,8; 0,0,4,8; 16,0,0; 0,0; 32)$^\ast$
\end{tabular}
\newpage
\emph{normal quaternionic matrices of order 6, level $>2$}\\ 
(nondegenerate, nonpythagorean, nonrigid)
\begin{tabular}{|lr|}
$\LL_3\times\LL_{1,0}^2\Delta$&(0,0,0,0,0; 0,0,0,1; 0,2,0; 4,0; 0)\\
$\LL_3\times \LL_{1,0}\Delta_2$&(0,0,0,0,0; 0,0,0,1; 2,4,0; 8,0; 0)\\
$\LL_1\times \LL_{1,0}\Delta\times \LL_{1,0}^2\Delta$&(0,0,0,0,0; 0,0,0,2; 0,0,4; 8,0; 0)\\
$\LL_1\times \LL_{1,0}\Delta\times \LL_{1,0}\Delta_2$&(0,0,0,0,0; 0,0,0,2; 4,8,0; 16,0; 0)\\
$\LL_3\times \LL_{1,1}^2\Delta$&(0,0,0,0,2; 0,0,0,2; 0,0,4; 1,0; 0)\\
$\LL_1\times \LL_{1,0}\Delta\times\LL_{1,1}^2\Delta$&(0,0,0,0,2; 0,0,0,2; 0,0,4; 8,0; 0)\\
$\LL_1\times\LL_{1,1}\Delta\times\LL_{1,0}^2\Delta$&(0,0,0,0,2; 0,0,0,2; 0,4,0; 8,0; 0)\\
$\LL_1\times \LL_{1,1}\Delta\times\LL_{1,0}\Delta_2$&(0,0,0,0,2; 0,0,0,2; 4,8,0; 16,0; 0)\\
$\LL_{1,0}\Delta\times (\LL_{1,1}^2\times\LL_1)\Delta$&(0,0,0,0,2; 0,0,0,4; 0,0,8; 16,0; 0)\\
$\LL_{1,0}^2\Delta\times (\LL_{1,1}\times\LL_1)\Delta$&(0,0,0,0,2; 0,0,0,4; 0,8,0; 16,0; 0)\\
$(\LL_{1,1}\times \LL_1)\Delta\times\LL_{1,0}\Delta^2$&(0,0,0,0,2; 0,0,0,4; 8,16,0; 32,0; 0)\\
$\LL_{1,0}\Delta\times\LL_3\Delta$&(0,0,0,0,2; 0,0,1,4; 8,0,0; 0,0; 16)\\
$\LL_1\times\LL_5$&(0,0,0,0,2; 0,0,2,0; 2,0,0; 0,0; 0)\\
$\LL_3^2$&(0,0,0,0,2; 0,0,2,0; 3,0,0; 0,0; 0)\\
$\LL_1\times\LL_3\times\LL_{1,0}\Delta$&(0,0,0,0,2; 0,0,2,0; 4,0,0; 0,0; 0)\\
$\LL_1^2\times\LL_{4,0}$&(0,0,0,0,2; 0,0,4,0; 4,0,0; 0,0; 0)\\
$\LL_1^2\times\LL_{1,0}\Delta^2$&(0,0,0,0,2; 0,0,4,0; 8,0,0; 0,0; 0)\\
$\LL_{1,0}\Delta\times(\LL_1\times\LL_{1,0}\Delta)\Delta$&(0,0,0,0,2; 0,0,4,0; 8,0,16; 0,32; 0)\\
$\LL_1\times\LL_{1,1}\Delta\times\LL_{1,1}^2\Delta$&(0,0,0,2,4; 0,0,0,4; 0,0,8; 2,0; 0)\\
$\LL_{1,1}^2\Delta\times(\LL_{1,1}\times\LL_1)\Delta$&(0,0,0,2,4; 0,0,0,4; 0,0,8; 16,0; 0)\\
$\LL_{1,1}\Delta\times(\LL_{1,1}^2\times\LL_1)\Delta$&(0,0,0,2,4; 0,0,0,4; 0,8,0; 16,0; 0)\\
$\LL_{1,1}\Delta\times\LL_3\Delta$&(0,0,0,2,4; 0,0,0,4; 1,8,0; 16,0; 0)\\
$\LL_1\times\LL_{1,1}\Delta\times\LL_3$&(0,0,0,2,4; 0,0,0,4; 2,0,0; 0,0; 0)\\
$\LL_1^2\times\LL_{4,1}$&(0,0,0,2,4; 0,0,0,4; 4,0,0; 0,0; 0)\\
$\LL_1^2\times\LL_{1,0}\Delta\times\LL_{1,1}\Delta$&(0,0,0,2,4; 0,0,0,4; 8,0,0; 0,0; 0)\\
$\LL_{1,1}\Delta\times(\LL_1\times\LL_{1,0}\Delta)\Delta$&(0,0,0,2,4; 0,0,0,4; 8,16,0; 32,0; 0)\\
$\LL_3\times(\LL_{1,1}\times\LL_1)\Delta$&(0,0,0,2,4; 0,0,0,8; 2,0,0; 0,0; 0)\\
$\LL_1\times\LL_{1,0}\Delta\times(\LL_{1,1}\times\LL_1)\Delta$&(0,0,0,2,4; 0,0,0,8; 16,0,0; 0,0; 0)\\
$\LL_{1,0}\Delta\times(\LL_{1,1}\times\LL_1^2)\Delta$&(0,0,0,2,4; 0,0,0,8; 16,0,0; 0,0; 32)\\
$\LL_3\times\LL_{1,1}\Delta_2$&(0,0,0,2,4; 0,0,2,4; 1,0,0; 0,0; 8)\\
$\LL_1\times\LL_{1,0}\Delta\times\LL_{1,1}\Delta_2$&(0,0,0,2,4; 0,0,2,4; 8,0,0; 0,0; 16)\\
$\LL_{1,0}\Delta\times(\LL_1\times\LL_{1,1}\Delta)\Delta$&(0,0,0,2,4; 0,0,2,8; 16,0,0; 0,0; 32)\\
$\LL_{1,0}\Delta\times(\LL_{1,1}\times\LL_1)\Delta_2$&(0,0,0,2,4; 0,0,8,16; 32,0,0; 0,0; 64)\\
$\LL_1^2\times\LL_{1,1}\Delta^2$&(0,0,2,4,8; 0,0,0,8; 0,4,0; 0,0; 0)\\
$\LL_1\times\LL_{1,1}\Delta\times(\LL_{1,1}\times\LL_1)\Delta$&(0,0,2,4,8; 0,0,0,8; 0,16,0; 0,0; 0)\\
$\LL_{1,1}\Delta\times(\LL_{1,1}\times\LL_1^2)\Delta$&(0,0,2,4,8; 0,0,0,8; 0,16,0; 32,0; 0)\\
$\LL_1\times\LL_{1,1}\Delta\times\LL_{1,1}\Delta_2$&(0,0,2,4,8; 0,0,0,8; 2,4,0; 16,0; 0)\\
$\LL_{1,1}\Delta\times (\LL_1\times\LL_{1,1}\Delta)\Delta$&(0,0,2,4,8; 0,0,0,8; 2,16,0; 32,0; 0)\\
$\LL_{1,1}\Delta\times(\LL_{1,1}\times\LL_1)\Delta_2$&(0,0,2,4,8; 0,0,0,8; 16,32,0; 64,0; 0)\\
$(\LL_{1,1}\times\LL_1)\Delta^2$&(0,0,2,4,8; 0,0,0,16; 0,32,0; 0,8; 0)\\
$(\LL_{1,1}\times\LL_1)\Delta\times\LL_{1,1}\Delta_2$&(0,0,2,4,8; 0,0,0,16; 2,4,0; 32,0; 0)\\
$\LL_1^3\times\LL_3$&(0,0,2,4,8; 2,0,0,0; 0,0,0; 0,0; 0)\\
$\LL_3\times\LL_1^2\Delta$&(0,0,2,4,8; 2,0,0,0; 0,0,0; 0,0; 16)
\end{tabular}
\end{center}
\normalsize

\bibliographystyle{elsarticle-harv} 
\bibliography{references}

@article {ArasonPfister_Krull,
    AUTHOR = {Arason, J\'on Kristinn and Pfister, Albrecht},
     TITLE = {Beweis des {K}rullschen {D}urchschnittsatzes f\"ur den
              {W}ittring},
   JOURNAL = {Invent. Math.},
  FJOURNAL = {Inventiones Mathematicae},
    VOLUME = {12},
      YEAR = {1971},
     PAGES = {173--176},
      ISSN = {0020-9910,1432-1297},
   MRCLASS = {10C05},
  MRNUMBER = {294251},
MRREVIEWER = {H.\ Gross},
       DOI = {10.1007/BF01404657},
       URL = {https://doi.org/10.1007/BF01404657},
}

@inproceedings {Becher_NoSquareClasses,
    AUTHOR = {Becher, Karim Johannes},
     TITLE = {On the number of square classes of a field of finite level},
 BOOKTITLE = {Proceedings of the {C}onference on {Q}uadratic {F}orms and
              {R}elated {T}opics ({B}aton {R}ouge, {LA}, 2001)},
   JOURNAL = {Doc. Math.},
  FJOURNAL = {Documenta Mathematica},
      YEAR = {2001},
     PAGES = {65--84},
      ISSN = {1431-0635,1431-0643},
   MRCLASS = {11E04 (11E81 16K50)},
  MRNUMBER = {1869389},
MRREVIEWER = {Carl\ Riehm},
}

@article {Bos_StructureWittRing,
    AUTHOR = {Bos, Rik},
     TITLE = {A structure theorem for abstract {W}itt rings containing rigid
              elements},
   JOURNAL = {Nederl. Akad. Wetensch. Indag. Math.},
  FJOURNAL = {Koninklijke Nederlandse Akademie van Wetenschappen.
              Indagationes Mathematicae},
    VOLUME = {51},
      YEAR = {1989},
    NUMBER = {2},
     PAGES = {125--140},
      ISSN = {0019-3577},
   MRCLASS = {11E81 (19G12)},
  MRNUMBER = {1005046},
MRREVIEWER = {Craig\ Cordes},
}

@article {CarsonMarshall_Decomposition,
    AUTHOR = {Carson, Andrew B. and Marshall, Murray A.},
     TITLE = {Decomposition of {W}itt rings},
   JOURNAL = {Canadian J. Math.},
  FJOURNAL = {Canadian Journal of Mathematics. Journal Canadien de
              Math\'ematiques},
    VOLUME = {34},
      YEAR = {1982},
    NUMBER = {6},
     PAGES = {1276--1302},
      ISSN = {0008-414X,1496-4279},
   MRCLASS = {10C05},
  MRNUMBER = {678670},
MRREVIEWER = {Thomas\ C.\ Craven},
       DOI = {10.4153/CJM-1982-089-1},
       URL = {https://doi.org/10.4153/CJM-1982-089-1},
}

@article {Cordes_WittGroupEquivalenve,
    AUTHOR = {Cordes, Craig M.},
     TITLE = {The {W}itt group and the equivalence of fields with respect to
              quadratic forms},
   JOURNAL = {J. Algebra},
  FJOURNAL = {Journal of Algebra},
    VOLUME = {26},
      YEAR = {1973},
     PAGES = {400--421},
      ISSN = {0021-8693},
   MRCLASS = {15A63 (13K05)},
  MRNUMBER = {325656},
MRREVIEWER = {D.\ G.\ James},
       DOI = {10.1016/0021-8693(73)90002-1},
       URL = {https://doi.org/10.1016/0021-8693(73)90002-1},
}

@article {Cordes_FiniteNoQuaternions,
    AUTHOR = {Cordes, Craig M.},
     TITLE = {Quadratic forms over nonformally real fields with a finite
              number of quaternion algebras},
   JOURNAL = {Pacific J. Math.},
  FJOURNAL = {Pacific Journal of Mathematics},
    VOLUME = {63},
      YEAR = {1976},
    NUMBER = {2},
     PAGES = {357--365},
      ISSN = {0030-8730,1945-5844},
   MRCLASS = {10C05},
  MRNUMBER = {417052},
MRREVIEWER = {K.\ Szymiczek},
       URL = {http://projecteuclid.org/euclid.pjm/1102867391},
}

@article {Kula_FieldsQFSchemes,
    AUTHOR = {Kula, Mieczys{\l}aw},
     TITLE = {Fields and quadratic form schemes},
   JOURNAL = {Ann. Math. Sil.},
  FJOURNAL = {Annales Mathematicae Silesianae},
    NUMBER = {13},
      YEAR = {1985},
     PAGES = {7--22},
      ISSN = {0860-2107,2391-4238},
   MRCLASS = {11E81},
  MRNUMBER = {813141},
MRREVIEWER = {Craig\ Cordes},
}

@book {Kula_FinitelyGeneratedWittRings,
    AUTHOR = {Kula, Mieczys{\l}aw},
     TITLE = {Finitely generated {W}itt rings},
    SERIES = {Prace Naukowe Uniwersytetu \'Slpolhk askiego w Katowicach
              [Scientific Publications of the University of Silesia]},
    VOLUME = {1207},
      NOTE = {With Polish and Russian summaries},
 PUBLISHER = {Uniwersytet \'Slpolhk aski, Katowice},
      YEAR = {1991},
     PAGES = {52},
      ISBN = {83-226-0373-8},
   MRCLASS = {11E81 (11E04 11E08)},
  MRNUMBER = {1151834},
MRREVIEWER = {Murray\ Marshall},
}

@article {KnebuschRosenbergWare_WittRings,
    AUTHOR = {Knebusch, Manfred and Rosenberg, Alex and Ware, Roger},
     TITLE = {Structure of {W}itt rings and quotients of {A}belian group
              rings},
   JOURNAL = {Amer. J. Math.},
  FJOURNAL = {American Journal of Mathematics},
    VOLUME = {94},
      YEAR = {1972},
     PAGES = {119--155},
      ISSN = {0002-9327,1080-6377},
   MRCLASS = {16A48 (10C05)},
  MRNUMBER = {296103},
MRREVIEWER = {J.\ S.\ Hsia},
       DOI = {10.2307/2373597},
       URL = {https://doi.org/10.2307/2373597},
}

@article {KulaSzczepanikSzymiczek_8SquareClasses,
    AUTHOR = {Kula, Mieczys{\l}aw and Szczepanik, Lucyna and Szymiczek,
              Kazimierz},
     TITLE = {Quadratic forms over formally real fields with eight square
              classes},
   JOURNAL = {Manuscripta Math.},
  FJOURNAL = {Manuscripta Mathematica},
    VOLUME = {29},
      YEAR = {1979},
    NUMBER = {2-4},
     PAGES = {295--303},
      ISSN = {0025-2611,1432-1785},
   MRCLASS = {10C04},
  MRNUMBER = {545046},
MRREVIEWER = {Craig\ Cordes},
       DOI = {10.1007/BF01303632},
       URL = {https://doi.org/10.1007/BF01303632},
}

@article {KSS_Equivalences,
    AUTHOR = {Kula, M. and Szczepanik, L. and Szymiczek, K.},
     TITLE = {Quadratic form schemes and quaternionic schemes},
   JOURNAL = {Fund. Math.},
  FJOURNAL = {Polska Akademia Nauk. Fundamenta Mathematicae},
    VOLUME = {130},
      YEAR = {1988},
    NUMBER = {3},
     PAGES = {181--190},
      ISSN = {0016-2736,1730-6329},
   MRCLASS = {11E81},
  MRNUMBER = {970903},
MRREVIEWER = {Bessie\ H.\ Kirkwood},
       DOI = {10.4064/fm-130-3-181-190},
       URL = {https://doi.org/10.4064/fm-130-3-181-190},
}

@book {Lam_IntroQF,
    AUTHOR = {Lam, T. Y.},
     TITLE = {Introduction to quadratic forms over fields},
    SERIES = {Graduate Studies in Mathematics},
    VOLUME = {67},
 PUBLISHER = {American Mathematical Society, Providence, RI},
      YEAR = {2005},
     PAGES = {xxii+550},
      ISBN = {0-8218-1095-2},
   MRCLASS = {11Exx},
  MRNUMBER = {2104929},
MRREVIEWER = {K.\ Szymiczek},
       DOI = {10.1090/gsm/067},
       URL = {https://doi.org/10.1090/gsm/067},
}

@book {Marshall_AbstractWittRings,
    AUTHOR = {Marshall, Murray},
     TITLE = {Abstract {W}itt rings},
    SERIES = {Queen's Papers in Pure and Applied Mathematics},
    VOLUME = {57},
 PUBLISHER = {Queen's University, Kingston, ON},
      YEAR = {1980},
     PAGES = {vi+257},
   MRCLASS = {10C05},
  MRNUMBER = {674651},
MRREVIEWER = {Thomas\ C.\ Craven},
}

@incollection {Marshall_ETC,
    AUTHOR = {Marshall, M.},
     TITLE = {The elementary type conjecture in quadratic form theory},
 BOOKTITLE = {Algebraic and arithmetic theory of quadratic forms},
    SERIES = {Contemp. Math.},
    VOLUME = {344},
     PAGES = {275--293},
 PUBLISHER = {Amer. Math. Soc., Providence, RI},
      YEAR = {2004},
      ISBN = {0-8218-3441-X},
   MRCLASS = {11E81 (12F10)},
  MRNUMBER = {2060204},
MRREVIEWER = {Thomas\ C.\ Craven},
       DOI = {10.1090/conm/344/06224},
       URL = {https://doi.org/10.1090/conm/344/06224},
}

@article {Merkurjev_KTheory,
    AUTHOR = {Merkurjev, A. S.},
     TITLE = {On the norm residue symbol of degree {$2$}},
   JOURNAL = {Dokl. Akad. Nauk SSSR},
  FJOURNAL = {Doklady Akademii Nauk SSSR},
    VOLUME = {261},
      YEAR = {1981},
    NUMBER = {3},
     PAGES = {542--547},
      ISSN = {0002-3264},
   MRCLASS = {12A60 (13A20)},
  MRNUMBER = {638926},
MRREVIEWER = {Thomas\ Zink},
}

@article {DickmannMiragliaSpecialGroups,
    AUTHOR = {Dickmann, M. A. and Miraglia, F.},
     TITLE = {Special groups: {B}oolean-theoretic methods in the theory of
              quadratic forms},
      NOTE = {With appendixes A and B by Dickmann and A. Petrovich},
   JOURNAL = {Mem. Amer. Math. Soc.},
  FJOURNAL = {Memoirs of the American Mathematical Society},
    VOLUME = {145},
      YEAR = {2000},
    NUMBER = {689},
     PAGES = {xvi+247},
      ISSN = {0065-9266,1947-6221},
   MRCLASS = {11E81 (12D15 19G12)},
  MRNUMBER = {1677935},
MRREVIEWER = {Murray\ Marshall},
       DOI = {10.1090/memo/0689},
       URL = {https://doi.org/10.1090/memo/0689},
}

@misc{efrat2025elementarytypeconjecturemaximal,
      title={The Elementary Type Conjecture for Maximal Pro-p Galois groups}, 
      author={Ido Efrat},
      year={2025},
      eprint={2509.10168},
      archivePrefix={arXiv},
      primaryClass={math.NT},
      url={https://arxiv.org/abs/2509.10168}, 
}

@article{JacobWareETCGalois,
	author = {Jacob, Bill and Ware, Roger},
	date = {1991/12/01},
	doi = {10.1007/BF02571520},
	id = {Jacob1991},
	isbn = {1432-1823},
	journal = {Mathematische Zeitschrift},
	number = {1},
	pages = {193--208},
	title = {Realizing dyadic factors of elementary type Witt rings and pro-2 Galois groups},
	url = {https://doi.org/10.1007/BF02571520},
	volume = {208},
	year = {1991}}

@article {MarshallYucasAP3,
    AUTHOR = {Marshall, Murray and Yucas, Joseph},
     TITLE = {Linked quaternionic mappings and their associated {W}itt
              rings},
   JOURNAL = {Pacific J. Math.},
  FJOURNAL = {Pacific Journal of Mathematics},
    VOLUME = {95},
      YEAR = {1981},
    NUMBER = {2},
     PAGES = {411--425},
      ISSN = {0030-8730,1945-5844},
   MRCLASS = {10C05 (10C01)},
  MRNUMBER = {632195},
MRREVIEWER = {Ron\ Brown},
       URL = {http://projecteuclid.org/euclid.pjm/1102735078},
}

@article {Szczepanik_RadicalNotExceeding16,
    AUTHOR = {Szczepanik, Lucyna},
     TITLE = {Fields and quadratic form schemes with the index of radical
              not exceeding {$16$}},
   JOURNAL = {Ann. Math. Sil.},
  FJOURNAL = {Annales Mathematicae Silesianae},
    NUMBER = {13},
      YEAR = {1985},
     PAGES = {23--46},
      ISSN = {0860-2107,2391-4238},
   MRCLASS = {11E81},
  MRNUMBER = {813142},
MRREVIEWER = {Craig\ Cordes},
}

@article {Szczepanik_QFSnontrivialrad,
    AUTHOR = {Szczepanik, L.},
     TITLE = {Quadratic form schemes with nontrivial radical},
   JOURNAL = {Colloq. Math.},
  FJOURNAL = {Colloquium Mathematicum},
    VOLUME = {49},
      YEAR = {1985},
    NUMBER = {2},
     PAGES = {143--160},
      ISSN = {0010-1354,1730-6302},
   MRCLASS = {11E81},
  MRNUMBER = {830799},
MRREVIEWER = {Daniel\ B.\ Shapiro},
       DOI = {10.4064/cm-49-2-143-160},
       URL = {https://doi.org/10.4064/cm-49-2-143-160},
}

@article {Cordes_BinaryRepr4,
    AUTHOR = {Cordes, Craig M.},
     TITLE = {Abstract {W}itt rings when certain binary forms represent
              exactly four elements},
   JOURNAL = {Canad. J. Math.},
  FJOURNAL = {Canadian Journal of Mathematics. Journal Canadien de
              Math\'ematiques},
    VOLUME = {45},
      YEAR = {1993},
    NUMBER = {6},
     PAGES = {1184--1199},
      ISSN = {0008-414X,1496-4279},
   MRCLASS = {11E81},
  MRNUMBER = {1247541},
MRREVIEWER = {Murray\ Marshall},
       DOI = {10.4153/CJM-1993-066-7},
       URL = {https://doi.org/10.4153/CJM-1993-066-7},
}

@misc{Repository ,
  author  = {Lorenz, Nico and Sch\"onert, Alexander},
  title   = {Normal Quaternionic Matrices},
  year    = {2026},
  howpublished = {Zenodo},
  doi     = {10.5281/zenodo.22873583},
}

\end{document}